\documentclass[11pt]{preprint}
\usepackage[english]{babel}
\usepackage{amssymb}
\usepackage{amsmath}
\usepackage{hyperref}
\usepackage{breakurl}
\usepackage{pifont}
\usepackage{mhenvs}
\usepackage{mhequ}
\usepackage{mhsymb}
\usepackage{booktabs}
\usepackage{stmaryrd}
\usepackage{float}
\usepackage{tikz}
\usepackage{mathrsfs}
\usepackage{array}
\usepackage{times}
\usepackage{microtype}
\usepackage{mathtools}
\usepackage{comment}
\usepackage{slashed}
\usepackage{upgreek}
\usepackage{cases}
\usepackage{mathtools}
\usepackage{centernot}
\usepackage{leftidx}
\usepackage{accents}
\usepackage{arydshln}
\usepackage{footnote}
\makesavenoteenv{tabular}
\usepackage{enumitem}

\usepackage{stackrel}
\usepackage{halloweenmath}
\DeclareMathAlphabet{\mathBM}{T1}{cmr}{bx}{it}
\usepackage{longtable}
\usepackage{array}
\usepackage{cprotect}
\usepackage{xstring}
\usepackage{ulem}
\usepackage{stmaryrd}
\usepackage{empheq}
\usepackage{color}
\usepackage{framed}

\definecolor{shadecolor}{gray}{0.875}

\DeclareSymbolFont{timesoperators}{T1}{ptm}{m}{n}
\SetSymbolFont{timesoperators}{bold}{T1}{ptm}{b}{n}

\makeatletter
\renewcommand{\operator@font}{\mathgroup\symtimesoperators}
\makeatother

\colorlet{symbolsgrey}{blue!30!black!50}
\colorlet{testcolor}{green!60!black}
\definecolor{purple}{rgb}{0.55,0.05,0.8}
\definecolor{symbols}{rgb}{0.55,0.05,0.8}

\usetikzlibrary{shapes.misc}
\usetikzlibrary{shapes.symbols}
\usetikzlibrary{shapes.geometric}
\usetikzlibrary{snakes}
\usetikzlibrary{decorations}
\usetikzlibrary{decorations.markings}

\usetikzlibrary{calc}
\usetikzlibrary{external}

\newtheorem{assumption}[lemma]{Assumption}

\let\oldskull\skull
\def\skull{\mathord{\oldskull}}

\DeclareMathAlphabet{\mathbbm}{U}{bbm}{m}{n}

\DeclareFontFamily{U}{BOONDOX-calo}{\skewchar\font=45 }
\DeclareFontShape{U}{BOONDOX-calo}{m}{n}{
  <-> s*[1.05] BOONDOX-r-calo}{}
\DeclareFontShape{U}{BOONDOX-calo}{b}{n}{
  <-> s*[1.05] BOONDOX-b-calo}{}
\DeclareMathAlphabet{\mcb}{U}{BOONDOX-calo}{m}{n}
\SetMathAlphabet{\mcb}{bold}{U}{BOONDOX-calo}{b}{n}

\setlist{noitemsep,topsep=4pt}
\newcommand{\nnorm}[1]{{\vert\kern-0.25ex\vert\kern-0.25ex\vert #1 
    \vert\kern-0.25ex\vert\kern-0.25ex\vert}}
\newcommand{\norm}[1]{{\| #1 
    \|}}

\makeatletter 
\newcommand*{\bigcdot}{}
\DeclareRobustCommand*{\bigcdot}{%
  \mathbin{\mathpalette\bigcdot@{}}%
}
\newcommand*{\bigcdot@scalefactor}{.5}
\newcommand*{\bigcdot@widthfactor}{1.15}
\newcommand*{\bigcdot@}[2]{%
  \sbox0{$#1\vcenter{}$}
  \sbox2{$#1\cdot\m@th$}%
  \hbox to \bigcdot@widthfactor\wd2{%
    \hfil
    \raise\ht0\hbox{%
      \scalebox{\bigcdot@scalefactor}{%
        \lower\ht0\hbox{$#1\bullet\m@th$}%
      }%
    }%
    \hfil
  }%
}
\makeatother

\def\t{\mathbf{t}}

\def\dash{\leavevmode\unskip\kern0.18em--\penalty\exhyphenpenalty\kern0.18em}
\def\slash{\leavevmode\unskip\kern0.15em/\penalty\exhyphenpenalty\kern0.15em}

\newcommand{\Norm}[1]{{\talloblong #1 \talloblong}}

\colorlet{darkblue}{blue!90!black}
\colorlet{darkgreen}{green!82!black}
\colorlet{darkyellow}{yellow!65!red}
\colorlet{darkred}{red!80!black}
\let\comm\comment

\newcommand{\treeOne}{\raisebox{-1ex}{\includegraphics[scale=1]{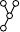}}}
\newcommand{\treeTwo}{\raisebox{-1ex}{\includegraphics[scale=1]{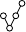}}}

\newcommand{\mfI}{\mathfrak{I}}

\newcommand{\mfn}{\mathfrak{n}}
\newcommand{\mfo}{\mathfrak{o}}
\newcommand{\mfe}{\mathfrak{e}}

\newcommand{\mfd}{\mathfrak{d}}

\newcommand{\mfc}{\mathfrak{c}}

\newcommand{\mfm}{\mathfrak{m}}
\newcommand{\mfM}{\mathfrak{M}}
\newcommand{\mfS}{\mathfrak{S}}

\newcommand{\mfl}{\mathfrak{l}}

\newcommand{\mfh}{\mathfrak{h}}

\newcommand{\mfb}{\mathfrak{b}}
\newcommand{\mff}{\mathfrak{f}}

\newcommand{\mfg}{\mathfrak{g}}
\newcommand{\mfk}{\mathfrak{k}}

\newcommand{\mcM}{\mathcal{M}}
\newcommand{\mcE}{\mathcal{E}}
\newcommand{\mcH}{\mathcal{H}}

\newcommand{\mcR}{\mathcal{R}}
\newcommand{\mcC}{\mathcal{C}}

\newcommand{\mcS}{\mathcal{S}}

\newcommand{\mcL}{\mathcal{L}}

\newcommand{\mcN}{\mathcal{N}}

\newcommand{\mcD}{\mathcal{D}}
\newcommand{\mcW}{\mathcal{W}}
\newcommand{\mcP}{\mathcal{P}}
\newcommand{\mcG}{\mathcal{G}}

\newcommand{\mcQ}{\mathcal{Q}}

\newcommand{\bfd}{\mathbf{d}}
\newcommand{\bfb}{\mathbf{b}}
\newcommand{\bfe}{\mathbf{e}}
\newcommand{\bff}{\mathbf{f}}
\newcommand{\bfg}{\mathbf{g}}
\newcommand{\bfh}{\mathbf{h}}
\newcommand{\bfc}{\mathbf{c}}

\newcommand{\bfL}{\mathbf{L}}
\newcommand{\bfI}{\mathbf{I}}

\newcommand{\tH}{\Tilde{H}}

\newcommand{\mfs}{\mathfrak{s}}

\newcommand{\D}{\mathrm{D}}
\newcommand{\Id}{\mathrm{Id}}
\newcommand{\vphi}{\varphi}
\newcommand{\bma}{{\mathBM a}}
\newcommand{\bmb}{{\mathBM b}}
\newcommand{\bmc}{{\mathBM c}}
\newcommand{\bmp}{{\mathBM p}}

\newcommand{\ttx}{{\tt x}}
\newcommand{\ttR}{{\tt R}}
\newcommand{\epmu}{{\eps,\mu}}
\newcommand{\epnu}{{\eps,\nu}}
\newcommand{\B}{\mathrm{B}}
\newcommand{\A}{\mathrm{A}}
\newcommand{\bfX}{\mathbf{X}}

\newcommand{\Pol}{\mathrm{Pol}}
\newcommand{\Cov}{\mathrm{Cov}}
\newcommand{\rel}{\mathrm{rel}}

\def\symbol#1{\textcolor{black}{#1}}
\def\1{\mathbf{\symbol{1}}}

\newcommand{\rmd}{\mathrm{d}}

\newcommand{\tF}{\tilde{F}}

\newcommand{\Kappa}{\text{{\Large $\kappa$}}}
\newcommand{\Partial}{\boldsymbol{\partial}}

\DeclareFontFamily{U}{mathx}{}
\DeclareFontShape{U}{mathx}{m}{n}{<-> mathx10}{}
\DeclareSymbolFont{mathx}{U}{mathx}{m}{n}
\DeclareMathAccent{\widecheck}{0}{mathx}{"71}

\makeatletter

\DeclareRobustCommand{\TitleEquation}[2]{\texorpdfstring{\StrLeft{\f@series}{1}[\@firstchar]$\if%
b\@firstchar\boldsymbol{#1}\else#1\fi$}{#2}}

\makeatother

\usepackage[margin=1.15in]{geometry}

\begin{document}

\title{A flow approach to the generalized KPZ equation}

\author{Ajay Chandra$^{1}$, L\'{e}onard Ferdinand$^{2}$}

\institute{Imperial College London, UK 
\and Universit\'{e} Paris-Saclay, France\\[1em]
\email{a.chandra@ic.ac.uk, lferdinand@ijclab.in2p3.fr}}

\maketitle

\begin{abstract}
We show that the flow approach of Duch \cite{Duch21} can be adapted to prove local well-posedness for the generalized Kardar-Parisi-Zhang equation. 
The key step is to extend the flow approach so that it can accommodate semi-linear equations involving smooth, non-polynomial, functions of the solution - this is accomplished by introducing coordinates for the flow built out of elementary differentials. 
\end{abstract}

\setcounter{tocdepth}{2}
\tableofcontents

\section{Introduction}
We consider the following generalized Kardar-Parisi-Zhang (gKPZ) equation 
\begin{equs}\label{eq:eqFormal}
    \big(\partial_{x_0} -\Delta_{{\tt x}}\big)\psi(x)&=S[\psi](x)\,,\\
\psi(0,\ttx)&=\varpi(\ttx)  \,,
\end{equs}
where
\begin{equs}\label{eq:Formalaction}
    S[\psi](x)&=b\big(\psi(x)\big)+\sum_{i\in[n]}d_i\big(\psi(x)\big)\partial_{\ttx_i}\psi(x)   +
    \sum_{i,j \in [n]}
    g_{ij}\big(\psi(x)\big)\partial_{\ttx_i}\psi(x)\partial_{\ttx_j}\psi(x)
    +h\big(\psi(x)\big)\xi(x)\,, \qquad \quad 
\end{equs}
and we are solving for the unknown $\psi: [0,1] \times \T^n \rightarrow \R$. 
Here and in what follows we write  $x=(x_0,{\tt x})\in[0,1]\times\T^n$ with $\ttx=(\ttx_1,\dots,\ttx_n)$, in particular throughout the paper $n$ always denotes our dimension of space. 
Above $\xi$ is rough Gaussian noise, which we now make precise. 
\begin{assumption}
We take $\xi$ to be a Gaussian noise over $\R \times \T^n$ with covariance $$\Cov(x)\equiv\E[\xi(x)\xi(0)]\eqdef\delta_{\R}(x_0)\,(1-\Delta_{\ttx})^{1-n/2-\alpha}(\ttx)\,,$$ 
where, for $n \geqslant 2$, we require $\alpha \in (0,1]$ while, for $n = 1$, we require $\alpha \in (1/4,1]$. 
Above, $\delta_{\R}$ denotes the Dirac delta distribution on $\R$. 

Note that when $(n,\alpha)=(1,1/2)$, we take $\xi$ to be a space-time white noise, that is $\Cov=\delta_{\R \times \T^{n}}$.
\end{assumption}\label{assump:noise}
Throughout the paper, given $\beta \in \R \setminus \Z$, $\CC^{\beta}$ will denote either a space or parabolic space-time H\"{o}lder-Besov space with regularity exponent $\beta$ (see Section~\ref{subsec:notation}).
We will then have $\xi \in \CC^{\alpha -\eta }(\R \times \T^n)$ for every $\eta>0$. 
The $b$, $d_i$, $g_{ij}$ and $h$ will be taken to be sufficiently regular functions on $\R$, see Theorem~\ref{thm:main} for details. 

The equation \eqref{eq:eqFormal} is singular (i.e. classically ill-posed) for $\alpha \leqslant 1$ since, in this case, $\psi$ is not expected to have sufficient regularity to allow any of the products on the RHS of \eqref{eq:Formalaction} to be canonically defined. 
The class of gKPZ equations has been an important test case for singular SPDE methods, gKPZ was a motivating example driving the development of a full local well-posedness theory within the theory of regularity structures \cite{Hai14,CH16,BHZ19,BCCH21}. 
The gKPZ equation also naturally appears in the local description of the stochastic evolution of loops on manifolds and studied in a geometric context in the landmark paper \cite{StringManifold}. 

Beyond regularity structures, gKPZ has also been studied in \cite{KM17}\footnote{For a less general equation, the nonlinear dependence on $\psi$ was required to be polynomial in this article.} via discrete Wilsonian renormalization group (RG) and \cite{BB19} via paracontrolled calculus. 
There has also been work on the convergence of discretized gKPZ in \cite{BN22} and quasilinear gKPZ in \cite{BGN24,bailleul2023regularity}. 

\begin{remark}
The equation \eqref{eq:eqFormal} is ``subcritical'' or ``super-renormalizable'' when
$\alpha > 0$ -- this is a crucial assumption required for frameworks that provide probabilistically strong local well-posedness for such singular SPDE, such as regularity structures \cite{Hai14}, paracontrolled calculus \cite{GIP15}, path-wise renormalization group methods \cite{Kupiainen2016,Duch21}, and spectral gap methods \cite{OSSW21,LOTT21}. 

The additional constraint $\alpha>1/4$ for $n=1$ is not tied to subcriticality, but instead rules out diverging variances for stochastic objects appearing in the description of the solution, this phenomena means there is little hope for a probabilistically strong solution theory - see \cite{Hairer24}.
\end{remark}

\begin{remark}
Another family of subcritical gKPZ equations can be obtained by fixing $n=2$, $\xi$ to be space-time white noise and, analogously to \cite{Duch21}, tuning the roughness of the equation by using fractional Laplacians in the linear part of \eqref{eq:eqFormal}. 
Replacing $(\partial_{x_0}  - \Delta_{\ttx})$ with $(\partial_{x_0} + (- \Delta_{\ttx})^s)$, the full subcritical regime for this equation would correspond to $s > 1$. 

While switching to this setting would create many changes in power-counting calculations throughout this entire article, in principle arguments along the same lines as ours should still allow one to prove local well-posedness for all $s > 1$. 

We chose not to work in this setting since the present one more easily includes the standard gKPZ equation where one takes $n=1$ and $\alpha = 1/2$. 
\end{remark}

\begin{remark}
Assumption~\ref{assump:noise} is quite rigid, but our proof would easily apply to any stationary space-time Gaussian noise $\xi$ with covariance
kernel $\Cov(x)$ bounded by $(|x_0|^{1/2}+|\ttx|)^{-(4-2\alpha)}$ for small $(x_0,\ttx)$, along with analogous bounds on
space-time derivatives of $\Cov$ (see Remark~\ref{rem:Rem4_16}).
\end{remark}

\subsection{Main results}
The local well-posedness theory for singular SPDE like \eqref{eq:eqFormal} is often formulated in terms of a regularization and renormalization procedure.
For $\eps \in (0,1]$ we consider regularized equations
\begin{equs}\label{eq:eqDef}
    \big(\partial_{x_0}-\Delta_{{\tt x}}\big)\psi_\eps(x)&=S_\eps[\psi_\eps](x)\,,\\    \psi_\eps(0,\ttx)&=\varpi_\eps(\ttx)\,,
\end{equs}
where $\varpi_\eps$ is a smooth function and
\begin{equs}  \label{eq:Feps}  S_\eps[\psi_\eps](x)&\eqdef\; 
b(\psi_\eps(x))+\sum_{i\in[n]}d_i(\psi_\eps(x))\partial_{\ttx_i}\psi_\eps(x)   +
    \sum_{i,j \in [n]}
    g_{ij}(\psi_\eps(x))\partial_{\ttx_i}\psi_\eps(x)\partial_{\ttx_j}\psi_\eps(x)\textcolor{white}{lal}\\&\quad\;+h(\psi_\eps(x))\xi_\eps(x)+\mfc_\eps(\psi_\eps(x),\partial_{\ttx}\psi_\eps(x))\,.
\end{equs}
Here $\xi_\eps\eqdef\rho_\eps*\xi$ is a mollification of the noise $\xi$, where $\rho_{\eps}(x) = \mcS_\eps\rho(x)\eqdef\eps^{-(n+2)}\rho(x_0/\epsilon^2,{\ttx}/\eps)$ and $\rho \in C^{\infty}(\R \times \T^{n})$ is even in space, compactly supported, and satisfies both $0 \leqslant \rho \leqslant 1$ and $\int_{\R \times \T^{n}} \rho = 1$.    
The $\mfc_\eps$ appearing  above is a renormalization counterterm, it is a linear combination of explicit functions of $\psi$ and $\partial_{\ttx} \psi$ with prefactors that are allowed to diverge as $\eps \downarrow 0$ - see \eqref{eq:counterterm}. 

\begin{notation}
In order to give our assumptions on the initial condition, we need to introduce the following three parameters depending only on $\alpha$. We set
\begin{equs}
    \Gamma\eqdef\lfloor 4/\alpha\rfloor\,,\quad\delta\eqdef-2+\alpha+\frac{\alpha}{2}(\Gamma+1)\,,\quad\text{and}\quad\kappa_0\eqdef(\alpha/2)\wedge\big(\delta/(2\Gamma+2)\big)\,.
\end{equs}
$\Gamma$ is a large integer, $\delta$ is strictly larger than $\alpha$, and $\kappa_0$ is a small parameter.
\end{notation}

We formulate two different assumptions our our initial data: one in the general case of the gKPZ equation, and one less restrictive in the case of the generalised parabolic Anderson model (gPAM) that is to say when $d_i=0$ and $g_{ij}=0$ for every $i,j\in[n]$.
\begin{assumption}\label{assump:IC}
In the general case, we assume that there exists $\kappa\in(0,\kappa_0]$ such that $\varpi_\eps$ converges to $\varpi$ in $\mcC^{\alpha-\kappa}(\T^n)$ as $\eps\downarrow0$.
\end{assumption}
\begin{assumption}\label{assump:IC1}
If $d_i=0$ and $g_{ij}=0$ for every $i,j\in[n]$, we make a slightly more general choice of initial data and only assume that $\varpi_\eps$ converges to $\varpi$ in $L^\infty(\T^n)$ as $\eps\downarrow0$.
\end{assumption}

For fixed $\eps > 0$, the equation \eqref{eq:eqDef} is always locally well-posed  by classical arguments since everything is smooth, our local well-posedness result is that there is a choice of local counterterms $\big(\mfc_\eps: \eps \in (0,1] \big)$, given in \eqref{eq:counterterm}, depending on $\psi_\eps$ and $\partial_{\ttx}\psi_\eps$, such that the corresponding solutions $\psi_{\eps}$ converge as $\eps \downarrow 0$ to a non-trivial limit. 
Below, and throughout the paper, we denote by $H\varpi_\eps(x)\eqdef e^{x_0\Delta}\varpi_\eps$ the heat flow of the initial condition $\varpi_\eps$, and we write $C^{k}(\R)$ for the standard Banach space of $k$-fold continuously differentiable functions on $\R$. 

We first state our main result in the general case.
\begin{theorem}\label{thm:main}
Fix $\alpha\in \big( 0\vee(1/2-n/4),1 \big]$ and $\kappa\in(0,\kappa_0]$. 
Suppose that for $1 \leqslant i \leqslant j \leqslant n$, we have $b, d_i, g_{ij}, h \in C^{1+\Gamma+3N_1^{3\Gamma+1}}(\R)$, and that $(\varpi_\eps)_{\eps\in[0,1]}$ is as per Assumption~\ref{assump:IC}.

Then, for the counter-terms $(\mfc_{\eps})_{\eps\in(0,1]}$ as given in \eqref{eq:counterterm}, there exists a random variable $0 < T \leqslant 1$ as follows: for any deterministic $\tilde{T} \in (0,1]$, on the event $\{\tilde{T} \leqslant T \}$, \eqref{eq:eqDef} is well posed on $\CC^{\alpha -\kappa }([0,\tilde{T}] \times \T^n)$  with solution $\psi_{\eps}$ which converges to a limit $\psi$ in $\CC^{\alpha -\kappa}( [0,\tilde{T}] \times \T^n)$ as $\eps \downarrow 0$.
\end{theorem}
In the case of gPAM we can start with rougher initial data, and our main result is as follows.
\begin{theorem}\label{thm:maingPAM}
Fix $\alpha\in \big( 0\vee(1/2-n/4),1 \big]$. 
Suppose that for $1 \leqslant i \leqslant j \leqslant n$, we have $d_i=0$, $g_{ij}=0$ and $b, h \in C^{1+\Gamma+3N_1^{3\Gamma+1}}(\R)$, and that $(\varpi_\eps)_{\eps\in[0,1]}$ is as per Assumption~\ref{assump:IC1}.

Then, for the counter-terms $(\mfc_{\eps})_{\eps\in(0,1]}$ as given in \eqref{eq:counterterm}, there exists a random variable $0 < T \leqslant 1$ as follows: for any deterministic $\tilde{T} \in (0,1]$, on the event $\{\tilde{T} \leqslant T \}$, \eqref{eq:eqDef} is well posed on $L^\infty([0,\tilde{T}] \times \T^n)$. Furthermore, for any $\eta>0$, $\psi_{\eps}-H\varpi_\eps$ converges to a limit $\psi-H\varpi$ in $\CC^{\alpha -\eta}( [0,\tilde{T}] \times \T^n)$ as $\eps \downarrow 0$. In particular, $x_0^{(\alpha-\eta)/2}\psi(x_0,\bigcdot)$ lies in $\mcC^{\alpha-\eta}(\T^n)$ uniformly in $x_0>0$.
\end{theorem}
\begin{proof}[of the main Theorems]
The fact that the regularized solution $\psi_\eps$ belongs to $\mcC^{\alpha-\kappa}([0,\tilde T]\times\T^n)$ uniformly in $\eps>0$ is the statement of Lemma~\ref{lem:solutionconstruction}. We discuss the convergence of $\psi_\eps$ as $\eps\downarrow0$ in Remark~\ref{rem:convergenceSol}. 
The bounds in Corollary~\ref{coro:1} and the definition of $\mcG$ in Definition~\ref{def:mcG} show that $b,d_{i},g_{ij},h \in C^{1+\Gamma+3N_1^{3\Gamma+1}}(\R)$ suffices.

The fact that in the case of gPAM Assumption~\ref{assump:IC} can be improved to Assumption~\ref{assump:IC1} stems from the following: the regularity of the initial condition in only used in Lemma~\ref{lem:harmoex} below, that contains two different kinds of bounds. On the one hand, \eqref{eq:harmoniccompletion2} and \eqref{eq:harmoniccompletion} only assume that the initial condition is bounded, while on the other hand \eqref{eq:harmoniccompletionwithderivative} and \eqref{eq:harmoniccompletion3}  require more regularity. However, it turns out that these last two bounds, that allow to bound the space derivatives of the quantities controlled in \eqref{eq:harmoniccompletion2} and \eqref{eq:harmoniccompletion} are only necessary when we need to control the gradient of the argument of the effective force, that is to say when our ansatz for the effective force contains some derivatives. Since in the case of gPAM the ansatz made in \ref{eq:ansatzF} and \eqref{eq:ansatzFa} does not contain any derivative, we then do not need to rely on \eqref{eq:harmoniccompletionwithderivative} and \eqref{eq:harmoniccompletion3}.
\end{proof}

Our main statements allow us to restart the equation from the solution. Therefore, going from local in time existence to maximal in time solutions using our analysis is straightforward. 

The fact that the limiting solution $\psi$ is non-trivial (i.e. not simply a constant or just the solution to the linear equation) is not hard to infer from the detailed local descriptions of $\psi_{\eps}$ and $\psi$ obtained from our local solution theory, we do not include these arguments here.

\subsection{The flow approach}\label{sec_intro_flowapproach}
For clarity we neglect the initial data in this subsection, that is we set $\varpi_\eps= \varpi=0$.

The pathwise RG approach to singular SPDE involves introducing an ``effective scale'' $\mu \in [0,1]$ in addition to the regularization scale $\eps$, and an associated family of regularized Green's functions $\big( G_{\mu}: \mu \in [0,1] \big)$ with $G_{0} = (\partial_{x_0}  - \Delta_{\ttx})^{-1}$ and $G_{1} = 0$.
The Green's function $G_{\mu}$ suppresses space (resp. time) scales $\lesssim$ $\mu $ (resp. $\mu^2$). 
The space-time field $\psi_{\eps,\mu}$ defined by 
\begin{equs}\label{eq:eqDefreg}
\psi_{\eps,\mu} (x)&= G_{\mu} S_\eps[\psi_\eps](x)
\end{equs}
is then a regularization of $\psi_{\eps}$ suppressing parabolic scales $\lesssim$ $\mu $. 

Control over the $\eps \downarrow 0$ limit of $\psi_{\eps}$ on small space-time scales (i.e. local well-posedness) 
can be obtained by establishing, for some\footnote{We will need to impose $\mu$ is very small, which will lead to existences for small times} $\mu > 0$, \textit{uniform in} $\eps$ control over $\psi_{\eps,\mu}$. 
The philosophy of RG is to obtain this type of control by first deriving an ``effective'' equation for $\psi_{\eps,\mu}$ and then obtaining good analytic control on how the effective equation evolves when $\mu$ is increased. 
The ``effective equation'' here, introduced in the context of singular SPDE by \cite{Kupiainen2016}, is given by
\begin{equs}\label{eq:eqDefreg2}
\psi_{\eps,\mu} (x)&= G_{\mu} S_{\eps,\mu}[\psi_{\eps,\mu}](x)\;,\\
\end{equs}
where we enforce that 
\begin{equ}\label{eq:force_identity}
S_{\eps,\mu}[\psi_{\eps,\mu}]
=
S_{\eps}[\psi_{\eps}]\;.
\end{equ}
The objects such as $S_{\eps,\mu}$ or $S_{\eps}$ are called forces, they are mappings on space-time fields that play the role of the RHS of an equation such as \eqref{eq:eqDefreg2}

The work \cite{Kupiainen2016} analyzed a discretized dynamic $S_{\eps,\mu_{n}} \mapsto S_{\eps,\mu_{n+1}}$ that preserved \eqref{eq:force_identity} in the context of a specific singular SPDE - the stochastic $\Phi^4_3$ equation given by
\[
(\partial_{x_0} - \Delta_{\ttx}) \psi = - \lambda \psi^3 + \xi\;,
\]
where formally $\psi:[0,\infty) \times \T^3 \rightarrow \R$ and $\xi$ is space-time noise on $[0,\infty) \times \T^3$. 

However, much of the analysis of \cite{Kupiainen2016} had to be done by hand, which made treating more irregular equations impractical. 
A major breakthrough came in \cite{Duch21}, the starting point of which is to take derivatives in $\mu$ of \eqref{eq:force_identity} to get the following continuous dynamic in the space of (random) forces:
\begin{equs}\label{eq:polch}
\partial_{\mu} S_{\eps,\mu}[\bigcdot]
&=
- \scal{\D S_{\eps,\mu}[\bigcdot], \dot{G}_{\mu} S_{\eps,\mu}[\bigcdot]}\\
S_{\eps,0}[\bigcdot] &= S_{\eps}[\bigcdot]\;.
\end{equs}
The idea of using a continuous change in scale in RG analysis first appeared in \cite{Pol84} in Quantum Field Theory, and the dynamic \eqref{eq:polch} is called a ``Polchinski flow''.
The major contribution of \cite{Duch21} was in implementing this idea in the context of solving a singular SPDE and showing that much of the analysis done by hand in \cite{Kupiainen2016} could be systematized in this continuous RG approach. 
We mention that continuous Polchinski flows have also been applied to some closely related problems involving renormalization and stochastic analysis - see \cite{BBBsurvey,BCGNote}. 

A crucial step in obtaining detailed control of the infinite dimensional dynamic \eqref{eq:polch} is choosing a suitable set of ``coordinates''.
The coordinates allow us to parametrize the space of forces and show that, as $\mu$ increases, the evolution can diverge in a (only) finite subset of components (called relevant components) and is contracting in all other components (called irrelevant components).
The fact that we can find coordinates so there are only finitely many relevant components is closely related to the assumption of the local-subcriticality.
Control over the flow is obtained by choosing initial data $S_{\eps}[\bigcdot]$ which diverges in relevant coordinates as $\eps \downarrow 0$ to compensate for the behaviour of the flow in these components -- this tuning of initial data corresponds to renormalization. 

The decomposition/coordinates used by \cite{Kupiainen2016},\cite{Duch21} are of ``polynomial type'' - they correspond to Taylor expanding $S[\psi]$ in $\psi$ (and its space-time derivatives) about the space-time field $\psi(\cdot) = 0$ and also expanding in $\lambda$ about $\lambda = 0$. 

However, all past work using RG methods for singular SPDE has treated equations where the force appearing in the singular SPDE depends on the solution and its derivatives only polynomially. 
Such polynomial coordinates seem to be very ill-suited to equations like \eqref{eq:eqFormal} where the right hand side is non-polynomial.
Even when restricting to smooth $g_{ij},h,p_{i}$ and $b$, trying to reduce to polynomial coordinates for the flow by performing a Taylor series in $\psi$ gives an initial force which appears to be intractable for controlling the flow \eqref{eq:polch}. 

The main contribution of the present article is to show that this obstruction can be bypassed if, instead of using polynomial coordinates, we use as coordinates an analog of \textit{elementary differentials}. 

Elementary differentials first appeared in the work of Cayley \cite{Cay1857}, and more recently attracted attention in the context of the approximation of solutions to ODEs by $B$-series \cite{Butcher72,hairer74}, the analysis of rough ODEs \cite{Trees} by branched rough paths, and related path-wise methods for singular SPDE \cite{BCCH21,BrunedBSeries}.
As an illustrative example, suppose one has a driven ODE of the type $\rmd y_{t} = f(y_t)\rmd X_t$, then combining Picard iteration with Taylor expanding $f$ leads one to expect that  for $s,t \in \R$ close together, one has a good local expansion 
\[
y_{t} - y_{s} \approx 
\sum_{a \in \N^\N, |a| < \infty }
\Upsilon^{a}(f)(y_s) 
\mathbf{X}^{a}_{t,s}\;.
\]
Above, for each multi-index $a$, $\mathbf{X}^{a}_{t,s}$ is an ``increment'' of a linear combination of iterated integrals of $X$. 
All the dependence on $y$ on the right hand side is encoded via the elementary differentials $\Upsilon^{a}(f)(y_s)$ defined as follows. 
For $a = (a_0,a_1,\dots)$ we set  
\[
\Upsilon^{a}(f)(\cdot)
=
\prod_{i \in \N}
\big(
f^{(i)}(\cdot) \big)^{a_i}\;.
\]
Elementary differentials are often indexed by trees rather than multi-indices, but much of this structure carried over to multi-indices when working with equations with scalar valued equations -- see \cite{LOTAlgebra,JZ23}. 
Since \eqref{eq:eqFormal} is a scalar equation, and since the non-linear term in the Polchinski flow equation is easier using multi-indices, we also adopt this approach here.
In the work \cite{CFODEs} we use the Polchinski flow approach to study a vector valued stochastic differential equation and index the flow coordinates by trees.  

We also note that Polchinski flow is non-local. 
For this reason we use a non-local version of the elementary differentials associated to \eqref{eq:eqFormal} as coordinates, see \eqref{eq:elem_differential}.  

We also follow an idea of \cite{GR23} (also developed in \cite{Duch23}) and study an approximate form of \eqref{eq:polch} with an associated remainder problem.
Instead of constructing a trajectory of forces $\big(S_{\eps,\mu}[\bigcdot] \big)_{ \mu \in (0,1]}$ satisfying \eqref{eq:force_identity}, we allow ourselves to postulate some trajectory of forces $\big(F_{\eps,\mu}[\bigcdot] \big)_{\mu \in (0,1]}$ and then derive a dynamic in the scale $\mu$ for a remainder space-time field $R_{\eps,\mu}$ by enforcing
\begin{equ}\label{eq:effectiveS}
F_{\eps,\mu}[\psi_{\eps,\mu}] + R_{\eps,\mu}
=
S_{\eps}[\psi_{\eps}] \;.
\end{equ}
In practice we choose $F_{\eps,\mu}$ to satisfy a truncated, finite dimensional, form of \eqref{eq:polch} chosen to give give an equation for $R_{\eps,\mu}$ that we can control without any further expansion into coordinates. 
The dynamic for $F_{\eps,\mu}$ contains all the relevant parts of the flow and it is also where all the renormalization is carried out. 

In \cite{GR23} the use of an approximate flow was crucial to prove global in space and time well-posedness. In our case, we need it in order to take the effect of the initial condition into account, see Remark~\ref{rem:newidea} below.

A related difference with \cite{GR23} is the following. While we will be able to construct forces $F_{\eps,\mu}$ for the full range $\mu \in (0,1]$, we will only be able to construct the remainder $R_{\eps,\mu}$ for $\mu \in (0,\mu_{T}]$ where $\mu_{T} \in (0,1]$ is a random scale. 
This is in analogy to how, in regularity structures, the ``model'' can be constructed globally over space and time but the distinctions between the problems of global versus local-posedness are encountered when solving a remainder equation. 

\subsection*{Acknowledgements}

{\small
LF thanks Sarah-Jean Meyer for a very helpful discussion. 
AC gratefully acknowledges partial support by the EPSRC through the Standard Grant “Multi-Scale Stochastic Dynamics with Fractional Noise” EP/V026100/1 and also the “Mathematics of Random Systems” CDT EP/S023925/1. 
The authors thank Pawe\l{} Duch for pointing out some errors in an earlier version of the article. 
}

\subsection{Notations}\label{subsec:notation}
We will write $\mcH$ for the Hilbert space $\R^n$ endowed with its usual inner product.

We let $\T^n\eqdef(\R/\Z)^n$ denote the $n$ dimensional torus. For $a<b\in\R$, we write $\Lambda_{a;b}\eqdef[a,b]\times\T^n$, and also write $\Lambda\eqdef(-\infty,1]\times\T^n$.

We often consider functions $\psi$ taking values in a finite dimensional vector space $\mcE$ (typically a Hilbert space) coming with a canonical basis $(e_i)_{i\in[\mathrm{dim}(\mcE)]}$. 
We write $\psi=\sum_{i\in[\mathrm{dim}(\mcE)]}\psi^i e_i$. For $p,q\in[1,\infty]$, and endow $\mcD(\Lambda_{a;b},\mcE)$ with the norms
\begin{equs}    \norm{\psi}_{L^p_{a;b}}\eqdef\max_{i\in[\mathrm{dim}(\mcE)]}\bigg(\int_{\Lambda_{a;b}}|\psi^i(x)|^p \rmd x\bigg)^{1/p}\,,\;\norm{\psi}_{L^{p,q}_{a;b}}\eqdef\max_{i\in[\mathrm{dim}(\mcE)]}\bigg(\int_{\T^n}\Big(\int_{[a,b]}|\psi^i(x)|^q\rmd x_0\Big)^{p/q}\rmd\ttx\bigg)^{1/p}\,,
\end{equs}
with the usual understanding when $p$ or $q$ is equal to $\infty$. 
We again drop the subscript $a;b$ when $\Lambda_{a;b}$ is replaced with $\Lambda$.

Throughout the paper, we repeatedly identify distributions $K\in\mcD'(\Lambda)$ with operators
    \begin{equs}
     K:\mcD(\Lambda)\ni F\mapsto   K(F)\equiv KF\eqdef K*F=\int_{\Lambda}K(\bigcdot-z)F(z)\rmd z\,.
    \end{equs}
We endow such operators $K$ with the norms
\begin{equs}    \norm{K}_{\mcL^{p,\infty}_{a;b}}\equiv\norm{K}_{\mcL(L^p_{a;b},L^\infty_{a;b})}\,,\;\norm{K}_{\mcL^{(p,q),\infty}_{a;b}}\equiv\norm{K}_{\mcL(L^{p,q}_{a;b},L^\infty_{a;b})}\,.
\end{equs}
Note that whenever the kernel of $K$ is sufficiently integrable, then the $\mcL_{a;b}^{p,\infty}$ norm of the operator $K$ corresponds to the $L^{p/(p-1)}_{a;b}$ norm of the kernel $K$.

We call an element $\mfl=(\mfl_0,\dots,\mfl_n)$ of $\N^{n+1}$ an (space-time) index\footnote{We will not call such an element a ``multi-index'' since we use this term for different objects,}, and we identify  $\N^{n} \hookrightarrow \N^{n+1}$ by mapping $(\mfm_1,\dots,\mfm_n) \mapsto (0,\mfm_1,\dots,\mfm_n)$. 
We use the standard factorial notation $\mfl!\eqdef\prod_{i=0}^n\mfl_i!$ and also endow indices $\mfl$ with a (parabolic) size $|\mfl|\eqdef 2\mfl_0+\sum_{i\in[n]}\mfl_i$. 
We extend these notions to collections of indices $\tilde{\mfl} = (\mfl_j: j \in J)$ with $\mfl_j \in \N^{n+1}$ by setting $\tilde{\mfl}! \eqdef \prod_{j \in J} \mfl_{j}!$ and $|\tilde{\mfl}| = \sum_{j \in J} |\mfl_{j}|$.

We associate indices with differential operators on space-time in the standard way, writing $\partial^{\mfl}\eqdef\partial_{x_0}^{\mfl_0}\prod_{i=1}^n\partial_{\ttx_i}^{\mfl_i}$. 
For $\mfl\in\N^n$, we write $\partial_{\ttx}^\mfl\eqdef\d^\mfl$ to recall the fact that the derivatives are in space only. 
We view $\partial_{\ttx}\equiv\partial_{\ttx}^1=(\partial_{\ttx_1},\dots,\partial_{\ttx_n})$ acting on functions as producing something $\mcH$-valued. 
In particular we write $\norm{\partial_{\ttx}\psi}_{L^\infty}$ for $\max_{i\in[n]}\norm{\partial_{\ttx_i}\psi}_{L^\infty}$.

Given any time weight $u\in C^\infty(\R)$ and any function $\psi\in\mcD(\Lambda)$, we write $u\psi(x)\eqdef u(x_0)\psi(x)$. 
For a linear time weight we introduce the notation $\t(x_0)=x_0$ and write, for $k \in \N$,  $\t^{k}\psi(x)\eqdef (x_0)^{k} \psi(x)$.

Below, and from now on, we often just write $\Delta_{\ttx}$ instead of $\Delta$ for the spatial Laplacian. We define, for $x \in \Lambda$, the kernel 
$$G(x)\eqdef\1_{(0,\infty)}(x_0) e^{-x_0(-\Delta)}({\tt x})$$ of the fundamental solution $G$ associated to the differential operator $ (\partial_{x_0}  - \Delta)$. 

We recall the definition of the kernel $K_{N,\mu}$, following the Section 4 of \cite{Duch21}.
\begin{definition}
Fix $\mu\in(0,1]$. For $x\in\Lambda$, we define
    \begin{equs}
        Q_\mu(x)\eqdef \mu^{-2}e^{-x_0/\mu^{2}}  \1_{(0,\infty)}(x_0)\big(1-\mu^{2}\Delta\big)^{-1}({\tt x})
        \;.
    \end{equs}
Observe that $Q_\mu$ is the fundamental solution for the differential operator $\mcP_\mu\eqdef \big(1+\mu^{2}\partial_{x_0}\big)\otimes\big(1-\mu^{2}\Delta\big)   $ on $\Lambda$, with null initial condition at $x_0=-\infty$.

We point out that $\mcP_{\mu}$ contains a first derivative in time, $\mcP_\mu$ is not self adjoint.
Since we will often use integration by parts with $\mcP_{\mu}$, we also define $\mcP_\mu^\dagger  \eqdef \big(1 - \mu^{2}\partial_{x_0}\big)\otimes\big(1-\mu^{2}\Delta\big) $.

We sometimes write $Q_\mu$ for the operator given by convolution with the kernel $Q_\mu(x)$ on $\mcD(\Lambda)$ with the convention that $Q_0=\delta$.\\
We now introduce the kernels we use to test scale-dependent quantities. 
Given $N\in\N_{\geqslant1}$ we define the operator $K_{N,\mu}$ by
\begin{equs}\label{eq:defKmu}
    K_{N,\mu}\eqdef Q_\mu^{*N}\,.
\end{equs}
Note that $K_{N,\mu}$ is the fundamental solution for $\mcP^N_\mu$.
With $N_1$ given as in Theorem~\ref{coro:2}, we will use the shorthand notations 
\begin{equ}
K_\mu\eqdef K_{N_1^{3\Gamma+1},\mu}\;,
\enskip \mcR_\mu\eqdef\mcP_\mu^{N_1^{3\Gamma+1}}\,,
\text{  and  }
\mcR_\mu^{\dagger} \eqdef \big(\mcP_\mu^{\dagger} \big)^{N_1^{3\Gamma+1}}\;,
\end{equ}
along with 
\begin{equ}
K_{+,\mu}\eqdef K_{2+N_1^{3\Gamma+1},\mu}=K_{2,\mu}K_\mu
\text{  and  }\mcR_{+,\mu}\eqdef\mcP_\mu^{2+N_1^{3\Gamma+1}}\,,
\end{equ}
Lastly, for $\lambda\geqslant\tau$, we write $\tilde K_{\lambda,\tau}\eqdef\mcR_\tau K_\lambda$. 
\end{definition}

Given $\beta<0$, we define the parabolic Hölder-Besov space $\mcC^{\beta}\big((-\infty,b]\times\T^n\big))$ 
as the completion of compactly supported smooth functions on $(-\infty,b]\times\T^n$ under the norm
\begin{equs}    \norm{\bigcdot}_{\mcC^\beta((-\infty,b]\times\T^n)}\eqdef\sup_{\mu\in(0,1]}\mu^{-\beta}\norm{K_{\lceil-\beta \rceil,\mu}\bigcdot}_{L^\infty((-\infty,b]\times\T^n)}\,.
\end{equs}
For $\beta\in(0,1)$ and $b>0$, we define $\mcC^{\beta}([0,b]\times\T^n)$ as the completion of smooth functions on $(-\infty,b]\times\T^n$ supported on $[0,b]\times\T^n$ under the norm
\begin{equs}    \norm{\bigcdot}_{\mcC^\beta([0,b]\times\T^n)}\eqdef\sup_{\mu\in(0,1]}\mu^{-\beta}\norm{(Q_\mu-\Id)\bigcdot}_{L^\infty([0,b]\times\T^n)}\,.
\end{equs}
We also define corresponding H\"older-Besov spaces over $\T^n$ where $K_\mu$ is replaced by $(1-\mu^2\Delta)^{-1}$.

Given $k \in \N$ and an interval $I \subset \R$ or $I = \T^{n}$, we write $C^k(I)$ for the usual Banach space of $k$-times continuously differentiable functions from $I$ to $\R$ For $t \geqslant 0$, we write $B_0(t)\eqdef[-t,t] \subset \R$. 

Finally, we will estimates from \cite[Lemma~10.52]{Duch21} on time localisation when proving Proposition~\ref{eq:prop1} and Theorem~\ref{thm:sto} so we restate it using our notation. 
\begin{lemma}
For any $N\in\N$, $p\in[1,\infty]$ and $\mu\in(0,1]$, we denote by $\mcW^p_{N,\mu}$ the set of all time weights $u_\mu\in \mcD(\R)$ verifying
\begin{equs}
  \max_{i\leqslant N} 
    \mu^{-2/p+2i}
    \norm{\d^i_t u_\mu}_{L^p}< \infty\,.
\end{equs}
Then, for every $\psi\in \mcD(\Lambda)$, $M\geqslant N$ and $q\in[1,\infty]$, uniform in $\mu\in(0,1]$ we have
    \begin{equs}\label{eq:1052a}      \norm{K_{M,\mu}\big(u_\mu\psi)}_{L^{\infty,q}}&\lesssim \norm{K_{M,\mu}\psi}_{L^{\infty,q}}\,,\;\text{provided} \; u_\mu \in\mcW^\infty_{N,\mu}\,,\\
  \label{eq:1052b}        \norm{K_{M,\mu}\big(u_\mu\psi)}_{L^{\infty,p}}&\lesssim\mu^{2/p} \norm{K_{M,\mu}\psi}_{L^\infty}\,,\;\text{provided} \; u_\mu \in\mcW^p_{N,\mu}\,.
    \end{equs} 
\end{lemma}

\section{Setting up the flow approach}\label{sec:flow_setup}
We start with some pre-processing to handle the initial data $\varpi_\eps$, writing
\begin{equs}
    \phi_\eps(x)\eqdef \delta(x_0)\varpi_\eps  (\ttx)\,.
        \end{equs}
We can then rewrite \eqref{eq:eqDef} as
\begin{equs}\label{eq:eqDef1}    \psi_\eps(x)&=G\big(\1^0_{(0,\infty)}S_\eps[\psi_\eps]+\phi_\eps\big)(x)\,,
\end{equs}
where here and in the sequel, for any $\psi:\R\times\T^n\rightarrow\R$ and Borel subset $A \subset \R$, we write $\1^0_{A}\psi(x)\eqdef\1_{A}(x_0)\psi(x)$.

Writing 
\begin{equ}
   F_\eps[\bigcdot]\eqdef \1^0_{(0,\infty)}S_\eps[\bigcdot]
\end{equ}
we then have 
\begin{equ}
    \psi_\eps(x)=G \big( F_\eps[\psi_\eps]+\phi_\eps\big)(x)\,,
\end{equ}
Note the functional $F_\eps[\bigcdot](t,\bigcdot)$ vanishes for $t \leqslant 0$.

\subsection{Introducing the effective scale and remainder flow}\label{sec:flow}
We now introduce the cut-off function $\chi$ we use to implement our effective scale-cutoff described in Section~\ref{sec_intro_flowapproach} along with the associated cut-off Green's functions $(G_{\mu})_{\mu \in (0,1]}$. 
\begin{definition}
Fix a smooth and increasing function $\chi:\R_{\geqslant0}\rightarrow[0,1]$ 
such that $\chi|_{[0,1]}=0$ and $\chi|_{[2,\infty)}=1$.
For any $\mu\in[0,1]$ and $t\geqslant0$, 
define $\chi_\mu(t)\eqdef\chi(t/\mu^2)$, with the convention that $\chi_0(t)=1$. 
    
For $\mu \in [0,1]$ we define the heat kernel $G_\mu$ by setting, for $x\in\Lambda$, 
   \begin{equs}\label{eq:DefGmu}
       G_\mu(x)\eqdef\chi_\mu(x_0)\1_{(0,\infty)}(x_0)e^{-x_0(-\Delta)}({\tt x}) \,. \end{equs} 
Note that $G_0=G$ and $G_1=0$. Finally, we write $\dot G_\mu(x)\eqdef\d_\mu G_\mu(x)$.

We also denote by $\bfX_0\dot G_\mu$ the operator with kernel $\bfX_0\dot G_\mu(x)\eqdef x_0\dot G_\mu(x)$, and for $k\in\{0,1\}$ we denote by $\bfX_0^k \dot G_\mu$ the operator equal to $\dot G_\mu$ when $k=0$ and to $\bfX_0\dot G_\mu$ when $k=1$.
\end{definition}
\begin{remark}
Several key estimates on the operator norms of the cut-off Green's are stated in Lemma~\ref{lem:dotG}. However, here we make a special note of important support properties of $G_\mu$ and $\dot G_\mu$:
    \begin{equs}\label{eq:suppGmu}        \text{supp}\,G_\mu\subset[\mu^2,1]\times\T^n\,,\;\text{and}\;\text{supp}\,\dot G_\mu \subset[\mu^2,2\mu^2]\times\T^n\,.
    \end{equs}
\end{remark}

For fixed $\eps \in (0,1]$ and a trajectory of ``effective forces'' $\big(F_{\eps,\mu}[\bigcdot] : \mu\in(0,1] \big)$ we define a trajectory of remainder random space-time fields $\big(R_{\eps,\mu} : \mu \in (0,1]\big)$ by requiring that for every $\mu \in (0,1]$ and $x \in \Lambda_{0;1}$,  
\begin{equs}\label{eq:FmuRmu}
  F_\eps[\psi_\eps](x)= F_{\eps,\mu}[\psi_{\eps,\mu}](x)+R_{\eps,\mu}(x)\,,
\end{equs}
where $\psi_\epmu$ is defined as
\begin{equs}   \psi_{\eps,\mu}(x)&\eqdef G_\mu \big(F_\eps[\psi_\eps]+\phi_\eps)(x)\label{eq:psimu}\,.
\end{equs}
Note that \eqref{eq:FmuRmu} and \eqref{eq:psimu} allow us to write 
\begin{equs}\label{eq:phi_mu}
    \psi_\epmu(x)=G_\mu\big( F_\epmu[\psi_\epmu]+R_\epmu+\phi_\eps\big)(x)\,.
\end{equs}
As described in Section~\ref{sec_intro_flowapproach}, we will in practice choose $F_{\eps,\mu}$ to be the solution to an approximate flow equation which is written below as \eqref{eq:Pol} with initial condition 
\begin{equs}\label{eq:F_0}
    F_{\eps,0}[\bigcdot]=F_\eps[\bigcdot]=\1^0_{(0,\infty)}S_\eps[\bigcdot]\;.
\end{equs}
$F_{\eps,\mu}[\bigcdot]$, as a function of the noise $\xi$, should be thought of as a fairly explicit polynomial enhancement of the noise. 
For fixed realizations of the noise $\xi$ and space-time, $F_{\eps,\mu}[\bigcdot]$ argument will take values in spaces of space-time fields that are as rough as the noise. 
On the other hand, $R_{\eps,\mu}$ will be an inexplicit random \textit{remainder}, obtained by closing a fixed point problem. We can think of the fixed point problem for $R_{\eps,\mu}$ as being solved in $L^{(\infty,1)}(\Lambda_{0;T})$, where $T$ is a small random time. 
This analysis for both $F_{\eps,\mu}[\bigcdot]$ and $R_{\eps,\mu}$ will be stable in $\eps$ as $\eps \downarrow 0$. 
\begin{remark}
Note that since $F_{\eps,0}=F_\eps$, we necessarily have $R_{\eps,0}=0$. Moreover, since $F_\eps[\bigcdot]$ is supported on positive times, so are $F_\epmu[\bigcdot]$ and $R_\epmu$.
\end{remark}
\begin{remark}
  Plugging the definition \eqref{eq:F_0} of $F_\eps$ into the definition \eqref{eq:psimu} of $\varphi_\epmu$ and using the support property \eqref{eq:suppGmu} of $G_\mu$ shows that for every $\mu\geqslant0$, $\varphi_\epmu$ is supported after time $\mu^2$. In particular, this implies that for any $T\in(0,1]$, writing $\mu_T\eqdef\sqrt{T}$, one has for every $x\in[0,T]\times\T^n$
  \begin{equs}
      \psi_\epmu(x)=G_\mu\phi_\eps(x)+(G_\mu-G_{\mu_T})F_\eps[\psi_\eps](x)=G_\mu\phi_\eps(x)-\int_\mu^{\mu_T}\dot G_\nu\big(F_\epnu[\psi_\epnu]+R_\epnu\big)(x)\rmd\nu\,.
  \end{equs}
\end{remark}
With all of this in place, we can formulate a dynamic in $\mu$ for $R_{\eps,\mu}$ that preserves the relationship \eqref{eq:FmuRmu} and allows us to solve for it scale by scale in $\mu$.  
Observing that
\begin{equ}
   0= \frac{\rmd}{\rmd\mu} F_\eps[\psi_\eps]=  \frac{\rmd}{\rmd\mu}\Big(F_{\eps,\mu}[\psi_{\eps,\mu}]+R_{\eps,\mu}\Big)=\d_\mu F_{\eps,\mu}[\psi_{\eps,\mu}]+\D F_{\eps,\mu}[\psi_{\eps,\mu}]\d_\mu\psi_{\eps,\mu}+\d_\mu R_{\eps,\mu}\,,
\end{equ}
and combining the latter with $\d_\mu\psi_{\eps,\mu}=\dot G_\mu\big(F_\epmu[\psi_\epmu]+R_{\eps,\mu}+\phi_\eps\big)$, which is just the  application of $\partial_{\mu}$ to  \eqref{eq:phi_mu}, we obtain on $[0,T]\times\T^n$ the system of equations
\begin{subequations}\label{eq:sys1}
  \begin{empheq}[left=\empheqlbrace]{alignat=2}
  \label{eq:flow}  
\d_\mu R_{\eps,\mu}&=-\D F_{\eps,\mu}[\psi_{\eps,\mu}]\dot G_\mu\big( R_\epmu+\phi_\eps\big)-\big(\d_\mu+\D F_{\eps,\mu}[\psi_{\eps,\mu}]\dot G_\mu\big) F_{\eps,\mu}[\psi_{\eps,\mu}]\,,\\\label{eq:phimu}
        \psi_\epmu&=G_\mu\phi_\eps -\int_\mu^{\mu_T}\dot G_\nu\big(F_\epnu[\psi_\epnu]+R_\epnu\big)\rmd\nu\,,
  \end{empheq}
\end{subequations}
for any $T\in(0,1]$.
The second term of the RHS of \eqref{eq:flow}, which is quadratic in  $F_{\eps,\mu}$, plays the role of a rough forcing term in the dynamic for $R_{\eps,\mu}$. 
However we will make sure it is not too rough by choosing $F_{\eps,\mu}$ to satisfy an approximate flow equation. 
We will furthermore tune the initial data $F_{\eps}$ (i.e., renormalize) so the cumulants of $F_{\eps,\mu}$ satisfies good running bounds in $\mu$, we will be able to use stochastic arguments to control the quadratic terms contributing to \eqref{eq:flow}.  
With this in hand we will solve~\eqref{eq:sys1} for $R$ using path-wise / deterministic analysis, but we will only be able to do so locally - for $\mu \in (0,\mu_{T}] \subset (0,1]$ for some random scale $\mu_{T} > 0$ --  this is enough to solve \eqref{eq:eqDef1} \textit{locally in time}.
\begin{remark} \label{rem:newidea}
We pause to discuss the strategy used in Section~\ref{Sec:Sec3} to solve the system \eqref{eq:sys1} \textit{with} the initial condition into account. 
In the parabolic case, in \cite{Duch21} the author places the initial condition inside the stochastic objects, but we believe that this is impossible when the regularity of the solution is positive (see Remark~\ref{rem:ICdiscussion} below for more comments about this possibility). We thus choose to integrate the initial condition to the remainder, which is why the system \eqref{eq:sys1} contains the terms $G_\mu\phi_\eps$, $\dot G_\mu\phi_\eps$ corresponding to the heat flow of the initial condition regularised at scale $\mu$. These terms are estimated in Lemma~\ref{lem:harmoex} below.

The aim is to solve \eqref{eq:sys2} by performing a fixed point argument. Moreover, since $R_{\eps,0}=0$, we must work in a topology in which $R_{\eps,\mu}$ vanishes as $\mu\downarrow0$. In \cite{GR23,Duch23}, the authors study a slightly different system in which $G_\mu\phi_\eps$, $\dot G_\mu\phi_\eps$ do not appear, and they are able to show that the quantity $\|K_{N,\mu} R_\epmu\|_{L^\infty}$ vanishes for small $\mu$. In our case, the additional term 
\begin{equs}
    -\D F_{\eps,\mu}[\psi_{\eps,\mu}]\dot G_\mu\phi_\eps
\end{equs}
present on the RHS of the equation for $\d_\mu R_\epmu$ makes it unlikely that $R_\epmu$ can be controlled in the $L^\infty$ norm. Indeed, because the force is an enhancement of the noise, we do not expect $\D F_{\eps,\mu}[\psi_{\eps,\mu}]$ to behave better than the noise, that is to say that it should blow up like $\mu^{-2+\alpha}$ as $\mu\downarrow0$. On the other hand, \eqref{eq:harmoniccompletion} (with $k=0$) shows that $\dot G_\mu\phi_\eps$ gives another factor $\mu^{-1}$: in total, the RHS of the equation for $\d_\mu R_\epmu$ thus blows up like $\mu^{-3+\alpha}$, which is not integrable! 

However, \eqref{eq:harmoniccompletion} with $k=1$ shows that the behaviour of this term improves drastically when it comes with a linear time weight. 
This suggests trying to control the quantity $\|K_{N,\mu}(\t R_\epmu)\|_{L^\infty}$, where we recall that $\t$ denotes a linear time weight. The problem is that $\|K_{N+2,\mu} R_\epmu\|_{L^\infty}$ still appears in the fixed point problem for $R_\epmu$. By Sobolev embedding, we know that this quantity is bounded by $\mu^{-2}\|K_{N,\mu} R_\epmu\|_{L^{(\infty,1)}}$. This is interesting because when the kernels $K_{N,\mu}$ are not here, for any $\eta>0$ it holds true that for smooth $f:[0,1]\rightarrow\R$,
\begin{equs}
    \|f(t)\|_{L^1_t}\lesssim \|t^{1-\eta}f(t)\|_{L_t^\infty}\,.
\end{equs}
Sadly, this property is not preserved when the kernels $K_{N,\mu}$ are present, because commuting $K_{N,\mu}$ and the time weight $t^{-1+\eta}$ creates lower powers of $t$ that are no longer integrable. This is why, in order to construct the remainder, we perform a fixed point using both $\|K_{N,\mu} R_\epmu\|_{L^{(\infty,1)}}$ and $\|K_{N,\mu}(\t R_\epmu)\|_{L^\infty}$ (see Definition~\ref{def:Rnorm} and Proposition~\ref{eq:prop1}).
\end{remark}
\begin{lemma}\label{lem:harmoex}
For every $\kappa\in(0,\kappa_0]$, $k\in\{0,1\}$ and $\mfl\in\N^n$ different from 0, uniform in $\mu\in(0,1]$ we have
    \begin{equs}
    \label{eq:harmoniccompletion2}
        \| \mcR_\mu^\dagger G_\mu \phi_\eps\|_{L^\infty_{0;1}}&\lesssim \|\varpi_\eps\|_{L^\infty_\ttx}\,,\\        
         \label{eq:harmoniccompletionwithderivative}
        \| \partial_\ttx^{\mfl}\mcR_\mu^\dagger G_\mu \phi_\eps\|_{L^\infty_{0;1}}&\lesssim \mu^{-|\mfl|+\alpha-\kappa}\|\varpi_\eps\|_{\mcC^{\alpha-\kappa}_\ttx}\,,        
        \\
    \label{eq:harmoniccompletion}
        \| \mcR_\mu^\dagger(\t^k\dot G_\mu \phi_\eps)\|_{L^\infty_{0;1}}&\lesssim \mu^{-1+2k}\|\varpi_\eps\|_{L_\ttx^\infty}\,,
        \\
            \label{eq:harmoniccompletion3}
        \| \partial_\ttx^\mfl\mcR_\mu^\dagger(\t^k\dot G_\mu \phi_\eps)\|_{L^\infty_{0;1}}&\lesssim \mu^{-1-|\mfl|+\alpha+2k-\kappa}\|\varpi_\eps\|_{\mcC_\ttx^{\alpha-\kappa}}\,.          \end{equs}
\end{lemma}
\begin{proof}
We first prove \eqref{eq:harmoniccompletion}. Observe that setting $\tilde\chi_k\eqdef -2\t^{1+k}\chi'$ we have
    \begin{equs}
        \t^k\dot G_\mu \phi_\eps(x)=\mu^{-1+2k}  \tilde\chi_k(x_0/\mu^2) e^{x_0\Delta}\varpi_\eps(\ttx)
    \end{equs}
Because $ \t^k\dot G_\mu \phi_\eps(x)$ only depends on $\ttx$ through the heat flow of $\varpi_\eps$, the effect of polynomials in $-\mu^2\Delta$ and $\mu^2\partial_{x_0}$ is the same, which is why we only discuss the action of the latter. Because $\tilde\chi_k$ is smooth and compactly supported, the action of $\mu^2\partial_{x_0}$ on $\tilde\chi_k(x_0/\mu^2)$ simply turns it into another smooth function of $x_0/\mu^2$ with same support. The conclusion is that there exist smooth functions $\big(\chi^{(i)}:i\in[4N_1^{3\Gamma+1}]\big)$ supported on $[1,2]$ such that
\begin{equs}
     \mcR_\mu^\dagger(   \t^k\dot G_\mu \phi_\eps)(x)=\mu^{-1+2k}  \sum_{i=0}^{4N_1^{3\Gamma+1}}\chi^{(i)}(x_0/\mu^2) \mu^{2i}(-\Delta)^ie^{x_0\Delta}\varpi_\eps(\ttx)\,.
\end{equs}
Therefore, by classical heat kernel estimates, \begin{equs}
      \| \mcR_\mu^\dagger(\t^k\dot G_\mu \phi_\eps)\|_{L^\infty_{0;1}}&\lesssim 
\mu^{-1+2k}  \max_{i\in[4N_1^{3\Gamma+1}]} \mu^{2i}\|e^{x_0\Delta}\varpi_\eps\|_{L^\infty_{\mu^2:1}\mcC^{2i}}      \lesssim\mu^{-1+2k}  \max_{i\in[4N_1^{3\Gamma+1}]} \mu^{2i}\|e^{\mu^2\Delta}\varpi_\eps\|_{\mcC_{\ttx}^{2i}}\\
&\lesssim\mu^{-1+2k}  \max_{i\in[4N_1^{3\Gamma+1}]} \mu^{2i}  \mu^{-2i}\|\varpi_\eps\|_{L_\ttx^\infty}\,.
\end{equs}
This concludes the proof of \eqref{eq:harmoniccompletion}.

To prove \eqref{eq:harmoniccompletion3}, note that if there are more derivatives, then we obtain 
\begin{equs}
      \|\d_\ttx^\mfl \mcR_\mu^\dagger(\t^k\dot G_\mu \phi_\eps)\|_{L^\infty_{0;1}}& \lesssim\mu^{-1+2k}  \max_{i\in[4N_1^{3\Gamma+1}]} \mu^{2i}\|e^{\mu^2\Delta}\varpi_\eps\|_{\mcC_{\ttx}^{2i+|\mfl|}}\lesssim\mu^{-1+2k}  \max_{i\in[4N_1^{3\Gamma+1}]} \mu^{2i}  \mu^{-2i-|\mfl|+\alpha-\kappa}\|\varpi_\eps\|_{\mcC_\ttx^{\alpha-\kappa}}\,.
\end{equs}

\eqref{eq:harmoniccompletion2} and \eqref{eq:harmoniccompletionwithderivative} follow in the same way, since 
\begin{equs}    G_\mu\phi_\eps(x)=\chi(x_0/\mu^2)e^{x_0\Delta}\varpi_\eps(\ttx)\,,
\end{equs}
so that the time in the heat kernel can always be traded for $\mu^2$. If $\mfl=0$, then proceeding as above one thus obtains
\begin{equs}
      \| \mcR_\mu^\dagger G_\mu \phi_\eps\|_{L^\infty_{0;1}}&\lesssim 
\max_{i\in[4N_1^{3\Gamma+1}]} \mu^{2i}\|e^{\mu^2\Delta}\varpi_\eps\|_{\mcC_{\ttx}^{2i}}\lesssim \|\varpi_\eps\|_{L^\infty_\ttx}\,.
\end{equs}
If $\mfl\neq0$ however, one ends up with
\begin{equs}
      \|\partial_\ttx^\mfl \mcR_\mu^\dagger G_\mu \phi_\eps\|_{L^\infty_{0;1}}&\lesssim 
\max_{i\in[4N_1^{3\Gamma+1}]} \mu^{2i}\|e^{\mu^2\Delta}\varpi_\eps\|_{\mcC_{\ttx}^{2i+|\mfl|}}\lesssim \|e^{\mu^2\Delta}\varpi_\eps\|_{\mcC^{|\mfl|}_\ttx}
\lesssim\mu^{-|\mfl|+\alpha-\kappa}\|\varpi_\eps\|_{\mcC^{\alpha-\kappa}_\ttx}
\,.
\end{equs}
\end{proof}
\subsection{Coordinates and the approximate flow}\label{sec:ansatz}
In this section we define our choice of $F_{\epmu}$ appearing in \eqref{eq:sys1}. 
As mentioned earlier, we will arrive at our definition of $F_{\epmu}$ by solving an approximate flow $\Pol^{\leqslant\Gamma}_{\mu}$, see \eqref{eq:defPol} and \eqref{eq:Pol}. 
The structure of the approximate flow facilitates the probabilistic analysis of Section~\ref{sec:Sec4} which is the key ingredient in proving the convergence of $F_{\epmu}$ as $\eps \downarrow 0$ and controlling its contribution to \eqref{eq:sys1}. 
Our definition of the approximate flow $\Pol^{\leqslant\Gamma}_{\mu}$  is done via truncation in terms of the coordinates we introduce in Section~\ref{subsubsec:coord}. 

We take a moment to point out some additional steps for using $\Pol^{\leqslant\Gamma}_{\mu}$ for the construction and analysis of $F_{\epmu}$ that are specific to studying an initial value problem -- they also appear in \cite{Duch21} but they do not appear in \cite{Duch22,Duch23} or \cite{GR23}\footnote{Here the authors do indeed study a parabolic dynamic but their focus is obtaining energy estimates at stationarity, there is no initial value problem.}.

We will obtain $F_{\epmu}$ by solving $\Pol^{\leqslant\Gamma}_{\mu}$ with initial data 
\[
F_{\eps,0}[\bigcdot] = F_\eps[\bigcdot]= \1^0_{(0,\infty)}  S_\eps[\bigcdot ] \;.
\]
The construction of $F_{\eps,\mu}$ is done in two steps:  
\begin{enumerate}
\item We first construct a \textit{stationary effective force} trajectory $S_{\eps,\mu}$ which solves a slightly different approximate flow $\Pol^{\leqslant\Gamma}_\mu \big( S_{\eps,\mu} \big) = 0$ with initial data 
\begin{equs}
    S_{\eps,0}[\bigcdot]=
S_{\eps}[\bigcdot]\;.
\end{equs}
The initial force $S_{\eps}$ is stationary in distribution over $\R \times \T^n$ and does not involve the initial data. 
This stationary effective force will be the focus of the probabilistic analysis of Section~\ref{sec:Sec4}, which also determines our renormalization -- working with the space-time stationary force here guarantees our counterterms do not depend on time or the initial data.

\item Then, we incorporate the multiplication of the indicator function $\1^0_{(0,\infty)}$ in our initial force by showing that losing stationarity in time does not produce any new divergences in our probabilistic analysis -- this is done in Section~\ref{Sec:Sec5}. 
\end{enumerate}

\begin{remark}\label{rem:ICdiscussion}
As discussed in remark \ref{rem:newidea}, we take the
initial condition into account by solving a system of equation for the remainder that takes the heat flow of the initial solution into account. In other words, we have integrated the initial solution to the remainder. 

    Here, we diverge from \cite{Duch21}, where the author rather integrates the initial condition to the stochastic objects, working with $S_\epmu[\bigcdot+G_\mu\phi_\eps]-\phi_\eps$, which also solves the flow equation. When the equation is polynomial, this allows him to obtain an optimal regularity for the initial condition.
    
    This relies on the fact that when the solution is of negative regularity say $\varsigma<0$, $\dot G_\mu\phi_\eps$ behaves like the propagation of the noise $\dot G_\mu\xi_\eps$ in the flow equation. 
    By this we mean that, uniform in $\mu$, 
    \begin{equs}\label{eq:negIC}
        \| \dot G_{\mu}\phi_\eps\|_{L^\infty_{0;1}}\lesssim\mu^{-1+\varsigma}\|\varpi_\eps\|_{\mcC^{\varsigma}(\T^n)}\,.
    \end{equs}
    The above estimate stems from the fact that $\dot G_\mu\phi_\eps(x)=\frac{\rmd}{\rmd\mu}\chi_\mu(x_0)e^{-x_0(-\Delta)}\varpi_\eps(\ttx)$ and that the support of $\frac{\rmd}{\rmd\mu}\chi_\mu$ in localized around $\mu^2$ which implies that we have
    \begin{equs}
          \| \dot G_{\mu}\phi_\eps\|_{L^\infty_{0;1}}\lesssim\mu^{-1}\|e^{\mu^2\Delta}\varpi_\eps\|_{L^\infty(\T^n)}\,,
    \end{equs}
so that \eqref{eq:negIC} follows by classical heat flow estimates.

As shown in the proof of Lemma~\ref{lem:harmoex}, the situation in the setting of positive regularity $\alpha>0$ solutions is quite different, as one has
 \begin{equs}
          \| \dot G_{\mu}\phi_\eps\|_{L^\infty_{0;1}}\lesssim\mu^{-1}\|e^{\mu^2\Delta}\varpi_\eps\|_{L^\infty(\T^n)}\lesssim\mu^{-1}\|\varpi_\eps\|_{\mcC^{\alpha}(\T^n)}\,,
    \end{equs}
and here we do not see any gain from assuming the initial condition is of regularity $\alpha > 0$. 
A standard approach would be to use time-weighted spaces, for instance we  have
\begin{equs}
          \| x_0^{\alpha/2}\dot G_{\mu}\phi_\eps\|_{L^\infty_{0;1}}\lesssim\mu^{-1+\alpha}\|\varpi_\eps\|_{L^\infty(\T^n)}\,.
    \end{equs}
However, we have not been able to close the flow equation argument in weighted spaces since this weight does not behave well with the action of $K_{\mu}$ which appears in our running bounds for the flow equation. This is why, rather than integrating it to the effective force, we have chosen to place the initial condition inside the remainder.
\end{remark}

\subsubsection{Multi-indices and elementary differentials}\label{subsubsec:coord}
In this section we introduce the coordinates for our approximate flow. 
 \begin{definition}
Given a finite dimensional vector space $V$, a $V$-valued \textit{local functional} $F:\mcD(\Lambda)\rightarrow\mcD(\Lambda,V)$ is a functional whose value $F[\psi](x)$ at $\psi\in\mcD(\Lambda)$ and $x\in\Lambda$ only depends on finitely many components of $\big( \partial_{\ttx}^{n} \psi(x) \big)_{n \in \N}$. 
When $V=\R$, we often just call $F$ a local functional. 

To lighten notation, we lift any function $f:\R\rightarrow V$ to a $V$-valued local functional defined by $f[\psi](x)\eqdef f(\psi(x))$.

For $i\in[n]$, we introduce the local functional ${\Partial}_i[\psi](x)\eqdef\d_i\psi(x)$, and view $\Partial=(\Partial_1,\dots,\Partial_n)$ as an $\mcH$-valued local functional.

\end{definition}
 \begin{definition}   
Let $F:\mcD(\Lambda)\rightarrow\mcD(\Lambda,V)$ be a Fréchet differentiable functional. 
Given $\psi,\vphi\in\mcD(\Lambda)$, we use the shorthand notation
$\D F[\psi]\vphi\eqdef\langle\D F[\psi],\vphi\rangle_{L^2(\Lambda)}$ to denote the Fréchet derivative $\D F$ of $F$ evaluated at $\psi$ and tested against $\vphi$. 
Observe that 
\begin{equs}
\D\Partial_i[\psi]\vphi=\Partial_i[\vphi]\,.
\end{equs}
\end{definition}

\begin{remark}
For convenience, when introducing our coordinates we will assume that all the functions $b$, $d_i$, $g_{ij}$, and $h$ appearing in \eqref{eq:Formalaction} are in fact smooth 
It will be clear that this smoothness assumption can be relaxed as described in our main theorem since only finitely many multi-indices/coordinates will appear in our analysis -- see Remark~\ref{rem:notsmooth}. 

We also view the families of functions $d=(d_i)_{i\in[n]}$ and  $g=(g_{ij})_{i,j\in[n]}$ as elements $d\in C^\infty(\R,\mcH)$ and $g\in C^\infty(\R,\mcH\otimes_s \mcH)$ -- here and in what follows $\otimes_{s}$ denotes symmetric tensor product. 
\end{remark}

\begin{definition}\label{def:usefullocalfunctionals}  
There is a natural bilinear map $(\mcH \otimes_{s} \mcH) \times \mcH \rightarrow \mcH$, and we use this map to define the $\mcH$-valued local functional
\begin{equs}
  f[\psi]\equiv  (g\Partial)[\psi]\eqdef g[\psi]\big(\Partial[\psi]\big)\,.
\end{equs}
Similarly, using the inner products on $\mcH$ and $\mcH \otimes_{s} \mcH$, we define the ($\R$-valued) local functionals 
\begin{equs}
   c[\psi]\equiv   (d\Partial)[\psi]\eqdef d[\psi]\big(\Partial[\psi]\big)\,,
   \qquad 
   e[\psi]\equiv   (g\Partial^2)[\psi]\eqdef g[\psi]\big(\Partial[\psi]\otimes\Partial[\psi]\big)\,,
\end{equs}
where the inner product is implied on both RHS appearing above. 

We can then rewrite \eqref{eq:Feps} as 
\begin{equs}
    S_{\eps}[\psi_\eps]=b[\psi_\eps]+c[\psi_\eps]+e[\psi_\eps]+h[\psi_\eps]\xi_\eps+\mfc_\eps[\psi_\eps]\,,
\end{equs}
for some local functional $\mfc_{\eps}$. 
\end{definition}
\begin{definition}
We let 
\begin{equs}
 \mcN_\Z\eqdef 
    \big\{ q = (q_{i})_{i=0}^{\infty} \in\Z^\N:\exists I\in\N,\forall i\geqslant I,q_i=0
    \big\}\,,\;\text{and}\;\mcN\eqdef\{q\in\mcN_\Z: q_i\in\N\;\forall i\geqslant0\}\,.
\end{equs}
For a sequence $q\in \mcN$, we define \textit{order} $\mfo(q)$ of $q$ and the \textit{size} $\mfs(q)$ of $q$ as
\begin{equ}
\mfo(q)\eqdef \sum_{i\geqslant0}i q_i\,,\;\text{and}\;\mfs(q)\eqdef\sum_{i\geqslant0}q_i\,.
\end{equ}
We write $\supp(q)\eqdef\{i\in\N:q_i\neq0\}$ for the \textit{support} of $q$, and call $\mfl(q)\eqdef\max \big( \mathrm{supp}(q) \big)$ the \textit{length} of $q$. 
We also write $[q]\eqdef\{(i,j)\in\N^2:i\in\text{supp}(q),j\in[q_i]\}$.
For $\tilde{\mcH} \in \big\{ \mcH,\mcH\otimes_s\mcH \big\}$ and $q \in \mcN$ we define the Hilbert space 
\begin{equs}
    \tilde{\mcH}^q
    &\eqdef \bigotimes_{(i,j)\in[q]}   
    \tilde{\mcH}\,.
\end{equs}
\end{definition}
\begin{definition}
Fix $q\in\mcN$, and a collection $y^q=(y^q_{ij})_{ij\in[q]}\in\Lambda^{[q]}$ of elements of $\Lambda$ indexed by $[q]$. 
Given $k\in C^{\mfl(q)}(\R)$, we introduce a multi-index notation for functionals made of products of the local functionals associated with the derivatives of $k$, defining the functional
\begin{equ}
    \mathbf k^{q}[\psi](y^q)\eqdef\prod_{(i,j)\in[q]}k^{(i)}[\psi](y^q_{ij})\,.
\end{equ}
In the sequel, we will only be interested in the cases $k=b,h$ (where $b,h$ are introduced in \eqref{eq:Formalaction}) and define $\bfb^q$, $\bfh^q$ accordingly.
  \end{definition} 
\begin{remark}\label{rem:notsmooth} For $q\in\mcN$, $\mathbf k^{q}$ is bounded as
\begin{equs}\label{eq:bound1jet}
    \norm{\mathbf k^{q}[\psi]}_{L^\infty}\leqslant\norm{k}^{\mfs(q)}_{C^{\mfl(q)}(B_0(\norm{\psi}_{L^\infty}))}\,.
\end{equs}   
\end{remark}
\begin{definition}
Recalling that $d=(d_i)_{i\in[n]} \in C^\infty(\R,\mcH)$ and $g=(g_{ij})_{i,j\in[n]} \in C^\infty(\R,\mcH\otimes_s \mcH)$ (where $d,g$ are introduced in \eqref{eq:Formalaction}), we define, for each $q \in \mcN$ and $y^q\in\Lambda^{[q]}$, the  $ \mcH^q$-valued functional
   \begin{equs}       \bfd^{q}[\psi](y^q)\eqdef\prod_{(i,j)\in[q]}d^{(i)}[\psi](y^q_{ij})\,.  
   \end{equs}    
   and the $\big(\mcH\otimes_s\mcH\big)^{q}$-valued functional 
   \begin{equs}       \bfg^{q}[\psi](y^q)\eqdef\prod_{(i,j)\in[q]}g^{(i)}[\psi](y^q_{ij})\,.  
   \end{equs} 
We also define $\mcH^q$-valued functional
\begin{equs}    \Partial^q[\psi](y^q)\eqdef\prod_{(i,j)\in[q]}\Partial[\psi](y^q_{ij})\,,
\end{equs}
the functional
\begin{equs}
 \bfc^q[\psi](y^q)\equiv   \big(\bfd\Partial\big)^q[\psi](y^q)\eqdef \bfd^{q}[\psi]\big(\Partial^q[\psi]\big)(y^q)\,,
\end{equs}
the $\mcH^q$-valued functional
\begin{equs}
\bff^{\,q}[\psi](y^q)\equiv    \big(\bfg\Partial\big)^q[\psi](y^q)\eqdef \bfg^{q}[\psi]\big(\Partial^q[\psi]\big)(y^q)\,,
\end{equs}
and the functional
\begin{equs}
 \bfe^q[\psi](y^q)\equiv   \big(\bfg\Partial^2\big)^q[\psi](y^q)\eqdef \bfg^{q}[\psi]\big(\Partial^q[\psi]\otimes_s\Partial^q[\psi]\big)(y^q)\,.
\end{equs}
Note in the definitions of $\bfc^q$, $\bff^{\,q}$, and $\bfe^q$ we are using on the RHS the tensor product mappings and dualities used in the definitions for $c$, $e$, and $f$ in Definition~\ref{def:usefullocalfunctionals}. 
\end{definition}
\begin{definition}
    We let $\mathfrak I\eqdef\{\mfb,\mfc,\mfd,\mfe,\mff,\mfg,\mfh\}$ be the \textit{set of labels}.
\end{definition}
\begin{definition}
We define a set of all septuples of elements of $\mcN$ indexed by ${\mathfrak I}$,  
\begin{equs}
    \mathring\mcM\eqdef \{a=(a^\mfk)_{\mfk\in{\mathfrak I}} 
    =
    (a^{\mfb},a^{\mfc},a^\mfd,a^{\mfe},a^{\mff},a^\mfg,a^\mfh) \in \mcN^7\}\,,
\end{equs}
and define $\mathring\mcM_\Z$ accordingly, with $\mcN$ replaced by $\mcN_\Z$.

We call elements of $\mathring{\mcM}$ \textit{pre-multi-indices}. 
We extend the notations of order and size to any pre-multi-index $a=(a^\mfb,a^{\mfc},a^\mfd,a^\mfe,a^{\mff},a^\mfg,a^\mfh)\in\mathring\mcM$, denoting by
\begin{equs}
   \mfo (a)&\eqdef
   \sum_{\mfk\in\mfI}\mfo(a^\mfk)   +\mfs(a^\mfd)+\mfs(a^{\mff})+2\mfs(a^\mfg)\\&=\mfo(a^\mfb)+\mfo(a^\mfc)+\mfo(a^\mfd)+\mfs(a^\mfd)+\mfo(a^\mfe)+\mfo(a^\mff)+\mfs(a^{\mff})+\mfo(a^\mfg)+2\mfs(a^\mfg)+\mfo(a^{\mfh})
\end{equs}
the \textit{order} of $a$, and by 
\begin{equs}
        \mfs(a)\eqdef\sum_{\mfk\in\mfI}\mfs(a^\mfk)=
        \mfs(a^\mfb)+\mfs(a^\mfc)
        +\mfs(a^\mfd)+\mfs(a^\mfe)+\mfs(a^\mff)+\mfs(a^{\mfg})+\mfs(a^\mfh)
\end{equs}
the \textit{size} of $a$. Moreover, we define the \textit{scaling} $|a|$ of $a$ as
\begin{equs}
    |a|\eqdef -(2-\alpha)\mfs(a)+2\mfo(a)+(2-\alpha)\big(\mfs(a^\mfb)+\mfs(a^\mfc)+\mfs(a^\mfe)\big)+(1-\alpha)\big(\mfs(a^{\mfd})+\mfs(a^{\mff})\big)-\alpha\mfs(a^{\mfg})\,,
\end{equs}
and we let 
\begin{equs}
    \mfl(a)\eqdef \max_{\mfk\in\mfI}\mfl(a^\mfk)=
    \mfl(a^\mfb)\vee\mfl(a^\mfc)\vee\mfl(a^\mfd)
    \vee    \mfl(a^\mfe)\vee\mfl(a^\mff)\vee\mfl(a^\mfg)\vee\mfl(a^\mfh)\,.
\end{equs}
stand for the \textit{length} of $a$. 

We also define the \textit{support} of $a$ as $\mathrm{supp}(a)\eqdef\{(\mfk,i)\in{\mathfrak I}\times\N:i\in\mathrm{supp}(a^\mfk)\}$, and we let $[a]\eqdef\{(\mfk,i,j)\in{\mathfrak I}\times\N^2:(i,j)\in[a^\mfk]\}$. Finally, we associate $a$ with a Hilbert space $\mcH^a$ given by
\begin{equ}
    \mcH^a\eqdef \mcH^{a^\mfd}\otimes\mcH^{a^\mff}\otimes\big(\mcH\otimes_s\mcH\big)^{a^\mfg}\,,
\end{equ} 
and we define for $y^a=(y^a_{\mfk ij})_{\mfk ij\in[a]}\in\Lambda^{[a]}$ the $\mcH^a$-valued functional
\begin{equs}\label{eq:elem_differential}
 \Upsilon^{a}&[\psi](y^a) \\&\eqdef 
  \bfb^{a^\mfb}[\psi](y^{a^\mfb})
 \bfc^{\,a^\mfc}[\psi](y^{a^\mfc})
 \bfd^{\,a^\mfd}[\psi](y^{a^\mfd})
 \bfe^{a^\mfe}[\psi](y^{a^\mfe})
 \bff^{\,a^\mff}[\psi](y^{a^\mff})
 \bfg^{\,a^\mfg}[\psi](y^{a^\mfg})
\bfh^{a^\mfh}[\psi](y^{a^\mfh}) \,,
\end{equs}
where, for $\mfk\in\mfI$, we use the shorthand notation $y^{a^\mfk}=(y^{a^\mfk}_{ij})_{ij\in[a^\mfk]}\eqdef (y^a_{\mfk ij})_{ij\in[a^\mfk]}$.
\end{definition}
\begin{definition}
    Fix a collection $(k_i)_{i\in I}$ of functions $k_i:\R\rightarrow\R$ indexed by a finite set $I$. For any $n\in\N$ and any interval $B$ of $\R$, we write
    \begin{equs}
        \norm{(k_i)_{i\in I}}_{C^n(B)}\eqdef\max_{i\in I}\norm{k_i}_{C^n(B)}\,.
    \end{equs}
\end{definition}
\begin{lemma}
We have the following estimate on $\Upsilon^{a}[\psi]$:
\begin{equs}    \label{eq:boundIa}\norm{\Upsilon^a[\psi]}_{L^\infty}\lesssim \Big(\norm{(b,d,g,h)}_{C^{\mfl(a)}(B_0(\norm{\psi}_{L^\infty}))}\Big)^{\mfs(a)}    \norm{\partial_{\ttx}\psi}_{L^\infty}^{\mfs(a^\mfc)+2\mfs(a^\mfe)+\mfs(a^\mff)}\,.
\end{equs}    
\end{lemma}
\begin{definition}
We define, for $\mfk\in\mfI$ and $k\in\N$, a pre-multi-index $\1^\mfk_k\in\mathring\mcM$ by
    \begin{equs}
        \big({\text{$\1$}}^{\text{\scriptsize{$\mfk$}}}_k\big)^{\mathfrak r}_{i}\eqdef \1\{(\mathfrak r,i)=(\mfk,k)\}\,.
    \end{equs}
\end{definition}
The set of pre-multi-indices  $\mathring\mcM$ is the natural one to use to index all of the local functionals that can be built out of products and derivatives of the terms appearing on the RHS of \eqref{eq:Formalaction}. 
However, this set will be too large to be useful. 

We are interested in using pre-multi-indices to expand solutions and remainders to a truncated analog of the flow \eqref{eq:polch}.  
Our initial data for this flow is described by the pre-multi-indices $\{ \1^\mfk_0,\; \mfk\in\{\mfb,\mfc,\mfe,\mfh\}\}$ and new pre-multi-indices are then populated through recursive ``insertions'' of multi-indices into each other due to the bilinear term on the RHS of \eqref{eq:polch}. 

Such an ``insertion recipe'' can be captured by trying to associate at least one rooted tree with typed vertices to a pre-multi-index $a\in\mathring\mcM$.
This is done by viewing $a$ as a multi-set of vertices which carry a type, indexed by $\mfk \in \mfI$, and have a prescribed number of descendants. 
Then $a^{\mfk}_{i} \in \N$ is the number of vertices in $a$ of type $\mfk$ with precisely $i+\1\{\mfk\in\{\mfd,\mff\}\}+2\1\{\mfk=\mfg\}$ descendants. 

Some pre-multi-indices are compatible with multiple such rooted trees. 
As an example, take $a=(0,\dots,0,a^\mfh)$ with $a^\mfh=(2,1,1,0,0,\dots)$  -- then $a$ is compatible with the following two different rooted trees, $\treeOne$ and $\treeTwo$, which each contain four vertices of type $\mfh$.

However there are also elements of $\mathring\mcM$ that do not correspond to any such tree, for instance take $a$ as above with $a^{\mfh} = (2,1,0,0,
\dots)$. 
Then the corresponding tree must have three vertices, so the root must have a child -- but then it is only allowed to have one child. 
We will see that pre-multi-indices that do not correspond to any tree will never be populated through our flow. 

This motivates the following definition, which gives a necessary and sufficient condition for an element of $\mathring{\mcM}$ to correspond to a tree.

\begin{definition}\label{def:pop}
We say $a \in \mathring \mcM$ is \textit{populated} if it satisfies the constraint
\begin{equ}\label{eq:pop_constraint}
\mfs(a)=\mfo(a)+1\,.
\end{equ} 
\end{definition}
\begin{remark}\label{rem:scaling relations}
An important consequence of the population constraint is that it allows us to rewrite scaling relations. 
Suppose that $a \in \mathring{\mcM}$ is populated, then observe we must have
\begin{equs}
    |a|&=-2+\alpha\big(\mfs(a)-\mfs(a^{\mfg})\big)+(2-\alpha)\big(\mfs(a^\mfb)+\mfs(a^\mfc)+\mfs(a^\mfe)\big)+(1-\alpha)\big(\mfs(a^\mfd)+\mfs(a^{\mff})\big)\\    &=-2+\alpha\mfs(a^{\mfh})+\big(\mfs(a^\mfd)+\mfs(a^\mff)\big)+2\big(\mfs(a^\mfb)+\mfs(a^{\mfc})+\mfs(a^{\mfe})\big)\label{eq:scalingGood}
    \,.\textcolor{white}{lalalala}
\end{equs}
Moreover, we also have
\begin{equs}
    \mfl(a)\leqslant\mfo(a) 
    \qquad
    \text{ and }
    \qquad 
        \mfs(a^\mfg)\leqslant\frac{\mfs(a)-1}{2}\,.
\end{equs}
For the second inequality above, observe that vertices of type $\mfg$ have at least 2 descendants so that the trees maximizing $\mfs(a^\mfg)$ are binary with internal nodes labeled $\mfg$ and leafs labeled $\mfb$, $\mfc$, $\mfe$ or $\mfh$. 
This last inequality implies that
\begin{equs}\label{eq:LowerBoundScaling}
    |a|&\geqslant-2+\alpha+\frac{\alpha}{2}\mfo(a)+(2-\alpha)\big(\mfs(a^\mfb)+\mfs(a^\mfc)+\mfs(a^\mfe)\big)+(1-\alpha)\big(\mfs(a^\mfd)+\mfs(a^{\mff})\big)\,.\textcolor{white}{lal}
\end{equs}
\end{remark}

Recall that will we actually solve a truncated form of \eqref{eq:polch} with a remainder satisfying \eqref{eq:flow}. 
Our truncation corresponds to dropping pre-multi-indices whose order is too large for their scaling to be negative. In order to keep a bit of room, we actually only drop pre-multi-indices with scaling strictly larger than $\alpha$. We then introduce $\Gamma=\lfloor 4/\alpha\rfloor<\infty$, the smallest integer such that 
\begin{equs}\label{eq:def_delta}
    \delta=-2+\alpha+\frac{\alpha}{2}(\Gamma+1)>\alpha\,.
\end{equs}
It follows that, since $\alpha\leqslant1$, for any populated pre-multi-index $a  \in \mathring \mcM$,  $\mfo(a)\geqslant\Gamma+1$ implies
\begin{equs}\label{eq:lastboundsizea}
    |a|\geqslant\delta+\alpha+(2-\alpha)\big(\mfs(a^\mfb)+\mfs(a^\mfc)+\mfs(a^\mfe)\big)+(1-\alpha)\big(\mfs(a^\mfd)+\mfs(a^{\mff})\big)
\end{equs}
thanks to \eqref{eq:LowerBoundScaling}. 

With all this in mind, we introduce the following sets of pre-multi-indices
\begin{definition} \label{def:multi-indices}
For $k\in\N$, we define $\mcM_{\leqslant k} \subset \mathring{\mcM}$ to be the set of all pre-multi-indices $b \in \mathring{\mcM}$ which are populated and also satisfy $\mfo(b) \leqslant k$. 
We write $\mcM_*\eqdef\mcM_{\leqslant2\Gamma+1}$ and $\mcM_{>k}\eqdef\mcM_*\setminus\mcM_{\leqslant k}$. 
Finally, we write $\mcM\eqdef\mcM_{\leqslant\Gamma}$ and call an element of $\mcM$ a \textit{multi-index}. 
\end{definition}

With all the previous notations in hand, we are now able to introduce our ansatz for $S_{\eps,\mu}$, the stationary force solving $\Pol^{\leqslant\Gamma}_\mu(S_\epmu)=0$ with initial condition $S_{\eps,0}=S_\eps$. We make the assumption that for all $\eps,\mu\in[0,1]$, $S_{\eps,\mu}$ is given for $x\in\Lambda$ and $\psi\in\mcD(\Lambda)$ by
\begin{equ}\label{eq:ansatzF}
    S_\epmu[\psi](x)=\sum_{a\in\mcM} S^a_\epmu[\psi](x)\,,
\end{equ}
where 
\begin{equs}\label{eq:ansatzFa}
     S^a_\epmu[\psi](x)=\int_{\Lambda^{[a]}} \langle \xi^a_\epmu(x,y^a) ,\Upsilon^a[\psi](y^a)\rangle_{\mcH^a}\rmd y^a\,,
\end{equs}
and the integration is over variables $y^a=(y^a_{\mfk ij})_{\mfk ij\in[a]}$ indexed by the set $[a]$ such that $y^a_{\mfk ij}\in\Lambda$, hence the notation $y^a\in\Lambda^{[a]}$. To lighten to notation, we also write that $(x,y^a)\in\Lambda^{[a]+1}\eqdef\Lambda\times\Lambda^{[a]}$. The kernels $\big(\xi^a_\epmu\big)_{\epmu\in(0,1]}^{a\in\mcM}$ of the $  (S^a_\epmu[\psi])_{a\in\mcM}$ belong to $\mcD(\Lambda^{[a]+1},\mcH^a)$ for $\eps>0$, and are called the \textit{stationary force coefficients}.

Now that we have the coordinates for the flow, and the ansatz for the stationary effective force, let us discuss the truncation of the flow.
\begin{definition}\label{def:proj} 
Fix $k\leqslant2\Gamma+1$. We define the projector $\Pi^{\leqslant k}$ acting on functionals $L:\mcD(\Lambda)\rightarrow\mcD(\Lambda)$ of the form for $\psi\in\mcD(\Lambda)$ and $x\in\Lambda$ 
\begin{equs}\label{eq:formL}
    L[\psi](x)=\sum_{a\in\mcM_*}\int_{\Lambda^{[a]}}\langle\lambda^a(x,y^a),\Upsilon^a[\psi](y^a)\rangle_{\mcH^a}\rmd y^a
\end{equs}
 by projection onto the space spanned by pre-multi-indices of order at most $k$. More precisely, we have
\begin{equs}
    \Pi^{\leqslant k}L[\psi](x)\eqdef\sum_{a\in\mcM_{\leqslant k}}\int_{\Lambda^{[a]}}\langle\lambda^a(x,y^a),\Upsilon^a[\psi](y^a)\rangle_{\mcH^a}\rmd y^a\,.
\end{equs}

Finally, we define an operator $\Pol^{\leqslant\Gamma}_{\mu}$ acting on families $(L_\mu)_{\mu\in(0,1]}$ of functionals of the form \eqref{eq:formL} by
\begin{equs}\label{eq:defPol}
    \Pol^{\leqslant\Gamma}_{\mu}(L_\mu)[\bigcdot]\eqdef\d_\mu L_\mu[\bigcdot]+\Pi^{\leqslant\Gamma}\big( \D L_\mu[\bigcdot]\dot G_\mu L_\mu[\bigcdot]\big)\,.
\end{equs}
$\Pol^{\leqslant\Gamma}_{\mu}$ is a truncated version of the Polchinski flow equation. 
\end{definition}
Our aim is first to construct $S_\epmu$ such that it solves the equation
\begin{equs}
   \Pol^{\leqslant\Gamma}_{\mu}(S_\epmu)=0  
\end{equs}
with initial condition $S_{\eps,0}=S_\eps$, and then to ultimately construct $F_\epmu$ such that it solves the equation 
\begin{equs}
   \Pol^{\leqslant\Gamma}_{\mu}(F_\epmu)=0 \label{eq:Pol} 
\end{equs}
 with initial condition $F_{\eps,0}=F_\eps=\1^0_{(0,\infty)} S_\eps[\bigcdot ] $. Provided this condition is satisfied, the system~\eqref{eq:sys1}
 rewrites as  \begin{subequations}\label{eq:sys2}
  \begin{empheq}[left=\empheqlbrace]{alignat=2}
  \label{eq:eqRwelldefined}  
\d_\mu R_{\eps,\mu}&=-\D F_{\eps,\mu}[\psi_{\eps,\mu}]\dot G_\mu\big( R_\epmu+\phi_\eps\big)-H_\epmu[\psi_\epmu]\,,\\
        \psi_\epmu&=G_\mu\phi_\eps -\int_\mu^{\mu_T}\dot G_\nu\big(F_\epnu[\psi_\epnu]+R_\epnu\big)\rmd\nu\,,\label{eq:vphimu}
  \end{empheq}
\end{subequations}
where we introduced
\begin{equs}\label{eq:expL}
    H_\epmu[\bigcdot]\eqdef(1-\Pi^{\leqslant\Gamma})\big( \D F_\epmu[\bigcdot]\dot G_\mu F_\epmu[\bigcdot]\big)\,.
\end{equs}
We will show that this last rewriting of \eqref{eq:eqDef1} holds uniformly in $\eps$, and that the equation is therefore formally well-posed.

\begin{remark}
$F_\epmu$ solving the truncated Polchinski flow equation with initial condition given by $F_\eps[\bigcdot]= \1^0_{(0,\infty)}S_\eps[\bigcdot]$, it will depend a collection of non stationary in time stochastic objects. Nevertheless, we will show in Section~\ref{Sec:Sec5} that it can be constructed starting from the stationary solution to $\Pol_\mu^{\leqslant\Gamma}$, $S_\epmu$, without any additional renormalisation, so that the renormalisation that $F_\epmu$ requires is thus constant in time.
\end{remark}

\subsubsection{The flow equation for the stationary force coefficients}
\begin{remark}
Fix $a\in\mcM$ and $(x,y^a)\in\Lambda^{[a]+1}$. By definition, $\xi^a_\epmu(x,y^a)$ is invariant under the action of 
\begin{equs}
  \mathfrak S^a\eqdef \prod_{(\mfk,i)\in\mathrm{supp}(a)}\mathfrak S_{a^\mfk_i}\,,  
\end{equs}
where for $\sigma=(\sigma^\mfk_i)_{(\mfk,i)\in\mathrm{supp}(a)} \in  \mathfrak S^a$, each $\sigma^\mfk_i \in \mathfrak S_{a^\mfk_i}$ acts by permuting the $(y^a_{\mfk ij})_{j\in[a^\mfk_i]}$. Moreover, we use multi-index notations writing 
\begin{equs}
    a!\eqdef\prod_{(\mfk,i)\in\mathrm{supp}(a)}a^\mfk_i!\,.
\end{equs}
\end{remark}
\begin{definition}
    We define the set $\mcD$ of all \textit{derivators} $\bfd$ $  =\big(\mfk_0(\bfd),\mfk_1(\bfd),k_0(\bfd),k_1(\bfd)\big)$ by
\begin{equs}
  \mcD\eqdef  \bigcup_{(\mfk,k)\in{\mathfrak I}\times\N}\{(\mfk,\mfk,k,k+1)\}
  \cup  \bigcup_{k\in\N}\{(\mfc,\mfd,k,k),\; (\mfe,\mff,k,k),\; (\mff,\mfg,k,k)\}  
  \subset{\mathfrak I}\times{\mathfrak I}\times\N\times\N \,.
\end{equs}
To lighten the notations, we often drop the dependence on $\bfd$ of its elements, simply writing $\bfd=(\mfk_0,\mfk_1,k_0,k_1)$. Moreover, we associate to each derivator $\bfd\in\mcD$ a septuplet of sequences in $\mcN_\Z$ given by
$$d(\bfd)\eqdef\1^{\mfk_1(\bfd)}_{k_1(\bfd)}-\1^{\mfk_0(\bfd)}_{k_0(\bfd)}\in\mathring\mcM_\Z\,. $$  

\end{definition}
\begin{lemma}
    Pick $b\in\mcM$, $\bfd\in\mcD$, and for $y^b\in\Lambda^{[b]}$ define $y^{b+d(\bfd)}\in\Lambda^{[b+d(\bfd)]}$ by
\begin{equs}     
y^{b+d(\text{\scriptsize{$\bfd$}})}_{\mfk ij}\eqdef \begin{cases}  
      y^b_{\mfk_0 k_0b^{\mfk_0}_{k_0}} & \text{if} \;(\mfk,i,j)= (\mfk_1, k_1,b^{\mfk_1}_{k_1}+1)\,, \\
        y^{b}_{\mfk ij} & \text{otherwise} \,. \\
\end{cases}   \end{equs}
    Then, it holds
    \begin{equs}\label{eq:exprDI}
     \int_{\Lambda^{[b]}}\langle\xi^b_\epmu &(x,y^b),   \big({\text{$\D$}} \Upsilon^b[\psi]\vphi\big)(y^b)\rangle_{\mcH^b}\rmd y^b\\&=  \int_{\Lambda^{[b]}}\sum_{\bfd\in\mcD}
        b^{\mfk_0}_{k_0}(1+\1\{(\mfk_0,\mfk_1)=(\mfe,\mff)\})
        \langle\xi^b_\epmu (x,y^b) \partial_{\ttx}^{k_0+1-k_1}\vphi(y^b_{\mfk_0 k_0b^{\mfk_0}_{k_0}}),   
        \Upsilon^{b+d(\bfd)}[\psi](y^{b+d(\bfd)})\rangle_{\mcH^{b+d(\bfd)}} \rmd y^b\,.
    \end{equs}
\end{lemma}

\begin{definition}\label{def:B}
    Fix multi-indices $a,b,c\in\mcM$ and a derivator $\bfd=(\mfk_0,\mfk_1,k_0,k_1)\in\mcD$ verifying $a=b+c+d(\bfd)$, and pick $\sigma=(\sigma^\mfk_i)_{(\mfk,i)\in\mathrm{supp}(a^\mfk)}\in\mfS^a$. Given some $y^a=(y^a_{\mfk ij})_{\mfk ij\in[a]}\in\Lambda^{[a]}$, we define some $y^b= (y^b_{\mfk ij})_{\mfk ij\in[b]}\in\Lambda^{[b]}$ and $y^c=(y^c_{\mfk ij})_{\mfk ij\in[c]}\in\Lambda^{[c]}$ as follows. First, let 
\begin{equs}
    z\eqdef y^{a}_{{\mfk_1}k_1\sigma^{\mfk_1}_{k_1}(a^{\mfk_1}_{k_1})}\,.
\end{equs}    
Then, $y^b$ is defined by
 \begin{equs}     
y_{\mfk ij}^{b}\eqdef \begin{cases}  
      z & \text{if} \;(\mfk,i,j)= (\mfk_0, k_0,b^{\mfk_0}_{k_0})\,, \\
        y^{a}_{\mfk i\sigma^\mfk_i(j)} & \text{otherwise} \,, \\
\end{cases}   \end{equs}
and $y^c$ is defined by
\begin{equs}     
y_{\mfk ij}^{c}\eqdef y^{a}_{\mfk i\sigma^\mfk_i(b^\mfk_i+j-\1\{(\mfk,i)=(\mfk_0,k_0)\})}\,.
\end{equs}       
Recall that for any function $\psi$, we view its spatial gradient $\d_{\ttx}\psi=\big(\d_{\ttx_1}\psi,\dots,\d_{\ttx_n}\psi\big)$ as an $\mcH$ valued function. For $X\in\mcD(\Lambda^{[b]+1},\mcH^b)$ and $Y\in\mcD(\Lambda^{[c]+1},\mcH^c)$ we define the $\mcD(\Lambda^{[a]+1},\mcH^a)$-valued bilinear map 
\begin{equs}\label{eq:defB}
    \B_\mu(X,Y)(x,&y^a)\\&\eqdef \frac{b^{\mfk_0}_{k_0}}{a!}(1+\text{$\1$}\{(\mfk_0,\mfk_1)=(\mfe,\mff)\})\,X(x,y^b)\int_\Lambda\partial_{\ttx}^{k_0+1-k_1}\dot G_\mu( z-w)Y(w,y^c)\rmd w\,.
\end{equs}
Note that $\B_\mu$ is indeed well-defined, since one does have $\mcH^a=\mcH^b\otimes \mcH^c\otimes \mcH^{\otimes(k_0+1-k_1)}$.\\
Finally, in view of this definition, we introduce, for $a\in\mcM_*$, the index set for the flow equation of $\xi^a_\epmu$:
  \begin{equs}
     \text{$ \mathrm{Ind}$}(a)\eqdef\Big\{(\sigma,b,c,\bfd)\in\mfS^a\times\mcM^2\times\mcD:a=b+c+d(\bfd)\Big\}\;.
  \end{equs}  
\end{definition}
With the notation of Definition~\ref{def:B}, we have
\begin{equs}\label{eq:rhsPol}
    \D S_\epmu[\psi]\dot G_\mu S_\epmu[\psi](x)= \sum_{a\in\mcM_{*}} \sum_{\substack{(\sigma,b,c,\bfd)\in\text{\scriptsize{$\mathrm{Ind}(a)$}}}} \int_{\Lambda^{[a]}}
   \langle \B_\mu(\xi^b_\epmu,\xi^c_\epmu)(x,y^a),\Upsilon^a[\psi](y^a)\rangle_{\mcH^a}\rmd y^a\,.\textcolor{white}{blablablab}
\end{equs}
We now give a crucial estimate on $\B_\mu$. 
\begin{lemma}
For every $N\in\N$, uniformly in $X\in\mcD(\Lambda^{[b]+1},\mcH^b)$, $Y\in\mcD(\Lambda^{[c]+1},\mcH^c)$ and $\mu\in(0,1]$ we have
    \begin{equs}\label{eq:propB}
        \|K_{N,\mu}^{\otimes[a]+1}&\B_\mu(X,Y)\|_{L^\infty_x L^1_{y^a}(\Lambda^{[a]+1})}\\&\lesssim \|   \mcP^{2N}_\mu \dot G_\mu\|_{\mcL^{\infty,\infty}}
\|K_{N,\mu}^{\otimes[b]+1}X\|_{L^\infty_x L^1_{y^b}(\Lambda^{[b]+1})}
\|K_{N,\mu}^{\otimes[c]+1}Y\|_{L^\infty_x L^1_{y^c}(\Lambda^{[c]+1})}    \,.   
    \end{equs}
\end{lemma}
\begin{proof}
Denoting by $C$ the prefactor in \eqref{eq:defB}, and by $y^{\hat b}$ the collection of all the variables inside $y^b$ expect for $z$, we have 
    \begin{equs}        {}&K_{N,\mu}^{\otimes[a]+1}\B_\mu(X,Y)(x,y^a)\\&=C\int_\Lambda K_{N,\mu}(z-v)\big(K_{N,\mu}^{\otimes[\hat b]+1}\otimes\Id\big)X(x,y^{\hat b},v)\int_\Lambda
\mcP^N_\mu\partial_{\ttx}^{k_0+1-k_1}\dot G_\mu( v-w)
K_{N,\mu}^{\otimes[c]+1}
Y(w,y^c)\rmd w  \rmd v  \\
&=C\int_\Lambda K_{N,\mu}^{\otimes[ b]+1}X(x,y^{\hat b},v)  \big(\mcP_\mu^N\big)^\dagger \Big(K_{N,\mu}(z-\bigcdot)\int_\Lambda
\mcP^N_\mu\partial_{\ttx}^{k_0+1-k_1}\dot G_\mu( \bigcdot-w)
K_{N,\mu}^{\otimes[c]+1}
Y(w,y^c)\Big)(v)\rmd w  \rmd v\,.
    \end{equs}
When the operator $\big(\mcP_\mu^{N}\big)^\dagger$ hits the kernel of $K_{N,\mu}$
this can create some space-time derivatives of $K_{N,\mu}$ multiplied by $\mu$. Then, by \eqref{eq:spacederivKmu}, these newly created kernels also are in $L^1$ uniformly in $\mu$. This means that there exist some kernels $\big(A_\mu^{(\mfl)}:\mfl\in \N^{n+1},|\mfl|\leqslant 4N\big)$ belonging to $L^1$ uniformly in $\mu$ such that 
    \begin{equs}        {}&K_{N,\mu}^{\otimes[a]+1}\B_\mu(X,Y)(x,y^a)\\ 
&=C\int_\Lambda K_{N,\mu}^{\otimes[ b]+1}X(x,y^{\hat b},v)  \sum_{\mfl\in \N^{n+1},|\mfl|\leqslant 4N} A_\mu^{(\mfl)}(z-v)\int_\Lambda
\mu^{|\mfl|}\partial^{\mfl}\mcP^N_\mu\partial_{\ttx}^{k_0+1-k_1}\dot G_\mu( v-w)
K_{N,\mu}^{\otimes[c]+1}
Y(w,y^c)\rmd w  \rmd v\,,
    \end{equs}
which concludes the proof.   
\end{proof}
\begin{lemma}
The definition of $\Pi^{\leqslant\Gamma}$ combined with \eqref{eq:rhsPol} implies that the equation $\Pol_\mu^{\leqslant\Gamma}(S_\epmu)$ rewrites, projected onto the multi-indices, as the following system of equations for the stationary force coefficients
\begin{equs}\label{eq:flow2}
    \d_\mu &\xi^a_\epmu=-\sum_{\substack{(\sigma,b,c,\bfd)\in\text{\scriptsize{$\mathrm{Ind}(a)$}}}}    \B_\mu(\xi^b_\epmu,\xi^c_\epmu)
    \,.
\end{equs}
\end{lemma}
\begin{remark}
    This system of equations is hierarchical in the order of the multi-indices, since it holds $\mfo(b),\mfo(c)\leqslant\mfo(a)-1$.
\end{remark}
Before we carry on with the statement of the properties of the stationary force coefficients, let us already point out one immediate consequence of the flow equation \eqref{eq:flow2}.
\begin{lemma}Fix $a\in\mcM$, $\epmu\in(0,1]$. We have the following support property of $\xi^a_\epmu$:   
\begin{equs}    \xi^a_\epmu(x,y^a)=0\;\mathrm{if}\; y^a_0\notin[x_0-2\mu^2\mfo(a),x_0]^{[a]}\,.\label{eq:supportXi}
\end{equs}
\end{lemma}
\begin{proof}
The proof is by induction over the order of $a$. With the notation of Definition~\ref{def:B}, the induction hypothesis already implies the result for all the $y^a$ that are contained in $y^b$. Regarding the ones that are contained in $y^c$, for $\xi^a_\epmu(x,y^a)$ to be non zero, the inductions hypothesis implies that $y^a_0$ must be at most at a distance $2\mu^2\mfo(c)$ of $w_0$ (in its past) which implies that, by the support properties of $\dot G_\mu$ (see \eqref{eq:suppGmu}) it must be at at most $2\mu^2(\mfo(c)+1)$ of $z_0$. Again, the induction hypothesis gives that $z_0\in[x_0-2\mu^2\mfo(b),x_0]$, and we and up with $y^a_0$ being at most at a distance $2\mu^2(\mfo(b)+\mfo(c)+1)=2\mu^2\mfo(a)$ of $x_0$, hence the thesis.
\end{proof}
\subsection{Norms for stochastic objects}
In Section~\ref{sec:Sec4}, we show that \eqref{eq:flow2} is sufficient to control by induction all the force coefficients, in the topology that we define hereafter.
\begin{definition}\label{def:normforxi}
    Consider a collection of random $(\lambda^a_\epmu)^{a\in\mcM}_{\epmu\in(0,1]}$ with $\lambda^a_\epmu\in\mcD(\Lambda^{[a]+1},\mcH^a)$. For every $N\geqslant1$, $a\in\mcM$ and $\eps,\mu\in(0,1]$, we let
    \begin{equs}
        K^{\otimes[a]+1}_{N,\mu}\lambda^a_\epmu(x,y^a)\eqdef \big(K_{N,\mu}\otimes\dots\otimes K_{N,\mu}\big)*\lambda^a_\epmu(x,y^a)\,,
    \end{equs}
and endow $\lambda^a_\epmu$ with the norm
    \begin{equs}     \label{eq:defTripleNorm}  \nnorm{\lambda^a_{\eps,\mu}}_N\eqdef \norm{K^{\otimes[a]+1}_{N,\mu}\lambda^a_{\eps,\mu}}_{L^\infty_xL^1_{y^a}(\Lambda^{[a]+1})}\equiv \sup_{x\in\Lambda}\int_{\Lambda^{[a]}} |K^{\otimes[a]+1}_{N,\mu} \lambda^a_{\eps,\mu}(x,y^a)|\rmd y^a\,.
    \end{equs}
For every $a\in\mcM$ and $(x,y^a)\in\Lambda^{[a]+1}$, given a family $\mfl^a=(\mfl^a_{\mfk ij})_{\mfk ij\in [a]}$ of indices $\mfl^a_{\mfk ij}\in\N^{n+1}$, we also set 
\begin{equ}
\text{$\mathbf X$}^{\mfl^a}(x,y^a)
\eqdef
\prod_{\mfk ij\in[a]}
\frac{(x-y^a_{\mfk ij})^{\mfl^a_{\mfk ij}}}{\mfl^a_{\mfk ij}!}\,,
\end{equ}
and we write
\begin{equs}
    \lambda^{a,\mfl^a}_{\eps,\mu}(x,y^a)\eqdef\bfX^{\mfl^a}(x,y^a)\lambda_{\eps,\mu}^a(x,y^a)\,.
\end{equs}    
Finally, for $P\in\N_{\geqslant1}$, $\eta>0$ and $c>0$, we define
\begin{equs}
\Norm{\lambda}_{P,N,\eta}  \eqdef \E\Big[\Big( \max_{a\in\mcM}\max_{\mfl_a\in(\{0,1\}^{n+1})^{[a]}}\sup_{\eps,\mu\in(0,1]}\mu^{-|a|-|\mfl^a|+\eta}\nnorm{\lambda^{a,\mfl^a}_{\eps,\mu}}_{N}\Big)^P\Big]^{1/P}\,,
\end{equs}
and 
\begin{equs}
 \Norm{\lambda}_{P,N,\eta,c} \eqdef \E\Big[\Big( \max_{a\in\mcM}\max_{\mfl_a\in(\{0,1\}^{n+1})^{[a]}}\sup_{\eps,\eps',\mu\in(0,1]}(\eps\vee\eps')^{-c}\mu^{-|a|-|\mfl^a|+\eta}\nnorm{\lambda^{a,\mfl^a}_{\eps,\mu}-\lambda^{a,\mfl^a}_{\eps',\mu}}_{N}\Big)^P\Big]^{1/P}\,.
\end{equs}
\end{definition}
The following Theorem is proven in Section~\ref{sec:Sec4}.
\begin{theorem}\label{thm:sto} 
Fix $u\in\mcD(\R)$ and let $u\xi=\big((u\xi)^a_\epmu\big)^{a\in\mcM}_{\epmu\in(0,1]}$ denote the collection of all the stationary for coefficients multiplied by the time weight $u$, that is to say for $(x,y^a)\in\Lambda^{[a]+1}$, we write $(u\xi)^a_\epmu(x,y^a)\eqdef u(x_0)\xi^a_\epmu(x,y^a)$. Then, there exist a universal/combinatorial constant $N_0\geqslant1$ such that the following holds: for every $P{\geqslant1}$, there exist $c^\star,\eta^\star>0$ such that for all $c\in(0,c^\star]$ and $\eta\in(0,\eta^\star]$, we have
\begin{equs}\label{eq:boundXi}
  \Norm{u\xi}_{P,N_0^{2\Gamma+1},\eta}\vee  \Norm{u\xi}_{P,N_0^{2\Gamma+1},\eta,c}\lesssim_{P,u} 1\,.
\end{equs}
\end{theorem}
In Section~\ref{Sec:Sec5}, we show that one can construct the effective force $F_\epmu$ knowing only the stationary force $S_\epmu$ without any additional renormalization. To do so, we introduce an ansatz for $F_\epmu$ \eqref{eq:expressionFmu} below, which is inspired by our ansatz for $S_\epmu$, and by the way we deal with the initial condition -- see the discussion at the beginning of Section~\ref{Sec:Sec5}. We have the following result about $F_\epmu$.
\begin{theorem}
  \label{coro:2} 
There exists a collection $(\zeta^a_\epmu)^{a\in\mcM}_{\eps,\mu\in(0,1]}$ of random variables $\zeta^a_\epmu\in\mcD(\Lambda_{0;1}^{[a]+1},\mcH^a)$ called $\mathrm{force}$ $ \mathrm{ coefficients} $ such that it holds
\begin{equs}\label{eq:expressionFmu}
      F_\epmu[\psi](x)=\sum_{a\in\mcM} F^a_\epmu[\psi](x)=\sum_{a\in\mcM}\int_{\Lambda^{[a]}} \langle \zeta^a_\epmu(x,y^a) ,\Upsilon^a[\psi](y^a)\rangle_{\mcH^a}\rmd y^a\,.
\end{equs}    
Moreover, there exist a universal/combinatorial constant $N_1\geqslant N_0$ such that the following holds: for all $P\in\N_{\geqslant1}$, there exist $c^\star,\eta^\star>0$ such that for all $c\in(0,c^\star]$ and $\eta\in(0,\eta^\star]$, $(\zeta^a_\epmu)^{a\in\mcM}_{\eps,\mu\in(0,1]}$, it holds 
\begin{equs}\label{eq:boundZeta}
  \Norm{\zeta}_{P,N_1^{3\Gamma+1},\eta}\vee  \Norm{\zeta}_{P,N_1^{3\Gamma+1},\eta,c}\lesssim_P 1\,.
\end{equs}
\end{theorem}
\begin{remark}\label{rem:convergenceZeta}
    The estimate $\Norm{\zeta}_{P,N_1^{3\Gamma+1},\eta,c}\lesssim_P 1$ implies the existence of a collection   $(\zeta^a_{0,\mu})^{a\in\mcM}_{\mu\in(0,1]}$ of random variables verifying for $\eta>0$ the estimate
\begin{equs}
    \E\Big[\Big( \max_{a\in\mcM}\sup_{\mu\in(0,1]}\mu^{-|a|+\eta}\nnorm{\zeta^a_{0,\mu}}_{N_1^{3\Gamma+1}}\Big)^P\Big]^{1/P}\lesssim_P 1
\end{equs}    
and such that the family $(\zeta^a_\epmu)_{\mu\in(0,1]}$ converges in probability to $(\zeta^a_{0,\mu})_{\mu\in(0,1]}$ as $\eps\downarrow0$.
\end{remark}
With the force coefficients at hand, we are ready to give the estimates for the three functionals necessary to solve \eqref{eq:sys2}. With these estimates in hand, it will be possible to make sense of the remainder $R_{\eps,\bigcdot}$ uniformly in $\eps>0$, and to use it to construct the solution $\psi_\eps$. We obtain $R_{\eps,\bigcdot}$ via a fixed point argument in Section~\ref{subseq:Remainder}, while we construct the solution $\psi_\eps$ in Section~\ref{subsec:solution}. Before we state the estimates for the three functionals, the following definitions are necessary.
\begin{definition}\label{def:mcG}
We define an increasing function $\mcG:\R_{\geqslant0}\rightarrow\R_{\geqslant0}$ by 
\begin{equs}\label{eq:defmcG}
    \mcG:t\mapsto  t^{8\Gamma+8+3N_1^{3\Gamma+1}}\Big(\norm{(b,d,g,h)}_{C^{\Gamma+1+3N_1^{3\Gamma+1}}(B_0(t))}\Big)^{2\Gamma+2}\,.
\end{equs}
\end{definition}
Next, we introduce the norm in which we control the argument of the force.
\begin{definition}\label{def:solnorm}
Through out the rest of the article, we fix a small positive parameter $\kappa\leqslant\kappa_0=(\alpha/2)\wedge\big(\delta/(2\Gamma+2)\big)$. For any $\mu,T\in(0,1]$, we define the \textit{solution norm} denoted $\|\bigcdot\|_{{\tt S},\mu,T}$ on smooth functions $\Lambda_{0;T}\rightarrow\R$ by
\begin{equs}
     \|\lambda\|_{{\tt S},\mu,T}\eqdef\norm{\lambda}_{L^\infty_{0;T}}\vee\big( \mu^{1-\alpha+\kappa}\norm{\partial_{\ttx}\lambda}_{L^\infty_{0;T}}\big)
\,.
\end{equs} 
\end{definition}
 \begin{definition}
   For any $\psi\in\mcD\big((-\infty,1]\big)$, the derivative $\D F_\epmu[\psi]\varphi$ of the effective force at $\psi$ in the direction $\varphi\in\mcD\big((-\infty,1]\big)$ can be seen as the operator $\D F_\epmu[\psi]$ acting on $\varphi$. Given $x,y\in\Lambda$, we denote the kernel of this operator (the gradient of $F_\epmu[\psi]$) by
   \begin{equs}
       \D F_\epmu[\psi](x,y)\equiv \D_y F_\epmu[\psi](x)\,.
   \end{equs}
   Moreover, we denote by $\bfX_0 \D F_\epmu[\psi]$ the operator with kernel 
    \begin{equs}
      (x_0-y_0) \D F_\epmu[\psi](x,y)\equiv(x_0-y_0) \D_y F_\epmu[\psi](x)\,.
   \end{equs}
Finally, for $k\in\{0,1\}$, we write $\bfX_0^k \D F_\epmu[\psi]$ to denote $ \D F_\epmu[\psi]$ when $k=0$ and $\bfX_0 \D F_\epmu[\psi]$ when $k=1$.
\end{definition}
\begin{corollary}\label{coro:1}
Recall the definition \eqref{eq:def_delta} of $\delta>0$, and the fact that for $T\in(0,1]$, we define a scale $\mu_T=\sqrt{T}$. Fix $T\in(0,1]$ and a family of functions $\theta=(\theta_\mu)_{\mu\in(0,\mu_T]}$ such that for every $\mu\in(0,\mu_T]$ $\theta_\mu\in\mcD(\Lambda_{0;T})$ and verifying the estimate
\begin{equs}  \label{eq:eqCoro224}  \|\theta_\mu\|_{{\tt S},\mu,T}
\leqslant C_{\vphi,R,\varpi}\,,
\end{equs}
for some positive constant $C_{\vphi,R,\varpi}$, uniformly in $\mu\in(0,\mu_T]$.

Recall that to lighten the notation, we set $K_\mu= K_{N_1^{3\Gamma+1},\mu}$ and $\mcR_\mu=\mcP_\mu^{N_1^{3\Gamma+1}}$. We let
\begin{equs}
    \tF_\epmu[\bigcdot]\eqdef K_\mu F_\epmu[\bigcdot]\,,\;
\bfX_0\widetilde{\D F}_\epmu[\bigcdot]\eqdef K^{\otimes2}_\mu\big(\bfX_0\D F_\epmu[\bigcdot]\big)    
\;{and}\;\tH_\epmu[\bigcdot]\eqdef K_\mu H_\epmu[\bigcdot]\,.
\end{equs}
Then, there exists a universal/combinatorial constant $C_F>0$ such that for all $k,k'\in\{0,1\}$ and $\eta>0$ it holds
    \begin{equs}
     \label{eq:boundF}   \norm{ \tF_\epmu[K_\mu\theta_\mu]}_{L^\infty_{0;T}}&\leqslant C_F\mcG(1\vee C_{\vphi,R,\varpi})\mu^{-2+\alpha-\eta}\,,\\
     \label{eq:boundDF}    \norm{ \bfX_0^k\widetilde{\D F}_\epmu[K_\mu\theta_\mu]}_{\mcL_{0;T}^{\infty,\infty}}&\leqslant C_F\mcG(1\vee C_{\vphi,R,\varpi})\mu^{-2+2k+\alpha-\kappa-\eta}\,,\\
    \label{eq:boundDFharmo}    \norm{ \bfX_0^k\widetilde{\D F}_\epmu[K_\mu\theta_\mu]\mcR_\mu^\dagger(\t^{k'}\dot G_\mu\phi_\eps)}_{L_{0;T}^{\infty}}&\leqslant C_F\mcG(1\vee C_{\vphi,R,\varpi})\mu^{-3+\alpha+2(k+k')-\eta}\,,\\
       \label{eq:boundL}       \norm{\tH_\epmu[K_\mu\theta_\mu]}_{L_{0;T}^\infty}&\leqslant C_F\mcG(1\vee C_{\vphi,R,\varpi})\mu^{-1+\alpha-\eta}\,,
    \end{equs}
    uniformly in $\eps\in(0,1]$ and $\mu\in(0,\mu_T]$. Note that \eqref{eq:boundDFharmo} is an improvement of\eqref{eq:boundDF} by a factor $\kappa$.
   
\end{corollary}
\begin{remark}\label{rem:convergenceFL}
The convergence when $\eps\downarrow0$ of the force coefficients stated in Remark~\ref{rem:convergenceZeta} implies that $(\tF_{\eps,\bigcdot},\tH_{\eps,\bigcdot})$ converge in probability to some functionals $(\tF_{0,\bigcdot},\tH_{0,\bigcdot})$ that verify the estimates \eqref{eq:boundF}, \eqref{eq:boundDF}, \eqref{eq:boundDFharmo} and \eqref{eq:boundL}.
\end{remark}
\begin{proof}[of Corollary~\ref{coro:1}]
    The proofs of all three relations are similar, and rely on a careful analysis of the power counting. We start by proving \eqref{eq:boundF}. Starting from \eqref{eq:expressionFmu} and \eqref{eq:defTripleNorm}, we have
     \begin{equs}
       \norm{\tF_\epmu[K_\mu\theta_\mu]}_{L_{0;T}^\infty}&\lesssim\sum_{a\in\mcM}\nnorm{\zeta^a_\epmu}_{N_1^{3\Gamma+1}}\norm{\big(\mcR_\mu^\dagger\big)^{\otimes[a]}\Upsilon^a[K_\mu\theta_\mu]}_{{L^\infty(\Lambda^{[a]})}}\,,
 \end{equs} 
where $\big(\mcR_\mu^\dagger\big)^{\otimes[a]}\Upsilon^a[\bigcdot]$ denotes the action of $\big(\mcR_\mu^\dagger\big)$ on $\Upsilon^a$ at the level of all its arguments.

Taking the example of $\mcR_\mu^\dagger h^{(i)}(K_\mu\theta_\mu)$, one sees using the Fa\`a di Bruno formula that it holds 
\begin{equs}
    \norm{\mcR_\mu^\dagger h^{(i)}(K_\mu\theta_\mu)}_{L^\infty_{0;T}}\lesssim \norm{\mcR_\mu^\dagger K_\mu \theta_\mu}_{L^\infty_{0;T}}^{3N_1^{3\Gamma+1}}\norm{h}_{C^{i+3N_1^{3\Gamma+1}}(B_0(\norm{K_\mu \theta_\mu}_{L^\infty_{0;T}}))} \lesssim C_{\vphi,R,\varpi}^{3N_1^{3\Gamma+1}}\norm{h}_{C^{i+3N_1^{3\Gamma+1}}(B_0(C_{\vphi,R,\varpi}))}\,,
\end{equs}
where we used \eqref{eq:spacederivKmu} that implies that the kernel of $\mcR_\mu^\dagger K_\mu$ is integrable uniformly in $\mu\in(0,1]$.

Therefore, still using the Fa\`a di Bruno formula, we obtain the generalization of \eqref{eq:boundIa}:
\begin{equs}    \textcolor{white}{A}&\norm{\big(\mcR_\mu^\dagger\big)^{\otimes[a]}\Upsilon^a[K_\mu\theta_\mu]}_{L^\infty(\Lambda^{[a]})}\\&\qquad\lesssim C_{\vphi,R,\varpi}^{4\Gamma+4+3N_1^{3\Gamma+1}}
 \Big(\norm{(b,d,g,h)}_{C^{\Gamma+3N_1^{3\Gamma+1}}(B_0(C_{\vphi,R,\varpi}))}\Big)^{\Gamma+1}
    \mu^{(\alpha-1-\kappa)(\mfs(a^\mfc)+2\mfs(a^\mfe)+\mfs(a^\mff))}\,,
\end{equs}
where we have used the fact that for all multi-indices, $\mfo(a)\leqslant\Gamma$ and $\mfs(a),\mfs(a^\mfc),\mfs(a^\mfe),\mfs(a^\mff)\leqslant\Gamma+1$.\\ 
We can now control the norm of $\zeta^a_\epmu$ using \eqref{eq:boundZeta}. We obtain
    \begin{equs}
       \textcolor{white}{A} &\norm{\tF_\epmu[K_\mu\theta_\mu]}_{L_{0;T}^\infty}\\ &\qquad\lesssim
    C_{\vphi,R,\varpi}^{4\Gamma+4+3N_1^{3\Gamma+1}}  \Big(\norm{(b,d,g,h)}_{C^{\Gamma+3N_1^{3\Gamma+1}}(B_0(C_{\vphi,R,\varpi}))}\Big)^{\Gamma+1}
    \sum_{a\in\mcM}\mu^{|a|+(\alpha-1-\kappa)(\mfs(a^\mfc)+2\mfs(a^\mfe)+\mfs(a^\mff))-\eta}      
    \,.
    \end{equs}
We first bound 
\begin{equs}
    (\alpha-1-\kappa)\big((\mfs(a^\mfc)+2\mfs(a^\mfe)+\mfs(a^\mff)\big)\geqslant -\mfs(a^\mfc)+(2\alpha-2-2\kappa)\mfs(a^\mfe)-\mfs(a^\mff)\,.
\end{equs}
Combining the above with \eqref{eq:scalingGood} yields
\begin{equs}
    |a|&+(\alpha-1-\kappa)\big(\mfs(a^\mfc)+2\mfs(a^\mfe)+\mfs(a^\mff)\big)\\&\geqslant-2+
    \alpha\mfs(a^{\mfh})+(2\alpha-2\kappa)\mfs(a^{\mfe})
    +\mfs(a^\mfc)+\mfs(a^\mfd)+
    2\mfs(a^{\mfb})\\  &\geqslant-2+\alpha\mfn(a)
    +(\alpha-2\kappa)\mfs(a^{\mfe})
    +(1-\alpha)\mfs(a^\mfc)+\mfs(a^\mfd)+
    (2-\alpha)\mfs(a^{\mfb})\,,\label{eq:firstlowerbounda}
\end{equs}
where we have set $\mfn(a)\eqdef\mfs(a^\mfb)+\mfs(a^\mfc)+\mfs(a^\mfe)+\mfs(a^\mfh)$. Because $\mfn(a)$ is larger than the number of leafs of the trees in $a$, we have $\mfn(a)\geqslant1$. Since $\kappa\leqslant\alpha/2$, we end up with 
\begin{equs}
     |a|+(\alpha-1-\kappa)\big(\mfs(a^\mfc)+2\mfs(a^\mfe)+\mfs(a^\mff)\big)\geqslant-2+\alpha\,,
\end{equs}
which concludes the proof of \eqref{eq:boundF}. 

We now turn to \eqref{eq:boundDF}. 
Using \eqref{eq:exprDI}
combined with \eqref{eq:expressionFmu} gives 
\begin{equs}
\text{$ \bfX_0 $} &\widetilde{\D  F}_\epmu[\psi](x,w)\\&= \sum_{b\in\mcM}\sum_{\bfd\in\mcD}C\int_{\Lambda^{[b]}} \langle \big(K_\mu\otimes\Id^{\otimes[b]}\big)\zeta^{(b,\text{\scriptsize{$\mfl^b$}})}_\epmu(x,y^b) \partial_\ttx^{k_0+1-k_1}K_\mu(w-z),\Upsilon^{b+\text{\scriptsize{$\bfd$}}(b)}[\psi](y^{b+\bfd(b)})\rangle_{\mcH^b}\rmd y^{b}\\
&=\sum_{b\in\mcM}\sum_{\bfd\in\mcD}C\int_{\Lambda^{[b]}} \langle \big(K^{\otimes[\hat b]+1}_\mu\otimes\Id\big)\zeta^{(b,\text{\scriptsize{$\mfl^b$}})}_\epmu(x,y^{\hat b},z) \partial_\ttx^{k_0+1-k_1}K_\mu(w-z),\big(\mcR_\mu^\dagger\big)^{\otimes[\hat b]}\Upsilon^{b+\text{\scriptsize{$\bfd$}}(b)}[\psi](y^{b+\bfd(b)})\rangle_{\mcH^b}\rmd y^{b}\,.
\end{equs}
Here to shorten the notation we have replaced the prefactor in \eqref{eq:exprDI} by $C$, we have set $z\eqdef y^b_{\mfk_0 k_0b^{\mfk_0}_{k_0}}$ and we let $y^{\hat b}$ denote the collection of all the other variables in $y^b$. 

Using the same argument as in the proof of \eqref{eq:propB}, we will place the last kernel $K_\mu$ in the variable $z$ in front of $\zeta_\epmu^{(b,\mfl^b)}(x,y^{\hat b},z)$ by integrating by parts. The derivatives inside this newly created $\mcR_\mu^\dagger$ can hit either $\Upsilon^{b+\text{\scriptsize{$\bfd$}}(b)}[\psi]$ or $ \partial_\ttx^{k_0+1-k_1}K_\mu(y-z)$. In the second case, this creates a kernel whose $\mcL^{\infty,\infty}$ norm is still bounded by $\mu^{-(k_0+1-k_1)}$. Overall, we end up with the bound
\begin{equs}\label{eq:somestepinproofDF}
       \|\bfX_0 &\widetilde{\D F}_\epmu[K_\mu\theta_\mu]\|_{\mcL_{0;T}^{\infty,\infty}}\\&\lesssim\sum_{b\in\mcM}\sum_{\bfd\in\mcD}\nnorm{\zeta^{(b,\mfl^b)}_\epmu}_{N_1^{3\Gamma+1}}\norm{\big(\mcR_\mu^\dagger\big)^{\otimes[b]}\Upsilon^{b+d(\bfd)}[K_\mu\theta_\mu]}_{L^\infty(\Lambda^{[b+d(\bfd)]})}\mu^{-(k_0+1-k_1)}  \,.        \end{equs}
Reasoning as before, and using the fact that going from $b$ to $b+d(\bfd)$ increases $\mfl(b)$ by at most one while leaving $\mfs(b)$ unmodified, again we can bound $\mfl(b)$ by $\Gamma $ and $\mfs(b)$ by $\Gamma+1$, and we end up with  
\begin{equs}    \textcolor{white}{A}&\norm{\big(\mcR_\mu^\dagger\big)^{\otimes[b]}\Upsilon^{b+d(\bfd)}[K_\mu\theta_\mu]}_{L^\infty(\Lambda^{[b+d(\bfd)]})}\\&\qquad\lesssim C_{\vphi,R,\varpi}^{4\Gamma+4+3N_1^{3\Gamma+1}}
 \Big(\norm{(b,d,g,h)}_{C^{\Gamma+1+3N_1^{3\Gamma+1}}(B_0(C_{\vphi,R,\varpi}))}\Big)^{\Gamma+1}
    \mu^{(\alpha-1-\kappa)(\mfs((b+d(\bfd))^\mfc)+2\mfs((b+d(\bfd))^\mfe)+\mfs((b+d(\bfd))^\mff))}\,.
\end{equs}
Moreover, the definition of $d(\bfd)$ implies that we have
 \begin{equs}
  \mfs((b+d(\bfd))^\mfc)+   2\mfs((b+d(\bfd))^\mfe)+\mfs((b+d(\bfd))^\mff)=
      \mfs(b^\mfc)+2\mfs(b^\mfe)+\mfs(b^\mff)-(k_0+1-k_1)\,.
 \end{equs}       
Gathering all these estimates and using \eqref{eq:heat1} and yields      
\begin{equs}
 \norm{\bfX_0 \widetilde{\D F}_\epmu[K_\mu\theta_\mu]}_{\mcL_{0;T}^{\infty,\infty}}&\lesssim C_{\vphi,R,\varpi}^{4\Gamma+4+3N_1^{3\Gamma+1}}
        \Big(\norm{(b,d,g,h)}_{C^{\Gamma+1+3N_1^{3\Gamma+1}}(B_0(C_{\vphi,R,\varpi}))}\Big)^{\Gamma+1}\\
        &\textcolor{white}{\lesssim}\times\sum_{b\in\mcM}\sum_{\bfd\in\mcD}\mu^{|b|+2k+(\alpha-1-\kappa)(\mfs(b^\mfc)+
         2\mfs(b^\mfe)+\mfs(b^\mff)-(k_0+1-k_1)
        )-(k_0+1-k_1)-\eta}    
       \,.
 \end{equs}
Here, not that in estimating the norm of $\zeta_\epmu^{a,\mfl^a}$, we used the fact that $|\mfl^b|=2k$.

Using \eqref{eq:firstlowerbounda}
 \begin{equs}
 |b| &+(\alpha-1-\kappa)\big(\mfs(b^\mfc)+
         2\mfs(b^\mfe)+\mfs(b^\mff)-(k_0+1-k_1)
        \big)-(k_0+1-k_1)\\
      &=   |b| +(\alpha-1-\kappa)\big(\mfs(b^\mfc)+
         2\mfs(b^\mfe)+\mfs(b^\mff)\big)-(\alpha-\kappa)(k_0+1-k_1)
             \\
  &\geqslant-2+\alpha\mfn(b)
    +(\alpha-2\kappa)\mfs(b^{\mfe})
    +(1-\alpha)\mfs(b^\mfc)+\mfs(b^\mfd)+
    (2-\alpha)\mfs(b^{\mfb})  -(\alpha-\kappa)(k_0+1-k_1)
     \,.
 \end{equs}
If $\mfn(b)\geqslant2$, then we can directly conclude that the above quantity is greater than $-2+\alpha$. The only cases where $\mfn(b)=1$ are $b\in\{ \1^\mfb_0,\1^\mfc_0,\1^\mfe_0,\1^\mfh_0\}$. The cases $b=\1^\mfb_0,\1^\mfh_0$ follow using the fact that if $b\in\{\1^\mfb_0,\1^\mfh_0\}$, then one necessarily has $k_0+1-k_1=0$. The case $b=\1^\mfc_0$ is simpler due to the good factor $(1-\alpha)\mfs(b^\mfc)$. The case $b=\1^\mfe_0$ is more problematic because for $b=\1^\mfe_0$ it can happen that $k_0+1-k_1=1$, and we only obtain the bound
\begin{equs}\label{eq:badboundb}
    |b| &+(\alpha-1-\kappa)\big(\mfs(b^\mfc)+
         2\mfs(b^\mfe)+\mfs(b^\mff)-(k_0+1-k_1)
        \big)-(k_0+1-k_1)\geqslant -2+\alpha-\kappa\,.\textcolor{white}{badbdd}
\end{equs}
This concludes the proof of \eqref{eq:boundDF}.

We now discuss the necessary modifications in order to upgrade \eqref{eq:boundDF} to \eqref{eq:boundDFharmo}. Following the same steps as before, we obtain the improvement of \eqref{eq:somestepinproofDF}
\begin{equs}
     \|\bfX^k_0 \widetilde{\D F}_\epmu[K_\mu\theta_\mu]&\mcR_\mu^\dagger(\t^{1-k}\dot G_\mu\phi_\eps)\|_{L_{0;T}^{\infty}}\\&\lesssim\sum_{b\in\mcM}\sum_{\bfd\in\mcD}\nnorm{\zeta^{(b,\mfl^b)}_\epmu}_{N_1^{3\Gamma+1}}\norm{\big(\mcR_\mu^\dagger\big)^{\otimes[b]}\Upsilon^{b+d(\bfd)}[K_\mu\theta_\mu]}_{L^\infty(\Lambda^{[b+d(\bfd)]})}\|\partial_\ttx^{k_0+1-k_1}\mcR_\mu^\dagger(\t^{k'}\dot G_\mu\phi_\eps)\|_{L^\infty_{0;T}}\\
     &\lesssim     \sum_{b\in\mcM}\sum_{\bfd\in\mcD}\nnorm{\zeta^{(b,\mfl^b)}_\epmu}_{N_1^{3\Gamma+1}}\norm{\big(\mcR_\mu^\dagger\big)^{\otimes[b]}\Upsilon^{b+d(\bfd)}[K_\mu\theta_\mu]}_{L^\infty(\Lambda^{[b+d(\bfd)]})}
     \mu^{(-(k_0+1-k_1)+\alpha-\kappa)\wedge0-1+2k'}
     \,,     \end{equs}
where in going from the first to the second inequality we used \eqref{eq:harmoniccompletion} and \eqref{eq:harmoniccompletion3}. The interest of this better estimates is that when $k_0+1-k_1=1$, then one gains a factor $\mu^{\alpha-\kappa}$. In particular, this is the case when $b=\1_0^\mfe$ and $k_0+1-k_1=1$, in which case the gain of the factor $\alpha-\kappa$ allows us to get rid of the bad factor $\kappa$ in \eqref{eq:badboundb} (because $\alpha-2\kappa\geqslant0$), whence the improvement \eqref{eq:boundDFharmo}.

 It remains to show the bound \eqref{eq:boundL}. In view of \eqref{eq:expL}, we have the bound
 \begin{equs}
    \norm{\tH_\epmu[K_\mu\theta_\mu]}_{L_{0;T}^\infty}&\lesssim\sum_{a\in\mcM_{>\Gamma}}  \sum_{\substack{(\sigma,b,c,\bfd)\in\text{\scriptsize{$\mathrm{Ind}(a)$}}}}  \norm{K^{\otimes[a]+1}_\mu\B_\mu(\zeta^b_\epmu,\zeta^c_\epmu)}_{L_{0;T}^\infty}
   \norm{\big(\mcR_\mu^\dagger\big)^{\otimes[a]}\Upsilon^a[K_\mu\theta_\mu]}_{{L^\infty(\Lambda^{[a]})}}\,, 
 \end{equs}
Then, \eqref{eq:defB} implies that
 \begin{equs}
    K^{\otimes[a]+1}_\mu\B_\mu(\zeta^b_\epmu,\zeta^c_\epmu)(x,y^a)= CK^{\otimes[b]+1}_\mu\zeta^b_\epmu(x,y^b)\int_\Lambda\mcR_\mu\partial_{\ttx}^{k_0+1-k_1}\dot G_\mu( z-w)K^{\otimes[c]+1}_\mu\zeta^c_\epmu(w,y^c)\rmd w\,,
\end{equs}
for $C$ an inessential constant, so that using \eqref{eq:propB} and \eqref{eq:boundZeta} we have 
\begin{equs}    \norm{K^{\otimes[a]+1}_\mu\B_\mu(\zeta^b_\epmu,\zeta^c_\epmu)}_{L_{0;T}^\infty}\lesssim\nnorm{\zeta^b_\epmu}_{N_1^{3\Gamma+1}}\norm{\mcR^2_\mu\partial_{\ttx}^{k_0+1-k_1}\dot G_\mu}_{\mcL^{\infty,\infty}}\nnorm{\zeta^c_\epmu}_{N_1^{3\Gamma+1}}\lesssim\mu^{|b|+|c|+1-(k_0+1-k_1)-\eta}\,.
\end{equs}
Moreover, proceeding as before, we have
\begin{equs}    \textcolor{white}{A}&\norm{\big(\mcR_\mu^\dagger\big)^{\otimes[a]}\Upsilon^a[K_\mu\theta_\mu]}_{{L^\infty(\Lambda^{[a]})}}\\&\qquad\lesssim C_{\vphi,R,\varpi}^{8\Gamma+8+3N_1^{3\Gamma+1}}
 \Big(\norm{(b,d,g,h)}_{C^{\Gamma+1+3N_1^{3\Gamma+1}}(B_0(C_{\vphi,R,\varpi}))}\Big)^{2\Gamma+2}
    \mu^{(\alpha-1-\kappa)(\mfs(a^\mfc)+2\mfs(a^\mfe)+\mfs(a^\mff))}\,,
\end{equs}
where this time we have bounded the order of $a\in\mcM_*$ by $2\Gamma+1$ (and thus its size by $2\Gamma+2$) and used the fact that $\mfl(a)\leqslant\mfl(b+d(\bfd))\vee\mfl(c)$. Since as before $\mfl(b+d(\bfd))\leqslant\mfl(b)+1\leqslant\Gamma+1$ and $\mfl(c)\leqslant\Gamma$, we have $\mfl(a)\leqslant\Gamma+1$. We thus end up with 
 \begin{equs}
    \norm{\tH_\epmu[K_\mu\theta_\mu]}_{L_{0;T}^\infty}&\lesssim\sum_{a\in\mcM_{>\Gamma}} \sum_{\substack{(\sigma,b,c,\bfd)\in\text{\scriptsize{$\mathrm{Ind}(a)$}}}} C_{\vphi,R,\varpi}^{8\Gamma+8+3N_1^{3\Gamma+1}}
    \Big(\norm{(b,d,g,h)}_{C^{\Gamma+1+3N_1^{3\Gamma+1}}(B_0(C_{\vphi,R,\varpi}))}\Big)^{2\Gamma+2}
    \\&\textcolor{white}{\lesssim}
    \mu^{1+|b|+|c|+(\alpha-1-\kappa)(\mfs(a^\mfc)+2\mfs(a^\mfe)+\mfs(a^\mff))-(k_0+1-k_1)-\eta}
    \,.
 \end{equs}
Moreover, observe that by \eqref{eq:lastboundsizea} we have
 \begin{equs}
    1&+ |b|+|c|+(\alpha-1-\kappa)\big(\mfs(a^\mfc)+2\mfs(a^\mfe)+\mfs(a^\mff)\big)-(k_0+1-k_1)\\&=|a|-1+(\alpha-1-\kappa)\big(\mfs(a^\mfc)+2\mfs(a^\mfe)+\mfs(a^\mff)\big)\\     
         &\geqslant-1+\alpha+\delta+(\alpha-2\kappa)  \mfs(a^\mfe)-\kappa\mfs(a^{\mff})\,.
 \end{equs}
 Here, we do not seek to optimise and bound $\mfs(a^\mff)\leqslant\mfs(a)=\mfo(a)+1\leqslant2\Gamma+2$, which gives
 \begin{equs}
    1&+ |b|+|c|+(\alpha-1-\kappa)\big(\mfs(a^\mfc)+2\mfs(a^\mfe)+\mfs(a^\mff)\big)-(k_0+1-k_1)\\  
         &\geqslant-1+\alpha+\big(\delta-(2\Gamma+2)\kappa\big)+(\alpha-2\kappa)  \mfs(a^\mfe)\,,
 \end{equs}
so that enforcing $\kappa\leqslant(\alpha/2)\wedge\big(\delta/(2\Gamma+2)\big)$ yields the desired result.
 \end{proof}

\section{Deterministic analysis}\label{Sec:Sec3}
In Section~\ref{sec:flow_setup}, we reformulated \eqref{eq:eqDef} on $\Lambda_{0;1}$ as
\begin{equs}\label{eq:eqFinal}    \psi_\eps(x)=G\big( F_\eps[\psi_\eps]+\phi_\eps\big)(x)\,.
\end{equs}
In order to construct the solution $\psi_\eps$, we need to rewrite it, taking the flow approach into account.
\begin{lemma}
    Fix $\eps\in(0,1]$, $T\in(0,1]$ and recall that $\mu_T=\sqrt{T}$. If $\psi_\eps$ solves \eqref{eq:eqDef1} on $\Lambda_{0;T}$, then for all $x\in\Lambda_{0;T}$ we have
    \begin{equs}\label{eq:Phi2}
        \psi_\eps(x)=G\phi_\eps+(G-G_{\mu_T})\big(F_{\eps,\mu_T}[0]+R_{\eps,\mu_T}\big)(x)\,.
    \end{equs}
\end{lemma}
\begin{proof} 
Using \eqref{eq:FmuRmu} for $\mu=\mu_T$ and the support properties of $G_{\mu_T}$ (see \eqref{eq:suppGmu}), we reexpress \eqref{eq:eqFinal} as
\begin{equs}    \psi_\eps(x)=G\phi_\eps+(G-G_{\mu_T})\big(F_{\eps,\mu_T}[\psi_{\eps,\mu_T}]+R_{\eps,\mu_T}\big)(x)\,,
\end{equs}
where $\psi_{\eps,\mu_T}$ is defined by \eqref{eq:psimu} with $\mu=\mu_T$. The support properties of $G_{\mu_T}$ then imply that $\psi_{\eps,\mu_T}$ is supported outside $\Lambda_{0;T}$, which concludes the proof.
\end{proof}
We can conclude from \eqref{eq:Phi2} shows that, if $R_\epmu$ can be constructed thank to \eqref{eq:eqRwelldefined} up to some scale $\mu_T$, then $\psi_\eps$ has local solutions on $\Lambda_{0;T}$. In Section~\ref{subseq:Remainder}, we show that for $T$ small enough, it is indeed possible to construct $R_{\eps,\mu_T}$, while in Section \ref{subsec:solution}, we construct the solution $\psi\eqdef\lim_{\eps\downarrow0}\psi_\eps$ in a suitable Besov space.

\subsection{Construction of the remainder}\label{subseq:Remainder} 
We aim to show that the system~\eqref{eq:sys2} allows one to construct the remainder $R_\epmu$ close to $\mu=0$. However, for technical convenience, following \cite{Duch23}, we rather solve a slightly different system. Recall the shorthand notation
\begin{equs}
    \tF_\epmu[\bigcdot]= K_\mu F_\epmu[\bigcdot]\,,\;\text{and}\;\tH_\epmu[\bigcdot]= K_\mu H_\epmu[\bigcdot]\,.
\end{equs}
Making the change of variables 
\begin{equs} \vphi_\epmu \eqdef \mcR_\mu\psi_\epmu-\mcR_\mu G_\mu\phi_\eps \Leftrightarrow \psi_\epmu=G_\mu \phi_\eps+K_\mu\vphi_\epmu \,,
\end{equs}
\eqref{eq:sys2} rewrites as
\begin{subequations}\label{eq:sys3}
  \begin{empheq}[left=\empheqlbrace]{alignat=2}  
\label{eq:sys31} R_{\eps,\mu}&=-\int_0^\mu \big(\D F_{\eps,\nu}[K_\nu\theta_\epnu]\dot G_\nu (R_\epnu+\phi_\eps)+H_\epnu[K_\nu\theta_\epnu]\big)\rmd\nu\,,\\
        \vphi_\epmu&= -\int_\mu^{\mu_T}\tilde K_{\nu,\mu}\big(\mcR_\nu^2\dot G_\nu \tilde F_\epnu[K_\nu\theta_\epnu]+\mcR_\nu\mcR_{+,\nu} \dot G_\nu K_{+,\nu}R_\epnu\big)\rmd\nu\,,
  \end{empheq}
\end{subequations}
where to lighten the notation we have set
\begin{equs}\label{eq:field_change_of_var}
\theta_\epmu\eqdef\vphi_\epmu+\mcR_\mu G_\mu\phi_\eps=\mcR_\mu\psi_\epmu    \,,
\end{equs}
and used the fact that $R_{\eps,0}=0$. Here, for $\lambda\geqslant\tau$, we write $\tilde K_{\lambda,\tau}=\mcR_\tau K_\lambda$. Note that by \eqref{eq:Kmunu}, for $\lambda\geqslant\tau$, $\tilde K_{\lambda,\tau}$ is a bounded operator $L^\infty\rightarrow L^\infty$. 

Our aim is to solves the system \eqref{eq:sys3} in the following topology.
\begin{definition}\label{def:Rnorm}
Recall that $\kappa\leqslant(\alpha/2)\wedge\big(\delta/(2\Gamma+2)\big)$ was fixed in Definition~\ref{def:solnorm}. For any $\mu,T\in(0,1]$, we define the \textit{remainder norm} denoted $\|\bigcdot\|_{{\tt R},\mu,T}$ on smooth functions $\Lambda_{0;T}\rightarrow\R$ by
    \begin{equs}
        \|\lambda\|_{{\tt R},\mu,T}\eqdef \mu^{-\alpha+\kappa/2}\Big(
\|K_{\mu} \lambda\|_{L^{\infty,1}_{0;T}}\vee\|K_{\mu}(\t  \lambda)\|_{L^\infty_{0;T}}\Big)\,.
    \end{equs}
    Moreover, for two families $\theta=(\theta_\mu)_{\mu\in(0,1]}$ and $\lambda=(\lambda_\mu)_{\mu\in(0,1]}$ such that $\theta_\mu$, $\lambda_\mu$ are smooth functions $\Lambda_{0;T}\rightarrow\R$, we set 
\begin{equs}    \label{eq:defbignorm}\nnorm{\theta,\lambda}_T\eqdef 
    \sup_{\mu\in(0,\mu_T]}\|\theta_\mu\|_{{\tt S},\mu,T}\vee\|\lambda_\mu\|_{\ttR,\mu,T}\,,
\end{equs}
where the solution norm $\|\bigcdot\|_{{\tt S},\mu,T}$ was introduced in Definition~\ref{def:solnorm}.
\end{definition}
This choice of the remainder norm is motivated by the following lemma, which states it controls some quantities which appear when performing the fixed point argument leading to the construction of the remainder.
\begin{lemma}\label{lem:RtoLinfty}
For every $k\in\{0,1\}$, it holds 
    \begin{equs}        \label{eq:RtoLinfty}\|K_{+,\mu}\lambda\|_{L^\infty_{0;T}}&\lesssim \mu^{-2+\alpha-\kappa/2}\|\lambda\|_{{\tt R},\mu,T} \\
    \label{eq:RtoXGdot}        \|\mcR^\dagger_\mu(\bfX^k_0\dot G_\mu) \lambda\|_{L^\infty_{0;T}}&\lesssim \mu^{-1+\alpha+2k-\kappa/2}\|\lambda\|_{{\tt R},\mu,T} \\\label{eq:RtoGdotT}
          \|\mcR^\dagger_\mu\dot G_\mu(\t \lambda)\|_{L^\infty_{0;T}}&\lesssim \mu^{1+\alpha-\kappa/2}\|\lambda\|_{{\tt R},\mu,T} 
    \end{equs}
    uniformly in smooth $\lambda$ and $\mu\in(0,1]$.
\end{lemma}
\begin{proof}
\eqref{eq:RtoLinfty} is a direct application of the Sobolev embedding type estimate \eqref{eq:KmuLp}:
\begin{equs}
     \|K_{+,\mu} \lambda\|_{L^\infty_{0;T}}\lesssim\mu^{-2}\|K_{\mu}\lambda\|_{L^{\infty,1}_{0;T}}\leqslant\mu^{-2+\alpha-\kappa/2}\|\lambda\|_{\ttR,\mu,T}\,.
\end{equs} 
\eqref{eq:RtoXGdot} is an immediate consequence of \eqref{eq:RtoLinfty}. Indeed, using \eqref{eq:heat1}, we have
    \begin{equs}
         \|\mcR^\dagger_\mu(\bfX^k_0\dot G_\mu) \lambda\|_{L^\infty_{0;T}}\lesssim \|\mcR^\dagger_\mu\mcR_{+,\mu}(\bfX^k_0\dot G_\mu)\|_{\mcL^{\infty,\infty}}
         \|K_{+,\mu}\lambda\|_{L^\infty_{0,T}}\lesssim \mu^{-1+\alpha+2k-\kappa/2}\|\lambda\|_{\ttR,\mu,T}\,.
    \end{equs}
\eqref{eq:RtoGdotT} also follows from \eqref{eq:heat1}:
\begin{equs}
    \|\mcR^\dagger_\mu\dot G_\mu(\t \lambda)\|_{L^\infty_{0;T}}\lesssim
     \|\mcR^\dagger_\mu\mcR_{\mu}\dot G_\mu\|_{\mcL^\infty_{0;T}}\|K_{\mu}(\t  \lambda)\|_{L^\infty_{0;T}}\lesssim\mu^{1+\alpha-\kappa/2}\|\lambda\|_{\ttR,\mu,T}\,.
\end{equs}
\end{proof}
With the norm $\nnorm{\bigcdot}_T$ is hand, we are ready to state the fixed point argument leading to the construction of the remainder. 
\begin{proposition}\label{eq:prop1} 
    For fixed $\eps\in(0,1]$ and $C_{\vphi,R}>0$, there exists a random $T\in(0,1]$ such that the map    
    \begin{equs}
    \Phi:\binom{\vphi_{\eps,\bigcdot}}{R_{\eps,\bigcdot}}\mapsto\binom{\Phi^\vphi_{\eps,\bigcdot}}{\Phi^{R}_{\eps,\bigcdot}}(\vphi_{\eps,\bigcdot},R_{\eps,\bigcdot})\eqdef\binom{-\int_{\bigcdot}^{\mu_T}\tilde K_{\nu,\bigcdot}\big(\mcR^2_{\nu} \dot G_{\nu}  \tF_{\eps,\nu}[K_{\nu}\theta_{\eps,\nu}]+\mcR_\nu\mcR_{+,\nu}\dot G_\nu K_{+,\nu}R_{\eps,\nu}\big)\rmd\nu}{-\int_0^{\bigcdot}\big(
      \D F_{\eps,\nu}[K_{\nu}\theta_{\eps,\nu}]\dot G_\nu (R_\epnu+\phi_\eps)+H_\epnu[K_\nu\theta_\epnu]\big)
      \rmd\nu}    
     \end{equs}
    is a contraction for the norm $\nnorm{\bigcdot}_T$ given by \eqref{eq:defbignorm} on the ball of radius $C_{\vphi,R}$. In particular, the system \eqref{eq:sys2} has a unique solution that we will denote by $\big( (\theta_{\eps,\mu},R_{\eps,\mu} ) : \mu \in (0,\mu_T] \big)$ where we recall that $\theta_{\eps,\mu}$ and $\varphi_{\eps,\mu}$ are related via \eqref{eq:field_change_of_var}. 
 Moreover, this solution is continuous in the data $\big( ( \tF_{\eps,\mu},\tH_{\eps,\mu} ) : \mu \in (0,\mu_T] \big)$.    
\end{proposition}
\begin{remark}\label{rem:convergenceR}
 In the sequel, we will mostly be interested in the remainder field $\big( R_{\eps,\mu} : \mu \in (0,\mu_T] \big)$ component of the solution promised by Proposition~\ref{eq:prop1}, and we note that the statement of Proposition~\ref{eq:prop1} implies that it satisfies
    \begin{equs}\label{eq:boundR}
        \sup_{\mu\in(0,\mu_T]}\mu^{2-\alpha+\kappa/2}\norm{K_{+,\mu} R_\epmu}_{L^\infty_{0;T}}\leqslant C_R\,,
    \end{equs}
where $C_R>0$ is a universal constant, the radius of the ball where the above map above is a contraction. Moreover, the convergence in probability of the data $(\tF_{\eps,\bigcdot},\tH_{\eps,\bigcdot})$ when $\eps\downarrow0$ (see Remark~\ref{rem:convergenceFL}) implies that there exists $R_{0,\bigcdot}$ verifying \eqref{eq:boundR} and such that it holds 
 \begin{equs}
      \lim_{\eps\downarrow0}  \sup_{\mu\in(0,\mu_T]}\mu^{2-\alpha+\kappa}\norm{K_{+,\mu}\big(R_\epmu-R_{0,\mu}\big)}_{L^\infty_{0;T}}=0\,.
    \end{equs} 
    \end{remark}
We now introduce some time localization to help with the proof of Proposition~\ref{eq:prop1}.
\begin{definition}\label{def:weightvw}
We define two families of time weights as follows. Fix two smooth function $v,w:\R\rightarrow[0,1]$ such that the following holds:
\begin{equs}    \text{supp}\,v\subset[2,\infty)\,,\;v\upharpoonright{[3,\infty)}=1\,,
\quad
\text{and}\quad    \text{supp}\,w\subset[-3,3]\,,\;w\upharpoonright{[-2,2]}=1\,,
\end{equs}
and $v$, $w$ verify
\begin{equs}    
(v+w)|_{[-2,\infty)}=1\,.
\end{equs} 
We then define two collections of time weights $v=(v_\mu)_{\mu\in(0,1]}$, $w=(w_\mu)_{\mu\in(0,1]}$ by
\begin{equs}
    v_\mu(t)\eqdef v(t/\mu^2)\,,\;\text{and}\;w_\mu(t)\eqdef w(t/\mu^2)\,.
\end{equs}
These weights are defined in such a way that we have $v_\mu\in\mcW^\infty_{N,\mu}$ and $w_\mu\in\mcW^1_{N,\mu}\cap\mcW^\infty_{N,\mu}$ for all $N\in\N$ and that it holds 
\begin{equs}    \text{supp}\,v_\mu\subset[2\mu^2,\infty)\,,\;\text{and}\;\text{supp}\,w_\mu\subset[-3\mu^2,3\mu^2]\,,
\end{equs}
as well as $v_\mu(t)+w_\mu(t)=1$ if $t\geqslant0$.
\end{definition}
\begin{proof}[of Proposition~\ref{eq:prop1}]
We show that $\Phi$ maps a ball of radius $C_{\vphi,R}$ into itself. The proofs that it is a contraction, and that it is continuous in the data are totally similar, and can be carried out with very close estimates.  
    
 Suppose that $\nnorm{\vphi_{\eps,\bigcdot},R_{\eps,\bigcdot}}_T\leqslant C_{\vphi,R}$, we aim to show that $\nnorm{\Phi(\vphi_{\eps,\bigcdot},R_{\eps,\bigcdot})}_T\leqslant C_{\vphi,R}$. To do so, we need to make the following preliminary observation: the hypothesis that $\nnorm{\vphi_{\eps,\bigcdot},R_{\eps,\bigcdot}}_T\leqslant C_{\vphi,R}$ implies that $\theta_\epmu=\vphi_\epmu+\mcR_\mu G_\mu\phi_\eps$ satisfies 
 \begin{equs}
     \|\theta_\epmu\|_{{\tt S},\mu,T}\leqslant C_{\vphi,R,\varpi}\,.
 \end{equs}
Indeed, 
 \begin{equs}
      \|\theta_\epmu\|_{{\tt S},\mu,T}\leqslant  \|\vphi_\epmu\|_{{\tt S},\mu,T}+\|  \mcR_\mu G_\mu\phi_\eps \|_{{\tt S},\mu,T}\,,
 \end{equs}
 and the second norm is controlled in \eqref{eq:harmoniccompletion2} and \eqref{eq:harmoniccompletionwithderivative}. 
 
$\theta_\epmu$ thus verifies the hypothesis of Corollary~\ref{coro:1} with $\theta_\mu=\theta_\epmu$, so that we can make use of \eqref{eq:boundF}, \eqref{eq:boundDF}, \eqref{eq:boundDFharmo} and \eqref{eq:boundL}.

We first study the $\vphi$ component. Fix $\mfl\in\N^n$ such that $|\mfl|\in\{0,1\}$. We have
\begin{equs}    \norm{\partial^\mfl_{\ttx}\Phi_\epmu^\vphi}_{L^\infty_{0;T}}&\lesssim\int_\mu^{\mu_T}\Big(
\norm{\mcR_\nu^2\partial^\mfl_{\ttx}\dot G_\nu\tF_\epnu[K_\nu\theta_\epnu]}_{L_{0;T}^\infty}+\norm{\mcR_\nu\mcR_{+,\nu}\partial^\mfl_{\ttx}\dot G_\nu K_{+,\nu} R_\epnu}_{L_{0;T}^\infty}\Big)
\rmd\nu\\&\lesssim
\int_\mu^{\mu_T}\Big(
\norm{\mcR^2_\nu\partial^\mfl_{\ttx}\dot G_\nu}_{\mcL_{0;T}^{\infty,\infty}}\norm{\tF_\epnu[K_\nu\theta_\epnu]}_{L_{0;T}^\infty}+\norm{\mcR_\mu\mcR_{+,\nu}\partial^\mfl_{\ttx}\dot G_\nu}_{\mcL_{0;T}^{\infty,\infty}}\norm{K_{+,\nu}R_\epnu}_{L_{0;T}^\infty}\Big)
\rmd\nu\,.
\end{equs}
By \eqref{eq:heat1}, the operator norms are bounded by $\nu^{1-|\mfl|}$, and the $L_{0;T}^\infty$ norm of $\tF_\epnu[K_\mu\theta_\epnu]$ can be controlled by $\nu^{-2+\alpha-\eta}$ using \eqref{eq:boundF}. Moreover, by \eqref{eq:RtoLinfty}, we have
\begin{equs}    \norm{K_{+,\nu}R_\epnu}_{L_{0;T}^\infty}\lesssim\nu^{-2+\alpha-\kappa/2}\|R_\epnu\|_{\ttR,\nu,T}\lesssim\nu^{-2+\alpha-\kappa/2}C_{\vphi,R}\,.
\end{equs}
Gathering all these estimates yields, taking $\eta=\kappa/2$,
\begin{equs}    \norm{\partial^\mfl_{\ttx}\Phi_\epmu^\theta}_{L_{0;T}^\infty}&\lesssim\int_\mu^{\mu_T}\Big(\nu^{-1+\alpha-|\mfl|-\eta}+\nu^{-1+\alpha-|\mfl|-\kappa/2}\Big)\rmd\nu\lesssim
\int_\mu^{\mu_T}\nu^{-1+\alpha-|\mfl|-\kappa/2}\rmd\nu\,.
\end{equs}
If $|\mfl|=0$ the argument of the integral is integrable at $\nu=0$ and we have
\begin{equs}
   \norm{\partial^\mfl_{\ttx}\Phi_\epmu^\theta}_{L_{0;T}^\infty}&\lesssim\int_\mu^{\mu_T}\nu^{-1+\alpha-\kappa/2}\rmd\nu\lesssim
     \int_0^{\mu_T}\nu^{-1+\alpha-\kappa/2}\rmd\nu\lesssim\mu_T^{\alpha-\kappa/2}\,,
\end{equs}
while if $|\mfl|=1$ we use
\begin{equs}
    \int_\mu^{\mu_T}\nu^{-2+\alpha-\eta}\rmd\nu\lesssim
    \mu_T^{\kappa/2} \int_\mu^{1}\nu^{-2+\alpha-\kappa}\rmd\nu\lesssim\mu^{-1+\alpha-\kappa}\mu_T^{\kappa/2}\,.
\end{equs}
Taking $T$ (and thus $\mu_T$) small enough, we can bring the implicit constant in the above inequalities to the desired value $C_{\theta,R}$.

We now deal with the remainder. Note that we need to control it in two topologies, since we need to estimate $\|K_\mu R_\epmu\|_{L^{\infty,1}_{0;T}}$ but also $\|K_{\mu}(\t v_\mu R_\epmu)\|_{L^{\infty}_{0;T}}$. We start by studying the first norm. \eqref{eq:sys31} combined with \eqref{eq:Kmunu} yields
\begin{equs}
\|K_{\mu} R_\epmu\|_{L^{\infty,1}_{0;T}}&\lesssim \int_0^\mu \Big(\|    K_\nu\big( \D F_\epnu[K_\nu\theta_\epnu] \dot G_\nu (R_\epnu+\phi_\eps)\big)\|_{L^{\infty,1}_{0;T}}
+\|\tH_\epnu[K_\nu\theta_{\epnu}]\|_{L^{\infty,1}_{0;T}}\Big)\rmd\nu\,.
\end{equs}
The second term is dealt with in Lemma~\ref{lem:fixedpointRH} below, see \eqref{eq:RHSeqR1H}. Moreover, inserting $1=v_\nu+w_\nu$ in front of $\D F_\epnu$, we split the first term as
\begin{equs}
    \|    K_\nu\big( &\D F_\epnu[K_\nu\theta_\epnu] \dot G_\nu (R_\epnu+\phi_\eps)\big)\|_{L^{\infty,1}_{0;T}}\\&\lesssim
     \|    K_\nu\big(v_\nu \D F_\epnu[K_\nu\theta_\epnu] \dot G_\nu (R_\epnu+\phi_\eps)\big)\|_{L^{\infty,1}_{0;T}}+ \|    K_\nu\big( w_\nu\D F_\epnu[K_\nu\theta_\epnu] \dot G_\nu (R_\epnu+\phi_\eps)\big)\|_{L^{\infty,1}_{0;T}}\,.
\end{equs}
These term are handled separately in Lemmas~\ref{lem:fixedpointRw} and \ref{lem:fixedpointR} below, see \eqref{eq:RHSeqR1w} and \eqref{eq:RHSeqR1}. 

Putting all the estimates together and taking $\eta=\kappa/4$, we have
\begin{equs}
    \|K_{\mu} R_\epmu\|_{L^{\infty,1}_{0;T}}&\lesssim \int_0^\mu\nu^{-1+\alpha-\kappa/4}\rmd\nu\lesssim\mu^{\alpha-\kappa/2}\mu_T^{\kappa/4}\,.
\end{equs}
This is the desired estimate, since taking $T$ small enough, we can bring the implicit constant back to the value $C_{\vphi,R}$.

It remains to control the norm $\|K_{\mu}(\t  R_\epmu)\|_{L^{\infty}_{0;T}}$. To do so, observe that \eqref{eq:sys31} implies that
\begin{equs}
    K_{\mu}(\t  R_\epmu)=-\int_0^\mu  \tilde K_{\mu,\nu} \Big(K_\nu\big(\t \D F_\epnu [K_\nu\theta_\nu]\dot G_\nu(R_\epnu+\phi_\eps)\big)+K_\nu\big(\t H_\epnu[K_\nu\theta_\epnu]\big)\Big)\rmd\nu\,.
\end{equs}
There, using \eqref{eq:Kmunu}, we have
\begin{equs}
\|K_{\mu}(\t  R_\epmu)\|_{L^{\infty}_{0;T}}&\lesssim \int_0^\mu \Big(\|    K_\nu\big( \t \D F_\epnu[K_\nu\theta_\epnu] \dot G_\nu (R_\epnu+\phi_\eps)\big)\|_{L^{\infty}_{0;T}}
+\|K_\nu\big(\t  H_\epnu[K_\nu\theta_{\epnu}]\big)\|_{L^{\infty}_{0;T}}\Big)\rmd\nu\,.
\end{equs}
As previously, the second term is handled in Lemma~\ref{lem:fixedpointRH} below, see \eqref{eq:RHSeqRinftyH}. To study the first term, as before, we insert $1=v_\nu+w_\nu$ in front of $\D F_\epnu$, so that this term splits as
\begin{equs}
    \|    K_\nu\big( &\t\D F_\epnu[K_\nu\theta_\epnu] \dot G_\nu (R_\epnu+\phi_\eps)\big)\|_{L^{\infty}_{0;T}}\\&\lesssim
     \|    K_\nu\big( \t v_\nu \D F_\epnu[K_\nu\theta_\epnu] \dot G_\nu (R_\epnu+\phi_\eps)\big)\|_{L^{\infty}_{0;T}}+ \|    K_\nu\big( \t w_\nu\D F_\epnu[K_\nu\theta_\epnu] \dot G_\nu (R_\epnu+\phi_\eps)\big)\|_{L^{\infty,1}_{0;T}}\,.
\end{equs}
Again, these term are handled separately in Lemmas~\ref{lem:fixedpointRw} and \ref{lem:fixedpointR}, see \eqref{eq:RHSeqRinftyw} and \eqref{eq:RHSeqRinfty}.

To conclude, proceeding as before, again we can bound $\|K_{\mu}(\t  R_\epmu)\|_{L^{\infty}_{0;T}}$ by $\mu^{\alpha-\kappa/2}\mu_T^{\kappa/4}$, and bring the implicit constant to a value $C_{\vphi,R}$ by taking the time $T$ small enough. This finishes the proof.
\end{proof}
Some technical results necessary to the proof of Proposition~\ref{eq:prop1} are proven below.
\begin{lemma}\label{lem:fixedpointRH}
Assume that $\|\theta_\nu\|_{{\tt S},\nu,T}\leqslant C_{\vphi,R,\varpi}$. Then, uniform in $\nu\in(0,1]$, the following estimates hold for every $\eta>0$:
 \begin{equs}
 \label{eq:RHSeqR1H}
    \|  \tH_\epnu[K_\nu\theta_\nu]\|_{L^{\infty,1}_{0;T}}&\lesssim\nu^{-1+\alpha-\eta} \,,
 \\
 \label{eq:RHSeqRinftyH}
    \|    K_\nu\big(\t H_\epnu[K_\nu\theta_\nu] \big)\|_{L^\infty_{0;T}}&\lesssim\nu^{-1+\alpha-\eta} \,.
\end{equs}
\end{lemma}
\begin{proof}
\eqref{eq:RHSeqR1H} is a straightforward consequence of \eqref{eq:boundL}:
\begin{equs}    \|\tH_\epnu[K_\nu\theta_{\epnu}]\|_{L^{\infty,1}_{0;T}}\lesssim\|\tH_\epnu[K_\nu\theta_{\epnu}]\|_{L^{\infty}_{0;T}}\lesssim\nu^{-1+\alpha-\eta}\,.
\end{equs}

To prove \eqref{eq:RHSeqRinftyH}, observe that by assumption $H_\epnu[K_\nu\theta_{\epnu}]$ is supported on $[0,1]$. We can thus multiply $\T$ by a smooth function $u$ which has support on $[-1,2]$ and is equal to one on $[0,1]$. 
Note that, contrary\footnote{Note that $\t \not \in \mcW^\infty_{N,\nu}$ because it is unbounded for very negative values of $x_0$} to $\t $, $u\t $ is in $\mcW^\infty_{N,\nu}$. 
The weight $u\T$ can thus be eliminated using \eqref{eq:1052a}. 
Combining this observation with \eqref{eq:boundL} we end up with
\begin{equs}
    \|K_\nu\big(\t H_\epnu[K_\nu\theta_{\nu}]\big)\|_{L^{\infty}_{0;T}}=\|K_\nu\big( u\t  H_\epnu[K_\nu\theta_{\nu}]\big)\|_{L^{\infty}_{0;T}}\lesssim\| \tH_\epnu[K_\nu\theta_{\nu}]\|_{L^{\infty}_{0;T}}\lesssim\nu^{-1+\alpha-\eta}\,.
\end{equs}
\end{proof}
\begin{lemma}\label{lem:fixedpointRw}
Assume that $\|\theta_\nu\|_{{\tt S},\nu,T}\leqslant C_{\vphi,R,\varpi}$. Then, uniform in smooth $\lambda$ and in $\nu\in(0,1]$, the following estimates hold for every $\eta>0$:
 \begin{equs}
 \label{eq:RHSeqR1w}
    \|    K_\nu\big( w_\nu \D F_\epnu[K_\nu\theta_\nu] \dot G_\nu (\lambda+\phi_\eps)\big)\|_{L^{\infty,1}_{0;T}}&\lesssim\nu^{-1+\alpha-\eta}\big(1+\|\lambda\|_{\ttR,\nu,T}\big) \,,
 \\
 \label{eq:RHSeqRinftyw}
    \|    K_\nu\big(\t w_\nu \D F_\epnu[K_\nu\theta_\nu] \dot G_\nu (\lambda+\phi_\eps)\big)\|_{L^\infty_{0;T}}&\lesssim\nu^{-1+\alpha-\eta}\big(1+\|\lambda\|_{\ttR,\nu,T}\big) \,.
\end{equs}
\end{lemma}
\begin{proof}
We first prove \eqref{eq:RHSeqR1w}. To do so, we rely on the fact that $w_\nu\in\mcW_{N,\nu}^1$. An application of \eqref{eq:1052b} thus yields
\begin{equs}
     \|    K_\nu\big(w_\nu \D &F_\epnu[K_\nu\theta_\nu] \dot G_\nu (\lambda+\phi_\eps)\big)\|_{L^{\infty,1}_{0;T}}\\&\lesssim\nu^2 \|    \D \tF_\epnu[K_\nu\theta_\nu] \dot G_\nu (\lambda+\phi_\eps)\|_{L^{\infty}_{0;T}}\\
     &\lesssim\nu^2\|  \widetilde{\D F}_\epnu[K_\nu\theta_\nu] \|_{\mcL^{\infty,\infty}_{0;T}}\|\mcR_\nu\dot G_\nu \lambda\|_{L^{\infty}_{0;T}}+
     \nu^2\|\  \widetilde{\D F}_\epnu[K_\nu\theta_\nu]\mcR_\nu^\dagger \dot G_\nu \phi_\eps\|_{L^{\infty}_{0;T}}
     \\
     &\lesssim\nu^{-1+2\alpha-3\kappa/2-\eta}\|\lambda\|_{{\tt R},\nu,T}+\nu^{-1+\alpha-\eta}\lesssim\nu^{-1+\alpha-\eta}\big(1+\|\lambda\|_{{\tt R},\nu,T}\big)\,.\label{eq:stepinfixedpoint}
\end{equs}
In order to go from the second to the third line, in the first term we controlled the force using \eqref{eq:boundDF} and the remainder using \eqref{eq:RtoXGdot}, while in the second term we used \eqref{eq:boundDFharmo}. In the last inequality we used $\kappa\leqslant2\alpha/3$.

We now prove \eqref{eq:RHSeqRinftyw}. To do so, we set $\tilde w_\nu(t)\eqdef\nu^{-2}t w_\nu(t/\nu^2)$. observe that right as $w_\nu$, $\tilde w_\nu$ lies in $\mcW^\infty_\nu$, so that we can use \eqref{eq:1052a} to eliminate it. Using the upper bound on $\nu^2 \|    \D \tF_\epnu[K_\nu\theta_\epnu] \dot G_\nu (\lambda+\phi_\eps)\|_{L^{\infty}_{0;T}}$ obtained in \eqref{eq:stepinfixedpoint}, we thus have
\begin{equs}
     \|    K_\nu\big(\t w_\nu \D &F_\epnu[K_\nu\theta_\nu] \dot G_\nu (\lambda+\phi_\eps)\big)\|_{L^{\infty}_{0;T}}=\nu^2 \|    K_\nu\big( \tilde w_\nu \D F_\epnu[K_\nu\theta_\nu] \dot G_\nu (\lambda+\phi_\eps)\big)\|_{L^{\infty}_{0;T}}\\&\lesssim\nu^2 \|    \D \tF_\epnu[K_\nu\theta_\nu] \dot G_\nu (\lambda+\phi_\eps)\|_{L^{\infty}_{0;T}}\lesssim\nu^{-1+\alpha-\eta}(1+\|\lambda\|_{{\tt R},\nu,T}\big)\,.
\end{equs}
\end{proof}
\begin{lemma}\label{lem:fixedpointR}
Assume that $\|\theta_\nu\|_{{\tt S},\nu,T}\leqslant C_{\vphi,R,\varpi}$. Then, uniform in smooth $\lambda$ and in $\nu\in(0,1]$, the following estimates hold for every $\eta>0$:
 \begin{equs}
  \label{eq:RHSeqR1}
    \|    K_\nu\big( v_\nu \D F_\epnu[K_\nu\theta_\nu] \dot G_\nu (\lambda+\phi_\eps)\big)\|_{L^{\infty,1}_{0;T}}&\lesssim\nu^{-1+\alpha-\eta}\big(1+\|\lambda\|_{\ttR,\nu,T}\big) \,,
 \\
 \label{eq:RHSeqRinfty}
    \|    K_\nu\big(\t v_\nu \D F_\epnu[K_\nu\theta_\nu] \dot G_\nu (\lambda+\phi_\eps)\big)\|_{L^\infty_{0;T}}&\lesssim\nu^{-1+\alpha-\eta}\big(1+\|\lambda\|_{\ttR,\nu,T}\big) \,.  
\end{equs}
\end{lemma}
\begin{proof}  We first prove \eqref{eq:RHSeqRinfty}. We start from
\begin{equs}\label{eq:beginningprooflong}
    K_\nu\big(\t &v_\nu \D F_\epnu[K_\nu\theta_\nu] \dot G_\nu (\lambda+\phi_\eps)\big)(x)\\&=\int_{\Lambda^2} K_\nu(x-y) y_0 v_\nu(y_0) \D F_\epnu[K_\nu\theta_\nu](y,z)\dot G_\nu (\lambda+\phi_\eps)(z)\rmd y\rmd z\,.
\end{equs}
Using the identity $y_0=y_0-z_0+z_0$, we reexpress the above as
\begin{equs}
    K_\nu\big(\t v_\nu \D F_\epnu[K_\nu\theta_\nu] \dot G_\nu (\lambda+\phi_\eps)\big)(x)&=\int_{\Lambda^2} K_\nu(x-y)  v_\nu(y_0) \bfX_0\D F_\epnu[K_\nu\theta_\nu] (y,z)\dot G_\nu (\lambda+\phi_\eps)(z)\rmd y\rmd z
    \\
    &\quad+\int_{\Lambda^2} K_\nu(x-y)  v_\nu(y_0) \D F_\epnu[K_\nu\theta_\nu] (y,z)z_0\dot G_\nu( \lambda+\phi_\eps)(z)\rmd y\rmd z
            \,.
\end{equs}
Writing $z_0\dot G_\nu \lambda(z)=\int_\Lambda z_0\dot G_\nu (z-w) \lambda(w)\rmd w$ and using $z_0=z_0-w_0+w_0$, we have
\begin{equs}\label{eq:tricktelescop}
    z_0\dot G_\nu \lambda(z)=(\bfX_0 \dot G_\nu) \lambda(z)+\dot G_\nu(\t \lambda)(z)\,,
\end{equs}
which implies that
\begin{equs}\label{eq:longexpression}
    K_\nu\big(\t &v_\nu \D F_\epnu[K_\nu\theta_\nu] \dot G_\nu (\lambda+\phi_\eps)\big)(x)\\&=\int_{\Lambda^2}K_\nu(x-y)  v_\nu(y_0) \bfX_0\D F_\epnu[K_\nu\theta_\nu] (y,z)\dot G_\nu (\lambda+\phi_\eps)(z)\rmd y\rmd z
    \\
    &\quad+\int_{\Lambda^2}K_\nu(x-y)  v_\nu(y_0) \D F_\epnu[K_\nu\theta_\nu] (y,z)\big((\bfX_0 \dot G_\nu) \lambda(z)+\dot G_\nu(\t \lambda)(z)\big)\rmd y\rmd z
     \\
    &\quad+\int_{\Lambda^2}K_\nu(x-y)  v_\nu(y_0) \D F_\epnu[K_\nu\theta_\nu] (y,z)z_0\dot G_\nu\phi_\eps(z)\rmd y\rmd z\\
    &=\int_{\Lambda^2}K_\nu(x-y)  v_\nu(y_0) K_\nu\big(\bfX_0\D F_\epnu[K_\nu\theta_\nu] (y,\bigcdot)\big)(z)\mcR_\nu^\dagger\dot G_\nu (\lambda+\phi_\eps)(z)\rmd y\rmd z
    \\
    &\quad+\int_{\Lambda^2}K_\nu(x-y)  v_\nu(y_0)K_\nu\big( \D F_\epnu[K_\nu\theta_\nu] (y,\bigcdot)\big)(z)\big(\mcR^\dagger_\nu(\bfX_0 \dot G_\nu) \lambda(z)+\mcR^\dagger_\nu\dot G_\nu(\t \lambda)(z)\big)\rmd y\rmd z
     \\
    &\quad+\int_{\Lambda^2}K_\nu(x-y)  v_\nu(y_0) K_\nu\big(\D F_\epnu[K_\nu\theta_\nu] (y,\bigcdot)\big)(z) \mcR^\dagger_\nu( \t\dot G_\nu\phi_\eps)(z)\rmd y\rmd z
            \,.
\end{equs}
Overall, we have thus obtained that
\begin{equs}
     \|  K_\nu&\big(\t v_\nu \D F_\epnu[K_\nu\theta_\nu] \dot G_\nu (\lambda+\phi_\eps)\big)\|_{L^\infty_{0;T}}\\
&\lesssim
\| K_\nu^{\otimes2}\big( (v_\nu\otimes1)\bfX_0\D F_\epnu[K_\nu\theta_\nu]\big)\mcR_\nu^\dagger\dot G_\nu(\lambda+\phi_\eps)\|_{L^{\infty}_{0;T}}\\
&\quad+\| K_\nu^{\otimes2}\big( (v_\nu\otimes1)\D F_\epnu[K_\nu\theta_\nu]\big)\big(\mcR_\nu^\dagger(\bfX_0\dot G_\nu)\lambda+\mcR_\nu^\dagger\dot G_\nu(\t\lambda)+\mcR_\nu^\dagger(\t\dot G_\nu\phi_\eps)\big)\|_{L^{\infty}_{0;T}}\,.
\end{equs}
At this stage, using the fact that $v_\nu$ belongs to $\mcW^\infty_{N,\nu}$, we can eliminate the weight $v_\nu$ using \eqref{eq:1052a}, which gives
\begin{equs}\label{eq:stepintediousproof}
     \|  K_\nu&\big(\t v_\nu \D F_\epnu[K_\nu\theta_\nu] \dot G_\nu (\lambda+\phi_\eps)\big)\|_{L^\infty_{0;T}}\\
&\lesssim
\|  \bfX_0\widetilde{\D  F}_\epnu[K_\nu\theta_\nu]\mcR_\nu^\dagger\dot G_\nu(\lambda+\phi_\eps)\|_{L^{\infty}_{0;T}}+\|\widetilde{\D F}_\epnu[K_\nu\theta_\nu]\big(\mcR_\nu^\dagger(\bfX_0\dot G_\nu)\lambda+\mcR_\nu^\dagger\dot G_\nu(\t\lambda)+\mcR_\nu^\dagger(\t\dot G_\nu\phi_\eps)\big)\|_{L^{\infty}_{0;T}}\\
&\lesssim\norm{\bfX_0\widetilde{\D F}_\epnu[K_\nu\theta_\nu]}_{\mcL^{\infty,\infty}_{0;T}}\norm{\mcR_\nu^\dagger\dot G_\nu\lambda}_{L^{\infty}_{0;T}}+\norm{\bfX_0\widetilde{\D  F}_\epnu[K_\nu\theta_\nu]\mcR_\nu^\dagger\dot G_\nu\phi_\eps}_{L^{\infty}_{0;T}}\\&\quad+\norm{\widetilde{\D F}_\epnu[K_\nu\theta_\nu]}_{\mcL^{\infty,\infty}_{0;T}}\Big(\norm{\mcR_\nu^\dagger(\bfX_0\dot G_\nu)\lambda}_{L^{\infty}_{0;T}}+\norm{\mcR_\nu^\dagger\dot G_\nu(\t\lambda)}_{L^{\infty}_{0;T}}\Big)+\norm{\widetilde{\D F}_\epnu[K_\nu\theta_\nu]\mcR_\nu^\dagger(\t\dot G_\nu\phi_\eps)}_{L^{\infty}_{0;T}}
\,.
\end{equs}
We can now control the first and third terms using \eqref{eq:boundDF}, \eqref{eq:RtoXGdot} and \eqref{eq:RtoGdotT}, and the second term and last terms using \eqref{eq:boundDFharmo}. This yields
\begin{equs}
      \|  K_\nu\big(\t v_\nu \D F_\epnu[K_\nu\theta_\nu] \dot G_\nu (\lambda+\phi_\eps)\big)\|_{L^\infty_{0;T}}
&\lesssim\nu^{-1+2\alpha-3\kappa/2-\eta}\|\lambda\|_{\ttR,\nu,T}+\nu^{-1+\alpha-\eta}\,.
\end{equs}
Observing that $\kappa\leqslant2\alpha/3$, \eqref{eq:RHSeqRinfty} is thus proven.

We now turn to the proof of \eqref{eq:RHSeqR1}. We start from 
\begin{equs}
    K_\nu\big( v_\nu \D F_\epnu[K_\nu\theta_\nu] \dot G_\nu (\lambda+\phi_\eps)\big)(x)&=\int_{\Lambda^2}K_\nu(x-y) y_0^{-1} y_0 v_\nu(y_0) \D F_\epnu[K_\nu\theta_\nu] (y,z)\dot G_\nu (\lambda+\phi_\eps)(z)\rmd y\rmd z \\
    &=\nu^{-2}\int_{\Lambda^2}K_\nu(x-y) y_0  \tilde 
    v_\nu(y_0) \D F_\epnu[K_\nu\theta_\nu] (y,z)\dot G_\nu (\lambda+\phi_\eps)(z)\rmd y\rmd z \,.
   \end{equs}
   Here, on the first line we inserted $1=y_0^{-1}y_0$ and on the second line we set $\tilde v_\nu(t)\eqdef \nu^2 t^{-1}v_\nu(t)$. This expression is absolutely the same as \eqref{eq:beginningprooflong}, except that $v_\nu$ is replaced by $\nu^{-2}\tilde v_\nu$. Following the same steps as in the proof of \eqref{eq:RHSeqRinfty} above, we thus end up with 
   \begin{equs}
     \|  K_\nu&\big( v_\nu \D F_\epnu[K_\nu\theta_\nu] \dot G_\nu (\lambda+\phi_\eps)\big)\|_{L^{\infty,1}_{0;T}}\\
&\lesssim
\nu^{-2}\| K_\nu^{\otimes2}\big( (\tilde v_\nu\otimes1)\bfX_0\D F_\epnu[K_\nu\theta_\nu]\big)\mcR_\nu^\dagger\dot G_\nu(\lambda+\phi_\eps)\|_{L^{\infty,1}_{0;T}}\\
&\quad+\nu^{-2}\| K_\nu^{\otimes2}\big( (\tilde v_\nu\otimes1)\D F_\epnu[K_\nu\theta_\nu]\big)\big(\mcR_\nu^\dagger(\bfX_0\dot G_\nu)\lambda+\mcR_\nu^\dagger\dot G_\nu(\t\lambda)+\mcR_\nu^\dagger(\t\dot G_\nu\phi_\eps)\big)\|_{L^{\infty,1}_{0;T}}\,.
\end{equs}
   At this stage, note that while $v_\nu$ only lies in $\mcW_{N,\nu}^\infty$, $\tilde v_\nu$ lies in $\mcW_{N,\nu}^{1+\eta}$ for every $\eta>0$. The weights $\tilde v_\nu$ can thus be eliminated using \eqref{eq:1052b}, which entails that $ \|  K_\nu\big( v_\nu \D F_\epnu[K_\nu\theta_\nu] \dot G_\nu (\lambda+\phi_\eps)\big)\|_{L^{\infty,1}_{0;T}}$ is smaller than the RHS of \eqref{eq:stepintediousproof} times  $\nu^{-2+2/(1+\eta)}\lesssim\nu^{-\eta}$. This concludes the proof of \eqref{eq:RHSeqR1}.
\end{proof}

\subsection{Convergence of the solution}\label{subsec:solution}
We are now ready to construct $$\psi_\eps-G\phi_\eps=-\int_0^{\mu_T}\dot G_\nu \big(F_{\eps,\mu_T}[0]+R_{\eps,\mu_T}\big)\rmd\nu$$ in the space $\mcC^{\alpha-\kappa}([0,T]\times\T^n)$. 
\begin{lemma}\label{lem:solutionconstruction}
Fix $\eps\in(0,1]$. Recall that $\psi_\eps$ is the solution defined by \eqref{eq:Phi2}. There exists a random $T\in(0,1]$ such that it holds
    \begin{equs}
\norm{\psi_\eps-G\phi_\eps}_{\mcC^{\alpha-\kappa}([0,T]\times\T^n)} = \sup_{\mu\in(0,1]}  \mu^{-\alpha+\kappa}   \norm{(Q_\mu-\Id)\big(\psi_\eps-G\phi_\eps\big)}_{L^\infty_{0;T}}\lesssim 1\,.
    \end{equs}
\end{lemma}
\begin{proof}
We first consider the more subtle case of $\mu\leqslant\mu_T$.
Starting from \eqref{eq:Phi2}, we have
    \begin{equs}
     \norm{(Q_\mu-\Id)\big(\psi_\eps-G\phi_\eps\big)}_{L^\infty_{0;T}}&\lesssim   \int_0^{\mu_T}   \norm{(Q_\mu-\Id)\dot G_\nu\big(F_{\eps,\mu_T}[0]+R_{\eps,\mu_T}\big)}_{L^\infty_{0;T}}\rmd\nu  \\  &\lesssim  \int_0^{\mu} \label{eq:eqinterproof45}  \norm{\dot G_\nu\big(F_{\eps,\mu_T}[0]+R_{\eps,\mu_T}\big)}_{L^\infty_{0;T}}\rmd\nu     
     \\&\textcolor{white}{\lesssim}+\int_\mu^{\mu_T}   \norm{(Q_\mu-\Id)Q_\nu\mcP_\nu\dot G_\nu\big(F_{\eps,\mu_T}[0]+R_{\eps,\mu_T}\big)}_{L^\infty_{0;T}}\rmd\nu\,.
    \end{equs}
The second term is handled using \eqref{eq:Lem45} (which is proven in Lemma~\ref{lem:Lem45} below), and \eqref{eq:heat1}:
\begin{equs}
\int_0^{\mu} &  \norm{\dot G_\nu\big(F_{\eps,\mu_T}[0]+R_{\eps,\mu_T}\big)}_{L^\infty_{0;T}}\rmd\nu \\  
    \lesssim &
      \int_0^{\mu} \norm{\mcR_{+,\nu}\dot G_\nu}_{\mcL^{\infty,\infty}_{0;T}}\norm{K_{+,\nu}\big(F_{\eps,\mu_T}[0]+R_{\eps,\mu_T}\big)}_{L^\infty_{0;T}}\rmd\nu \lesssim
      \int_0^\mu \nu^{-1+\alpha-\kappa/2}\rmd\nu\lesssim\mu^{\alpha-\kappa/2}\,.
\end{equs}
The third term is a bit more subtle and requires using \eqref{eq:comKmuKnu}:
\begin{equs}
    \int_\mu^{\mu_T} &  \norm{(Q_\mu-\Id)Q_\nu\mcP_\nu\dot G_\nu\big(F_{\eps,\mu_T}[0]+R_{\eps,\mu_T}\big)}_{L^\infty_{0;T}}\rmd\nu\lesssim\mu^{2}\int_\mu^{\mu_T}\nu^{-2} \norm{\mcP_\nu\mcR_{+,\nu}\dot G_\nu K_{+,\mu}\big(F_{\eps,\mu_T}[0]+R_{\eps,\mu_T}\big)}_{L^\infty_{0;T}}\rmd\nu\,.
\end{equs}
This plus \eqref{eq:Lem45} and \eqref{eq:heat1} gives
\begin{equs}
     \int_\mu^{\mu_T}   \norm{(Q_\mu-\Id)Q_\nu\mcP_\nu\dot G_\nu\big(F_{\eps,\mu_T}[0]+R_{\eps,\mu_T}\big)}_{L^\infty_{0;T}}\rmd\nu&\lesssim\mu^{2}\int_\mu^{\mu_T}\nu^{-1}\rmd\nu \,\mu^{-2+\alpha-\kappa/2} \lesssim\mu^{\alpha-\kappa}\,.
\end{equs}
In the last inequality, we have used the fact that the integrand is now no longer integrable at $\nu=0$. 

The case $\mu\geqslant\mu_T$ is easier, since we do not have to leverage the presence of the operator $Q_\mu-\Id$. Proceeding exactly as for the first integral in \eqref{eq:eqinterproof45} yields
\begin{equs}
     \norm{(Q_\mu-\Id)\big(\psi_\eps-G\phi_\eps\big)}_{L^\infty_{0;T}}&\lesssim\mu_T^{\alpha-\kappa/2}\lesssim\mu^{\alpha-\kappa/2}\,.
\end{equs}
\end{proof}
\begin{remark}\label{rem:convergenceSol}
    Carrying out the analysis of the proof of Lemma~\ref{lem:solutionconstruction} applied to $$\norm{(Q_\mu-\Id)\big(\psi_\eps-G\phi_\eps\big)-(Q_\mu-\Id)\big(\psi_{\eps'}-G\phi_{\eps'}\big)}_{L^\infty_{0;T}}\,,$$
one is led to bound this quantity using only $
\norm{K_\nu\big(F_{\eps,\mu_T}[0]-F_{\eps',\mu_T}[0]\big)}_{L^\infty_{0;T}}$ and 
$
\|K_{+,\nu} \big(R_{\eps,\mu_T}-R_{\eps',\mu_T}\big)\|_{L^\infty_{0;T}}$. In particular, by the convergence properties of $F_{\eps,\bigcdot}$ and $R_{\eps,\bigcdot}$ (see Remarks~\ref{rem:convergenceFL} and \ref{rem:convergenceR}), it suffices to slightly modify the proof of Lemma~\ref{eq:Lem45} below in order to obtain that $\phi_\eps-G\phi_\eps$ converges in probability in $\mcC^{\alpha-\kappa}([0,T]\times\T^n)$. 
\end{remark}
Finally, the following lemma is necessary to the above construction of the solution. While its proof partly follows the structure of the Section 11 of \cite{Duch22}, it is considerably different, since we leverage the fact that the solution is of positive regularity to balance the non-polynomial growth of the non-linearity. 
\begin{lemma}\label{lem:Lem45}
Pick $\tau\in(0,1)$ depending only on $\Gamma$ close enough to one so that 
\begin{equs}\label{eq:assumptau}
    (2/\tau^2-1)^{N_1^{3\Gamma+1}}\leqslant2\,.
\end{equs}
Let $C_\varpi$ be the maximum of the implicit constants in \eqref{eq:harmoniccompletion2} and \eqref{eq:harmoniccompletionwithderivative}, and set $C_S\eqdef C_F\mcG(1+C_\varpi)+C_R$, where $C_F$ was defined in Corollary~\ref{coro:1} and $C_R$ in Remark~\eqref{rem:convergenceR}. 
Then, there exists a random $\tilde{T}\in(0,1)$ such that, for any $T \in (0,\tilde{T}]$ and any $p\in\N$, it holds
    \begin{equs}
         \norm{K_{+,\tau^p\mu_T}\big( F_{\eps,\mu_T}[0]+R_{\eps,\mu_T}\big)}_{L^\infty_{0;T}}\leqslant  C_S(\tau^p\mu_T)^{-2+\alpha-\kappa/2}\,.
    \end{equs}
Consequently, by interpolation, we have that 
\begin{equs}\label{eq:Lem45}
     \norm{K_{+,\mu}\big( F_{\eps,\mu_T}[0]+R_{\eps,\mu_T}\big)}_{L^\infty_{0;T}}\lesssim\mu^{-2+\alpha-\kappa/2}
\end{equs}
uniformly in $\mu\in(0,1]$.
\end{lemma}
\begin{proof}
We prove the thesis by recursion on $p$. First, observe that by \eqref{eq:boundF} and \eqref{eq:boundR} we have for every $\eta>0$
\begin{equ}
    \norm{K_{+,\mu_T}\big( F_{\eps,\mu_T}[0]+R_{\eps,\mu_T}\big)}_{L^\infty_{0;T}}\leqslant C_F\mcG(1)\mu_T^{-2+\alpha-\eta}+C_R\mu_T^{-2+\alpha-\kappa/2}\,.
\end{equ}
Taking $\eta=\kappa/2$, we can deduce from the above inequality that
\begin{equ}
   \norm{K_{+,\mu_T}\big( F_{\eps,\mu_T}[0]+R_{\eps,\mu_T}\big)}_{L^\infty_{0;T}}\leqslant C_S\mu_T^{-2+\alpha-\kappa/2}\,.    
\end{equ}
Now, suppose that that the desired result
 \begin{equs}
         \norm{K_{+,\tau^p\mu_T}\big( F_{\eps,\mu_T}[0]+R_{\eps,\mu_T}\big)}_{L^\infty_{0;T}}\leqslant  C_S(\tau^p\mu_T)^{-2+\alpha-\kappa/2}
    \end{equs}
holds for some fixed $p\in\N_{>0}$. Using \eqref{eq:Kmunu} to replace $K_{\tau^{p}\mu_T}$ by $K_{\tau^{p+1}\mu_T}$ in the above inequality, along with the assumption \eqref{eq:assumptau} yields
\begin{equs}    \norm{K_{+,\tau^{p+1}\mu_T}\big( F_{\eps,\mu_T}[0]+R_{\eps,\mu_T}\big)}_{L^\infty_{0;T}}
 &\leqslant 2C_S(\tau^p\mu_T)^{-2+\alpha-\kappa/2}\leqslant 2C_S(\tau^{p+1}\mu_T)^{-2+\alpha-\kappa/2}
    \,.\textcolor{white}{blablabla}\label{eq:intertau}
\end{equs}
Note also that by interpolation, the above above bound holds with $\tau^{p+1}\mu_T$ replaced by any $\nu\in[\tau^{p+1}\mu_T,\mu_T]$. Here, recall that we have for all $x\in\Lambda_{0;T}$ and $\mu\in(0,\mu_T]$
\begin{equs}\label{eq:FmuRmuBis}
   \big( F_{\eps,\mu_T}[0]+R_{\eps,\mu_T}\big)(x)=\big(F_{\eps,\tau^{p+1}\mu_T}[\psi_{\eps,\tau^{p+1}\mu_T}]+R_{\eps,\tau^{p+1}\mu_T}\big)(x)\,.
\end{equs}
Using \eqref{eq:vphimu}, \eqref{eq:FmuRmuBis} and the support properties of $G_{\mu_T}$, we obtain that $C_\varpi\mu^{(-|\mfl|+\alpha-\kappa)\wedge0}+$
\begin{equs}    \norm{\partial_\ttx^\mfl\big(\psi_{\eps,\tau^{p+1}\mu_T}-G_{\tau^{p+1}\mu_T}\phi_\eps\big)}_{L^\infty_{0;T}}&\leqslant\int^{\mu_T}_{\tau^{p+1}\mu_T}\norm{\partial_\ttx^\mfl\dot G_{\nu} \big( F_{\eps,\mu_T}[0]+R_{\eps,\mu_T}\big)}_{L^\infty_{0;T}}\rmd\nu\\
&\leqslant \int^{\mu_T}_{\tau^{p+1}\mu_T}\norm{\partial_\ttx^\mfl\mcR_{+,\nu}\dot G_{\nu}}_{\mcL^{\infty,\infty}_{0;T}}\norm{K_{+,\nu} \big( F_{\eps,\mu_T}[0]+R_{\eps,\mu_T}\big)}_{L^\infty_{0;T}}\rmd\nu\,.
\end{equs}
Using \eqref{eq:heat1} and \eqref{eq:intertau}, we finally obtain
\begin{equs}    \norm{\partial_\ttx^\mfl\big(\psi_{\eps,\tau^{p+1}\mu_T}-G_{\tau^{p+1}\mu_T}\phi_\eps\big)}_{L^\infty_{0;T}}&\leqslant 2C_SC_G
\int^{\mu_T}_{\tau^{p+1}\mu_T}\nu^{-1+\alpha-|\mfl|-\kappa/2}\rmd\nu\,.
\end{equs}
where $C_G$ is the implicit constant in \eqref{eq:heat1}. If $\mfl=0$ then the integral over $\nu$ is bounded by $ 2C_SC_G(\alpha-\kappa/2)^{-1}\mu_T^{\alpha-\kappa/2}$. If $|\mfl|=1$, then it is bounded by $2C_SC_G(1-\alpha+\kappa)^{-1}\mu_T^{\kappa/2}(\tau^{p+1}\tau_T)^{-1+\alpha-\kappa}$.

By taking the time $T$ small enough, we can enforce 
$$\big( 2C_SC_G(\alpha-\kappa/2)^{-1}\mu_T^{\alpha-\kappa/2}\big)\vee\big(2C_SC_G(1-\alpha+\kappa)^{-1}\mu_T^{\kappa/2}\big)\leqslant1\,.$$ $\theta_{\tau^{p+1}\mu_T}=\mcR_{\tau^{p+1}\mu_T}\psi_{\eps,\tau^{p+1}\mu_T}$ thus verifies the hypothesis of Corollary~\ref{coro:1} with $C_{\theta,R,\varpi}=1+C_\varpi$, and we can make use of \eqref{eq:boundF} to control $F_{\eps,\tau^{p+1}\mu_T}[\psi_{\eps,\tau^{p+1}\mu_T}]$ in \eqref{eq:FmuRmuBis}. Hence, using \eqref{eq:boundF} and \eqref{eq:boundR}, we finally obtain
\begin{equs}
    \norm{K_{+,\tau^{p+1}\mu_T}\big( F_{\eps,\mu_T}[0]+R_{\eps,\mu_T}\big)}_{L^\infty_{0;T}}&\leqslant\norm{\tF_{\eps,\tau^{p+1}\mu_T}[\psi_{\eps,\tau^{p+1}\mu_T}]}_{L^\infty_{0;T}}+\norm{K_{+,\tau^{p+1}\mu_T}R_{\eps,\tau^{p+1}\mu_T}}_{L^\infty_{0;T}}
       \\&\leqslant   C_F\mcG(1+C_\varpi)(\tau^{p+1}\mu_T)^{-2+\alpha-\eta}+C_R(\tau^{p+1}\mu_T)^{-2+\alpha-\kappa/2}\,,
\end{equs}
which is the desired result. 
\end{proof}

\section{Probabilistic analysis}\label{sec:Sec4}
This Section is devoted to the construction of the stationary force coefficients,  primarily through probabilistic arguments.
A crucial observation first made in \cite{Duch21} (also used in \cite{Duch22,GR23}) was that the Polchinski equation could be used to control cumulants of force coefficients -- we apply this approach to our setting to obtain estimates for the stationary force coefficients.

The need for renormalization appears when trying to close the Polchinski flow for cumulants, we want to impose finite boundary conditions at large scales for the expectations of relevant force coefficients which requires us to choose divergent in $\eps$ initial data at scale $\eps$ for the same quantities.  
However, the coordinates introduced in section~\ref{subsubsec:coord} are non-local, so simply choosing our renormalization counterterm along the lines aboveabove would give a counterterm $\mfc_{\eps}$ in \eqref{eq:eqDefreg} that might be a non-local function of $\psi_{\eps}$ and $\partial_{\ttx} \psi_{\eps}$. 

Non-local counterterms are not satisfactory, and so in order to obtain a local counterterm like given in \eqref{eq:eqDefreg} we perform another space-time localization step as in \cite{Duch21,Duch22} using Taylor expansion. 
We can then write relevant force coefficients as a sum of completely local terms along with a non-local remainder which will be power-counting irrelevant. 
This Taylor expansion produces a larger system of coordinates\footnote{It would have also been possible to implement these localization within our probabilistic argument without introducing new ``generalized coordinates'' as in \cite{Duch23}, but this would have been quite messy since, unlike \cite{Duch23}, our scale cut-off doesn't force coefficients to have sharp space/space-time support properties.} which we describe in \eqref{subsec:gen_force_coeff}.

In Section~\ref{subsec:cumulant_analysis} we use the Polcshinki flow and the imposing of renormalization conditions to obtain cumulant estimates for the larger class of generalized coordinates mentioned above, the main result being the cumulant bounds Lemma~\ref{lem:cumul}. 

These are then used as input for a Kolmogorov-type argument in Section~\ref{subsec:cumulant_analysis}. 
This implementation of the Kolmogorov argument gives us control of Besov-type norms of relevant generalized force coefficients, summarized in Lemma~\ref{lem:ConcluProba}. 
However, our Kolmogorov argument isn't suitable to control irrelevant generalized force coefficents since it doesn't control ``large scales''.  
We bypass this issue in Section~\ref{subsec:44} where we use a deterministic argument (using the Polchsinki flow for force coefficients, not their cumulants) to show that path-wise control on the relevant generalized force coefficients gives control over the irrelevant ones -- this is the last ingredient in proving Theorem~\ref{thm:sto}.

\subsection{Generalized force coefficients}\label{subsec:gen_force_coeff}
We now introduce the promised extension of our previous set of coordinates. 
\begin{definition}  
Fix $a\in\mcM$, $(x,y^a)\in\Lambda^{[a]+1}$, and a family $\mfl^a=(\mfl^a_{\mfk ij})_{\mfk ij\in [a]}$ of indices $\mfl^a_{\mfk ij}\in\N^{n+1}$. 
We then set 
\begin{equ}
\mfM
\eqdef
\big\{\big(a,\mfl^a=(\mfl^a_{\mfk ij})_{\mfk ij\in [a]}\big)\in\mcM\times\big(\N^{n+1}\big)^{[a]}\big\}\,.
\end{equ}
We also define the set of \textit{generalized multi-indices} $\tilde{
\mfM}$ by setting
\begin{equs}    \widetilde{\mfM}\eqdef\mfM\times\{0,1,2\}\times\{0,1\}\;.
\end{equs}
We also write, for $t\in\{0,1\}$,  $\widetilde{\mfM}_t\eqdef\mfM\times\{0,1,2\}\times\{t\}\subset\widetilde{\mfM}$.    

Finally, for $\tilde{a}=(a,\mfl^a,s,t)\in\widetilde{\mfM}$ and $(x,y^a)\in\Lambda^{[a]+1}$, we introduce the \textit{generalized force coefficients}
\begin{equs}
    \xi_{\eps,\mu}^{\tilde{a}}(x,y^a)\eqdef\bfX^{\mfl^a}(x,y^a)\partial_\eps^s\partial_\mu^t\xi_{\eps,\mu}^a(x,y^a)\,,
\end{equs}
where the notation $\bfX^{\mfl^a}(x,y^a)$ was introduced in Definition~\ref{def:normforxi}.
\end{definition}

We introduce a corresponding extension of our set of derivators and use them to express the non-linearity in the flow equation. 
\begin{definition}  
We define
\begin{equ}    \tilde\mcD\eqdef\mcD\times\N^{n+1}\times\{0\}\times\{1\}\,.
\end{equ}
Fix $\tilde a=(a,\mfl^a,s,1)\in\widetilde\mfM_1$, $\tilde b=(b,\mfl^{b},s'_1,0)\in\widetilde{\mfM}_0$, $\tilde c=(c,\mfl^{c},s'_2,0)\in\widetilde{\mfM}_0$, $\tilde \bfd=(\bfd,\mfl^{\bfd},0,1)\in\tilde\mcD$, and $\sigma=(\sigma^\mfk_i)_{(\mfk,i)\in\mathrm{supp}(a_i^\mfk)}\in\mfS^{a_i}$. We say that
\begin{equs}
    \tilde a=\tilde b +\tilde c+\tilde d(\tilde\bfd)
\end{equs}
if the following three conditions hold:
\begin{enumerate}
\item  as multi-indices, we have
\begin{equs}
    a=b+c+d(\bfd)\,;
\end{equs}
\item  there exist families of indices $\mfm^b=(\mfm^b_{\mfk ij})_{\mfk ij\in [a]}\in\big(\N^{n+1}\big)^{[a]}$, $\mfm^c=(\mfm^c_{\mfk ij})_{\mfk ij\in [a]}\in\big(\N^{n+1}\big)^{[a]}$, $\mfm^\bfd=(\mfm^\bfd_{\mfk ij})_{\mfk ij\in [a]}\in\big(\N^{n+1}\big)^{[a]}$ such that 
\begin{equs}
    \mfm^b_{\mfk ij}+\mfm^c_{\mfk ij}+\mfm^{\bfd}_{\mfk ij}=\mfl^a_{\mfk i\sigma_i^\mfk(j)}\,,\;\forall\,(\mfk, i,j)\in[a]\,,
\end{equs}
and
\begin{equs}
    \mfm^c_{\mfk ij}=\mfm^{\bfd}_{\mfk ij}=0\,,\text{ if}\; j\leqslant b_i^\mfk-\1\{(\mfk,i)=(\mfk_0,k_0)\}\,,
\end{equs}
and
\begin{equs}     
\mfl_{\mfk ij}^{b}&= \begin{cases}  
      \sum_{\substack{(\mfk, i,j):j\geqslant 
 1 +b_{i}^\mfk-\1\{(\mfk,i)=(\mfk_0,k_0)\}
 }}\mfm^b_{\mfk ij} & \text{if} \;(\mfk,i,j)= (\mfk_0, k_0,b^{\mfk_0}_{k_0})\,, \\
        \mfm^{b}_{\mfk ij} & \forall (\mfk,i,j)\in[b]\setminus\{(\mfk_0, k_0,b^{\mfk_0}_{k_0})\} \,, \\
\end{cases}  \\  \mfl^c_{\mfk ij}&=\mfm^c_{\mfk i(j +b_i^\mfk-\1\{(\mfk,i)=(\mfk_0,k_0)\})}\;\forall\,(\mfk,i,j)\in[c]\,,\\ 
    \mfl^{\bfd}&= \sum_{(\mfk, i,j)\in[a]}\mfm^{\bfd}_{\mfk ij}\,,
\end{equs}
(all these conditions imply that $|\mfl^a|=|\mfl^b|+|\mfl^c|+|\mfl^\bfd|$);
\item  $s=s'_1+s'_2$.
\end{enumerate}
Finally, using the notation of Definition~\ref{def:B}, we define, for
$X\in\mcD(\Lambda^{[b]+1},\mcH^b)$ and $Y\in\mcD(\Lambda^{[c]+1},\mcH^c)$, the $\mcD(\Lambda^{[a]+1},\mcH^a)$-valued bilinear map 
\begin{equs}
    \tilde\B_\mu(X,Y)(x,&y^a)\\&\eqdef \frac{b^{\mfk_0}_{k_0}}{a!}\frac{s!}{s'_1!s'_2!}(1+\1\{(\mfk_0,\mfk_1)=(\mfe,\mff)\})\,X(x,y^b)\int_\Lambda\big(\text{$\bfX$}^{\mfl^{\text{\tiny{$\bfd$}}}}\partial_{\ttx}^{k_0+1-k_1}\dot G_\mu\big)(z-w)Y(w,y^c)\rmd w\,,
\end{equs}
where we write $\text{$\bfX$}^{\mfl^\bfd}(z-w)\equiv\text{$\bfX$}^{\mfl^\bfd}(z,w)\eqdef\frac{(z-w)^{\mfl^\bfd}}{\mfl^\bfd!}$.
\end{definition}
The following lemma verifies the above non-linearity captures the flow equation in our new coordinates. 
\begin{lemma}
For $\tilde a\in\widetilde{\mfM}_1$, we define
  \begin{equs}
     \text{$ \mathrm{Ind}$}(\tilde a)\eqdef\Big\{(\sigma,\tilde b,\tilde c,\tilde\bfd)\in\mfS^a\times\text{$\widetilde{\mfM}$}_0^2\times\tilde\mcD:\tilde a=\tilde b+\tilde c+\tilde d(\tilde\bfd)\Big\}
  \end{equs}
the index set for the flow of $\xi^{\tilde a}_\epmu$. Then, the generalized force coefficients verify the following system of flow equations:
\begin{equs}\label{eq:flowGen}
     \xi^{\tilde a}_\epmu=-\sum_{\substack{(\sigma,\tilde b,\tilde c,\bfd)\in\text{\scriptsize{$\mathrm{Ind}(\tilde a)$}}}}    \tilde\B_\mu(\xi^{\tilde b}_\epmu,\xi^{\tilde c}_\epmu)
    \,.
\end{equs}
\end{lemma}
\begin{proof} 
The key part of going from \eqref{eq:flow2} to \eqref{eq:flowGen} is tracking the effect of both polynomials and derivatives in $\eps$. 
While the effect of the derivatives in $\eps$ stems from the general Leibniz rule, we give more explanation concerning the polynomials, following Section 7 of \cite{Duch22}. If fact, observing that in $\bfX^{\mfl^a}(x,y^a)$, for some $y^a_{\mfk i\sigma^\mfk_i(j)}$ located in $y_c$ (with the notation of Definition~\ref{def:B}), i.e. $j\geqslant b^\mfk_i-\1\{(\mfk,i)=(\mfk_0,k_0)\}+1$, the term ${ (x-y^a_{\mfk i\sigma^\mfk_i(j)})^{\mfl^a_{\mfk i\sigma^\mfk_i(j)}}}$ can be expanded with the multinomial formula as
\begin{equs}
  (x-y^a_{\mfk i\sigma^\mfk_i(j)})^{\mfl^a_{\mfk i\sigma^\mfk_i(j)}}&=(x-z+z-w+w-y^a_{\mfk i\sigma^\mfk_i(j)})^{\mfl^a_{\mfk i\sigma^\mfk_i(j)}}\\
  &=\sum_{\substack{\mfm^b_{\mfk ij},\mfm^c_{\mfk ij},\mfm^\bfd_{\mfk ij}\\  \mfm^b_{\mfk ij}+\mfm^c_{\mfk ij}+\mfm^{\bfd}_{\mfk ij}=\mfl^a_{\mfk i\sigma_i^\mfk(j)}}}\frac{\mfl^a_{\mfk i\sigma^\mfk_i(j)}!}{\mfm^b_{\mfk ij}!\mfm^c_{\mfk ij}!\mfm^\bfd_{\mfk ij}!}
(x-z)^{\mfm^b_{\mfk ij}}
  (z-w)^{\mfm^\bfd_{\mfk ij}}  
  (w-y^a_{\mfk i\sigma^\mfk_i(j)})^{\mfm^c_{\mfk ij}}\,.
\end{equs}
Performing this expansion for every $(\mfk,i,j)\in[a]$ with $j\geqslant b^\mfk_i-\1\{(\mfk,i)=(\mfk_0,k_0)\}+1$ and organizing contributions by powers of $(z-w)$ and $(x-w)$ then gives the desired result.
\end{proof}
Writing the flow equation for relevant expectations requires an additional localization step which we now introduce notation for.  

We close this subsection by recalling the main statement \eqref{eq:localization} of the localization procedure presented in Section 8 of \cite{Duch22}.
Consider $(a,\mfl^a)\in\mfM$ and a function $f\in\mcD(\Lambda^{[a]+1})$. 

We first define $\delta^a\in\mcD'(\Lambda^{[a]+1})$ by setting 
\begin{equs}    \delta^a(x,y^a)\eqdef\prod_{\mfk ij\in[a]}\delta(x-y^a_{\mfk ij})\,.
\end{equs}
For $\mfm^a\in\big(\N^{n+1}\big)^a$ we also set 
\begin{equs}
    \partial_{y_a}^{\mfm^a}\delta^a(x,y^a)\eqdef\prod_{\mfk ij\in[a]}\delta^{({\mfm^a_{\mfk ij}})}(x-y^a_{\mfk ij})\,.
\end{equs}
We also define
\begin{equs}
    \bfI^{a} f(x)\eqdef \int_{\Lambda^{[a]}}f(x,y^a)\,\rmd y^a\,,
\end{equs}
and
\begin{equs}
      \bfL_\tau f(x,y^a)\eqdef \tau^{-(n+2)|[a]|}f(x,x+(y^a-x)/\tau)\,.
\end{equs}
It then holds that, for $\ell\in\{1,2\}$, the following equality holds in a distributional sense:
\begin{equs}\label{eq:localization}
   \bfX^{\mfl^a} f=&\sum_{\text{\scriptsize{$\mfm$}}^a:|\text{\scriptsize{$\mfl$}}^a+\text{\scriptsize{$\mfm$}}^a|<\ell}\binom{\text{$\mfl$}^a+\text{$\mfm$}^a}{\text{$\mfl$}^a}\partial_{{ y}^a}^{\text{\scriptsize{$\mfm$}}^a}\delta^a \, \text{$\bfI^{a}$}\big(   \bfX^{\mfl^a+\mfm^a} f\big) \\&+\sum_{\mfm^a:|\mfl^a+\mfm^a|=\ell}|\mfm^a|\binom{\mfl^a+\mfm^a}{\mfl^a}\int_0^1(1-\tau)^{|\mfm^a|-1}\partial_{{y}^a}^{\mfm^a} \text{$\bfL$}_\tau \big(\text{$\bfX$}^{\mfl^a+\mfm^a}f\big)\rmd\tau\,.
\end{equs}

\subsection{Cumulant analysis}\label{subsec:cumulant_analysis}
This section is dedicated to the cumulant analysis with the main resulting estimate being Lemma~\ref{lem:cumul} below. 
From now on and until the end of Section~\ref{sec:Sec4}, we fix $P\in2\N_{\geqslant1}$ and write $\eta=\frac{2+n/r}{P}$.

We start by stating a standard but useful cumulant identity. 
\begin{definition}
Fix a finite subset $J\subset \N_{\geqslant1}$. 
We denote by $\mcP(J)$ the set of all partitions of $J$. For $\rho\in \mcP(J)$, we write $|\rho|$ for the number of elements of $\rho$. We denote the elements of $\rho$ by $(\rho_q)_{q\in[|\rho|]}$: they are non-empty subsets of $J$, non-overlapping, and their union is $J$.
We adopt the convention of ordering these subsets by the order of their minima, writing $\rho=(\rho_q)_{q\in[|\rho|]}$ where $k<p\Rightarrow\min\rho_k<\min\rho_p$. 

Moreover, for two finite non-empty subsets $I,J \subset \N_{\geqslant1}$ we define 
\begin{equs}
    \mcQ(I,J)\eqdef\Big\{(\pi,\rho):\rho\in\mcP(J)\;\text{and}\;\pi:I\rightarrow[|\rho|]\Big\}\,,
\end{equs}
and we adopt the convention that if $I=\emptyset$, then we set $  \mcQ(I,J)=\mcP(J)$.

Finally, given $(\pi,\rho)\in\mcQ(I,J)$, for every $q\in[|\rho|]$, we use the shorthand $\pi_q\eqdef \pi^{-1}(q)$. The $\pi_q$'s are therefore some possibly empty subsets of $I$, non-overlapping, and whose union is $I$.
\end{definition}
We let $\kappa_{|I|}\big((X_i)_{i\in I}\big)$ denote the joint cumulant of the family of random variables $(X_i)_{i\in I}$. 
The promised cumulant identity can then be stated as follows. 
\begin{lemma}
Pick $I$ and $J$ two finite subsets of  $\N_{\geqslant1}$ ($I$ is possibly empty), and $(X_i)_{i\in I}$ and $(Y_j)_{j\in J}$ two families of random variables. Then we have
    \begin{equs}   \label{eq:relcum} \kappa_{|I|+1}\big((X_i)_{i\in I},\prod_{j\in J}Y_j\big)=\sum_{(\pi,\rho)\in\mcQ(I,J)}\prod_{k=1}^{|\rho|}\kappa_{|\pi_k|+|\rho_k|}\big((X_i)_{i\in \pi_k},(Y_j)_{j\in \rho_k}\big)\,.
\end{equs}
\end{lemma}
\begin{definition}
We introduce a set of \textit{lists of generalized multi-indices}
\begin{equs}
    \widehat\mfM\eqdef\{\bma=(\tilde{a}_1,\dots,\tilde{a}_{p(\text{\scriptsize{$\bma$}})})\in\widetilde{\mfM}^{p(\bma)}:p(\bma)\in[P]\}\,.
\end{equs}
Whenever $\text{$\bma$}=(\tilde{a}_1,\dots,\tilde{a}_{p(\text{\scriptsize{$\bma$}})})\in\widehat\mfM$, we always write $\tilde a_i=(a_i,\mfl^{a_i},s_i,t_i)$ for $i\in[p(\bma)]$ .\\ 
For two lists of generalized multi-indices $\text{$\bma$},\text{$\bmb$}\in\widehat\mfM$ and an ordered subset $I=(i_1,\dots,i_{|I|})$ of $[p(\bma)]$, we write
\begin{equs}
    \bma_I\eqdef\big(\tilde a_{i_1},\dots,\tilde a_{i_{|I|}}\big)\,,
\end{equs}
and we define the concatenation of $\bma$ and $\bmb$ that we write $\bma\sqcup\bmb$ by
\begin{equs}    \bma\sqcup\bmb\eqdef(\tilde{a}_1,\dots,\tilde{a}_{p(\bma)},\tilde{b}_1,\dots,\tilde{b}_{p(\bmb)})\,.
\end{equs}
Fix $\bma\in\widehat\mfM$. Given for every for $i\in[p(\bma)]$ space-time points $(x_{a_i},y^{a_i})\in\Lambda^{[a_i]+1}$, we write that
\begin{equs}
    x_{\bma}\eqdef (x_{a_1},\dots,x_{a_{p(\bma)}})  \,,\;\text{and}\;y^\bma\eqdef(y^{a_1},\dots,y^{a_{p(\bma)}})\,,
\end{equs}
belong respectively to $\Lambda^{[p(\bma)]}$ and $\Lambda^{[\bma]}$. 
We use the shorthand notation $\Lambda^{[\bmp(\bma)]}\eqdef\Lambda^{[p(\bma)]}\times\Lambda^{[\bma]}$ to write $(x_\bma,y^\bma)\in\Lambda^{[\bmp(\bma)]}$ and let $\bmp(\bma)\eqdef p(\bma)+|[\bma]|$.
\end{definition}
\begin{definition}
    In the next definition, we will introduce the cumulants of the generalized force coefficients, along with a topology in which we will control them. When $n\geqslant2$ or $\alpha>1/2$ in $n=1$, the covariance of the noise is integrable, which makes the $L^1$ norms used in \cite{Duch21, Duch22} unsuitable. Indeed, being integrable, the covariance of the noise should be irrelevant  when controlled in the $L^1$ norm and we would need it to vanish as $\mu \downarrow 0$, but this is hopeless since the noise itself is constant along the flow. 

To deal with this issue, we adapt how we control covariances to make them either slightly relevant (or marginal if the noise is white). Fixing a small $\iota>0$, we define the integrability index
    \begin{equs}
        r\equiv r(\alpha,n,\iota)\eqdef\begin{cases}
            1&\text{if }\; n=1 \;\text{and }\;\alpha\leqslant1/2\,,\\
           \frac{n}{2-2\alpha} (1+\iota)&\text{otherwise.}
        \end{cases}
    \end{equs}  
The parameter $r$ will the integrability index in space, we recall the shorthand notation $L^{r,1}_x=L_\ttx^rL^1_{x_0}$ to denote the $L^r$ in space and $L^1$ in time norm. Explicitly, 
\begin{equs}
    \norm{\psi}_{L^{r,1}}\eqdef \bigg(\int_{\T^n}\Big(\int_{(-\infty,1]}|\psi(x)|\rmd x_0 \Big)^r\rmd\ttx\bigg)^{1/r}\,.
\end{equs}
\end{definition}
\begin{definition}
For $\eps,\mu\in(0,1]$, $\text{$\bma$}\in\widehat\mfM$ and $(x_\bma,y^\bma)\in\Lambda^{[\bmp(\bma)]}$ we define the joint cumulant of the generalized force coefficients indexed by $\bma$ by
\begin{equ}
    \Kappa_{\eps,\mu}^{\bma}( x_\bma,y^\bma)\eqdef\kappa_{p(\bma)}\big((\xi_{\eps,\mu}^{\tilde{a}_i}(x_{a_i},y^{a_i}))_{i\in[p(\bma)]}\big)\,.
\end{equ}
Note that $  \Kappa_{\eps,\mu}^{\bma}$ takes values in
\begin{equs}
    \mcH^{\bma}\eqdef\bigotimes_{i\in[p(\bma)]}\mcH^{a_i}\,.
\end{equs} 
Moreover, writing for $N\geqslant1$
\begin{equs}
 K_{N,\mu}^{\otimes \bmp(\bma)}    &\Kappa_{\eps,\mu}^{\bma}( x_\bma,y^\bma)\eqdef (K_{N,\mu}\otimes\dots\otimes K_{N,\mu})*\Kappa_{\eps,\mu}^{\bma}( x_\bma,y^\bma)
\end{equs}
for the convolution of $\Kappa_{\eps,\mu}^{\bma}$ with $K_{N,\mu}$ at the level of all its arguments, we endow the cumulants of the stationary force coefficients with the norm
\begin{equs}    
{}&\nnorm{\Kappa_{\eps,\mu}^{\bma}}_N \eqdef\norm{K_{N,\mu}^{\otimes \bmp(\bma)}\Kappa_{\eps,\mu}^{\bma}}_{L^\infty_{x_{a_1}}L^r_{{\ttx}_{a_2}}\dots L^r_{{\ttx}_{a_{p(\bma)}}}L^1_{x_{a_2,0}}\dots L^1_{x_{a_{p(\bma)},0}}
L^1_{y^{\bma}}(\Lambda^{[\bmp(\bma)]})}\\
&\equiv \sup_{x_{a_1}\in\Lambda}\bigg(
\int_{(\T^n)^{p(\text{\tiny{$\bma$}})-1}}
\Big(\int_{\Lambda^{[\text{\tiny{$\bma$}}]}\times(-\infty,1]^{p(\bma)-1}}
| K_{N,\mu}^{\otimes \text{\scriptsize{$\bmp$}}(\text{\scriptsize{$\bma$}})}    \Kappa_{\eps,\mu}^{\text{\scriptsize{$\bma$}}}( x_{\text{\scriptsize{$\bma$}}},y^{\text{\scriptsize{$\bma$}}})|\text{$\rmd$} y^{\text{\scriptsize{$\bma$}}}\rmd x_{a_2,0}\dots\rmd x_{a_{p(\bma)},0}\Big)^r
\rmd {\ttx}_{a_2}\dots\rmd {\ttx}_{a_{p(\bma)}}\bigg)^{1/r}
\,.
\end{equs}
Here, $x_{a_i,0}$ denote the time component of $x_{a_i}$, so that $x_{a_i} = (x_{a_i,0},\ttx_{a_{i}})$. \\
Finally, the \textit{scaling} of $\Kappa_{\eps,\mu}^{\bma}$ is defined by
\begin{equs}    |\bma|\eqdef\sum_{i=1}^{p(\bma)}|a_i|+\sum_{i=1}^{p(\bma)}|\mfl^{a_i}|+\big(p(\bma)-1\big)\Big(2+\frac{n}{r}\Big)\,,\label{eq:defScalingCum}
\end{equs}
the \textit{order} of $\Kappa_{\eps,\mu}^{\bma}$ by $\mfo(\bma)\eqdef \sum_{i\in[p(\bma)]}\mfo(a_i) $, and we set $s(\bma)\eqdef\sum_{i\in[p(\bma)]}s_i$ and $t(\bma)\eqdef\sum_{i\in[p(\bma)]}t_i$. 
\end{definition}
\begin{remark}\label{rem:Covariance}
With respect to the scaling \eqref{eq:defScalingCum}, the only relevant cumulants are the expectations of the relevant generalized force coefficients, and the covariance of the noise (which is marginal only if $(n,\alpha)=(1,1/2)$, i.e. is the case of the space-time white noise).
    
If we took $n=1$ and $\alpha \leqslant 1/4$, then new divergent covariances of objects of higher order would appear for which the renormalization prescriptions here would be insufficient.  
\end{remark}
In order to control the cumulants of the stationary force coefficients, we need to establish three flow equations for cumulants, starting from \eqref{eq:flow2}. To do so, the following piece of notation is required.
\begin{definition}
Fix $\text{$\bma$}\in\widehat\mfM$ and $i\in[p(\bma)]$ (if $p(\bma)=1$ then automatically $i=1$). We define a new generalized multi-index $\hat a_i\in\widetilde{\mfM}$ by $$\hat a_i\eqdef(a_i,\mfl^{a_i},s_i,1\vee t_i)\in\widetilde{\mfM}_1\,.$$ Pick $(\sigma,\tilde b_1,\tilde b_2,\tilde\bfd)\in\mathrm{Ind}(\hat a_i)$, and two partitions $(\pi,\rho)\in\mcQ\big([p(\bma)]\setminus\{i\},[2]\big)$ with $|\pi|=|\rho|$. We write for $i\in[2]$ $\tilde b_i=(b_i,\mfl^{b_i},s'_i,0)$, we denote $\text{$\bmb$}\eqdef(\tilde b_1,\tilde b_2)\in\widehat\mfM$, and we introduce the shorthand notation $(\sigma,\text{$\bmb$},\tilde\bfd)\in\mathrm{Ind}(\hat a_i)$ to denote the fact that $\bmb=(\tilde b_1,\tilde b_2)$ and $(\sigma,\tilde b_1,\tilde b_2,\tilde\bfd)\in\mathrm{Ind}(\hat a_i)$.

For $k\in|\rho|$, we introduce a new list of enhanced multi-indices defined by
\begin{equs}
    \bmc_k\eqdef {\bma}_{\pi_k}\sqcup{\bmb}_{\rho_k}\,,
\end{equs}
with the understanding that $\bmc_k=\bmb_{\rho_k}$ if $p(\bma)=1$. When $|\rho|=1$, by convention, we write $$\bmc\equiv\bmc_1=\bma_{[p(\bma)]\setminus\{i\}}\sqcup \bmb\,.$$

Recall the notations of Definition~\ref{def:B}, and in particular the points $y^b$, $y^c$, $w$ and $z$ introduced therein. With this notation in hand, we write
\begin{equs}
    (x_{b_1},y^{b_1})\eqdef(x_{a_i},y^b)\in\Lambda^{[\bmp(\bmb_1)]}\,,\;\text{and}\;  (x_{b_2},y^{b_2})\eqdef(w,y^c)\in\Lambda^{[\bmp(\bmb_2)]}\,.
\end{equs}
We can finally introduce the operator entering the RHS of the flow equation:
\begin{equs}    \label{eq:defA}\text{A}_\mu&\Big(\big(\Kappa_{\eps,\mu}^{\bmc_k}\big)_{k\in|\rho|}\Big)(x_{\bma},y^\bma)\\&\eqdef\frac{b^{\text{\scriptsize{$\mfk$}}_0}_{1,k_0}}{a_i!}
    (1+\text{$\1$}\{(\text{{$\mfk$}}_0,\text{{$\mfk$}}_1)=(\text{{$\mfe$}},\text{{$\mff$}})\})\frac{s_i!}{s'_1!s'_2!}\,\int_\Lambda\big(\text{$\bfX$}^{\text{\scriptsize{$\mfl$}}^{\text{\tiny{$\bfd$}}}}\partial_{\ttx}^{k_0+1-k_1}\dot G_\mu\big)( z-w)
    \prod_{k\in|\rho|}\Kappa_{\eps,\mu}^{\bmc_k}(x_{\bmc_k},y^{\bmc_k})\text{$\rmd$} w\,.
\end{equs}
\end{definition}
We read from its definition that the operator $\A$ verifies the following.
\begin{lemma}\label{lem:boundA}
Fix a partition $\rho\in\mcP([2])$, and introduce the shorthand notation $${\tt I}(\rho)\eqdef \begin{cases}  
    (r,1) & \text{if} \;|\rho|= 1\,, \\
       \infty & \text{if} \;|\rho|= 2\,. \\
\end{cases} $$
Moreover, fix $N\geqslant1$, $\text{$\bma$}\in\widehat\mfM$, $i\in[p(\bma)]$, and $(\sigma,\text{$\bmb$},\tilde\bfd)\in\mathrm{Ind}(\hat a_i)$. Then, for any collection $(\psi_k)_{k\in|\rho|}$ of functions such that, for every $k\in|\rho|$, $\psi_k\in\mcD(\Lambda^{\bmp[\bmc_k]})$, it holds 
    \begin{equs}        \nnorm{\A_\mu\big((\psi_k)_{k\in|\rho|}\big)}_N\lesssim\norm{\mcP^{2N}_\mu\text{$\bfX$}^{\text{\scriptsize{$\mfl$}}^{\text{\tiny{$\bfd$}}}}\partial^{k_0+1-k_1}_{\ttx}\dot G_\mu}_{\mcL^{{\tt I}(\rho),\infty}}\prod_{k\in|\rho|}\nnorm{\psi_k}_N
\end{equs}
uniformly in $\mu\in(0,1]$.
\end{lemma}
\begin{proof}
Recall the notations $(z,y^b,w,y^c)$ in the definition of $\A$. The difference between the cases $|\rho|=1$ and $|\rho|=2$ is that in the latter case, since there are two terms in the product over $k$, we can indeed take the $L^\infty$ norm in $w$ while, in the former case, we do need the $L^{r,1}$ norm in $w$, which forces us to take the $\mcL^{(r,1),\infty}$ norm of $\dot G_\nu$. 

The situation is slightly different from \cite{Duch21,Duch22}, since the fact that the $L^\infty$, $L^r$ and $L^1$ norms constituting $\nnorm{\bigcdot}_N$ do not commute might at first seem worrying. However, this is handled by the fact that we have taken care to apply first all the $L^1$ norms, then all the $L^r$ norms, and finally the $L^\infty$ norm. We provide more details in the case $|\rho|=1$, the case $|\rho|=2$ being similar, but closer to the context of \cite{Duch21,Duch22}. 

To lighten the notation, we write $\psi=\psi_1$, $G=\text{$\bfX$}^{\text{\scriptsize{$\mfl$}}^{\text{\tiny{$\bfd$}}}}\partial_{\ttx}^{k_0+1-k_1}\dot G_\mu$, and denote by $\tilde y^\bma$ the collection of all $y^\bma_{\mfk ij}$ apart from $z$, and by $\tilde x_\bma$ the collection of all $x_{a_i}$ apart from $x_{a_1}$. For simplicity, we assume $i\ne1$. With this shorthand notation, $\A_\mu(\psi)(x_\bma,y^\bma)$ rewrites as 
\begin{equs}
    \A_\mu(\psi)(x_\bma,y^\bma)=C \int_\Lambda G(z-w)\psi\big((x_\bma,w),(\tilde y^\bma,z)\big)\rmd w
\end{equs}
for an inessential constant $C\in\R$. As in the proof of \eqref{eq:propB}, we can insert by hand the kernel $K_{N,\mu}$ is the variable $z$ in front of $\psi$. With the same argument as in the proof of \eqref{eq:propB}, we have the existence of some kernels $\big(A_\mu^{(\mfl)}:\mfl\in \N^{n+1},|\mfl|\leqslant 4N\big)$ belonging to $L^1$ uniformly in $\mu$ such that 
    \begin{equs}        {}& K_{N,\mu}^{\otimes\bmp(\bma)}  \A_\mu(\psi)(x_\bma,y^\bma)\\
&=C\int_\Lambda  \sum_{\text{\scriptsize{$\mfl$}}\in \N^{n+1},|\text{\scriptsize{$\mfl$}}|\leqslant 4N} A_\mu^{(\text{\scriptsize{$\mfl$}})}(z-v)\int_\Lambda
\mu^{|\text{\scriptsize{$\mfl$}}|}\partial^{\text{\scriptsize{$\mfl$}}}\mcP^N_\mu G( v-w)
K_{N,\mu}^{\otimes\bmp(\bmc)}\psi\big((x_\bma,w),(\tilde y^\bma,v)\big)\rmd w  \rmd v\,.
    \end{equs}
Write $\tilde\psi\eqdef K_{N,\mu}^{\otimes\bmp(\bmc)}\psi$ and $\tilde G=\mu^{|\text{\scriptsize{$\mfl$}}|}\partial^{\text{\scriptsize{$\mfl$}}}\mcP^N_\mu G$. Using H\"older's inequality to remove the kernels $A_\mu^{(\mfl)}$, and then using H\"older's inequality firstly in time and then in space, we have
\begin{equs}
   \textcolor{white}{a} &\nnorm{  \A_\mu(\psi)}_N  
    \\&=\norm{  K_{N,\mu}^{\otimes\bmp(\bma)}  \A_\mu(\psi)}_{L^\infty_{x_{a_1}}L^r_{\tilde \ttx_\bma}
    L^1_{\tilde x_{\bma,0}}L^1_{y^\bma}}\\
&\lesssim\sup_{x_1\in\Lambda}\bigg(\int_{(\T^n)^{p(\bma)-1}}\Big(\int_{\R_{\leqslant1}^{p(\bma)-1}}\int_{\Lambda^{[\bma]}}\Big|\int_\Lambda\tilde G(z-w)\tilde\psi\big((x_\bma,w),(\tilde y^\bma,z)\big)\rmd w\Big|\rmd y^\bma\rmd \tilde x_{\bma,0}\Big)^r\rmd\tilde\ttx_\bma\bigg)^{1/r}\\
&\lesssim
\sup_{x_1\in\Lambda}\bigg(\int_{(\T^n)^{p(\bma)-1}}\Big(\int_{\T^n}\int_{\R_{\leqslant1}}\int_{\R_{\leqslant1}^{p(\bma)-1}}\int_{\Lambda^{[\bma]}}|\tilde G(z-w)\tilde\psi\big((x_\bma,w),(\tilde y^\bma,z)\big)|\rmd y^\bma\rmd \tilde x_{\bma,0}\rmd w_0\rmd{\tt w}\Big)^r\rmd\tilde\ttx_\bma\bigg)^{1/r}\\
&\lesssim
\sup_{x_1\in\Lambda}\bigg(\int_{(\T^n)^{p(\bma)-1}}\Big(\int_{\T^n}\int_{\R_{\leqslant1}}\int_{\R_{\leqslant1}^{p(\bma)-1}}\int_{\Lambda^{[\bma]}}\sup_{u_0\in\R_{\leqslant1}}|\tilde G({\tt z-w},u_0)||\tilde\psi\big((x_\bma,w),(\tilde y^\bma,z)\big)|\rmd y^\bma\rmd \tilde x_{\bma,0}\rmd w_0\rmd{\tt w}\Big)^r\rmd\tilde\ttx_\bma\bigg)^{1/r}\\
&\lesssim\norm{\tilde G}_{\mcL^{(r,1),\infty}}
\sup_{x_1\in\Lambda}\bigg(\int_{(\T^n)^{p(\bma)-1}}\int_{\T^n}\Big(\int_{\R_{\leqslant1}}\int_{\R_{\leqslant1}^{p(\bma)-1}}\int_{\Lambda^{[\bma]}}|\tilde\psi\big((x_\bma,w),(\tilde y^\bma,z)\big)|\rmd y^\bma\rmd \tilde x_{\bma,0}\rmd w_0\Big)^r\rmd{\tt w}\rmd\tilde\ttx_\bma\bigg)^{1/r}\\
&\lesssim\norm{\mcP^{2N}_\mu G}_{\mcL^{(r,1),\infty}}\nnorm{\psi}_N
\,,
\end{equs}
which is the desired result. Note that we have made use of the translation invariance of $\tilde G$, which implies that for all $z\in\Lambda$
\begin{equs}
\int_{\T^n} \Big(   \sup_{w_0\in(-\infty,1]}|\tilde G({\tt z-w},w_0-z_0)|\Big)^{r'}\rmd{\tt w}=\int_{\T^n} \Big(   \sup_{u_0\in(-\infty,1]}|\tilde G({\tt u},u_0)|\Big)^{r'}\rmd{\tt u}=\norm{\tilde G}_{\mcL^{(r,1),\infty}}^{r'}\,,
\end{equs}
where $r'=r/(r-1)$ stands for the H\"older conjugate to $r$.
\end{proof}
We are now ready to present the flow equations of the cumulants.
\begin{lemma}
   First, we consider $\text{$\bma$}\in\widehat\mfM$ such that $t(\bma)\geqslant1$, $\mfo(\bma)\geqslant1$, and let $i\eqdef\min\{j\in[p(\bma)]:t_j=1\}$. We define
\begin{equs}
    {\mathrm{Ind}}(\text{${\mathBM a}$})\eqdef\Big\{&(\sigma,\text{${\mathBM b}$},\text{$\tilde{\bfd}$},\pi,\rho)\in\text{$\mathrm{Ind}$}(\hat{a}_i)\times\mcQ\big([p(\text{${\mathBM a}$})]\setminus\{i\}, [2]\big)\Big\}\,,
\end{equs}
the index set for the flow of the cumulant $\Kappa^{\bma}_{\eps,\mu}$. Then, we have the following flow equation for the cumulants containing at least one $\mu$ derivative:
\begin{equs}\label{eq:flowcumuldmu}
\Kappa_{\eps,\mu}^{\bma}=&-  \sum_{(\sigma,\text{\scriptsize{$\bmb,\tilde\bfd$}},\pi,\rho)\in\text{\scriptsize{$\mathrm{Ind}$}}(\text{\scriptsize{$\bma$}})     }  \text{$\mathrm A$}_\mu\Big(\big(\Kappa_{\eps,\mu}^{\bmc_k}\big)_{k\in|\rho|}\Big)
       \,.
\end{equs}
Next, we turn to $\text{$\bma$}\in\widehat\mfM$ such that $t(\bma)=0$, $\mfo(\bma)\geqslant1$, and $|\bma|>0$. We define
\begin{equs}
     {\mathrm{Ind}}(\text{${\mathBM a}$})\eqdef\Big\{&(i,\sigma,\text{${\mathBM b}$},\text{$\tilde{\bfd}$},\pi,\rho)\in[p(\text{${\mathBM a}$})]\times \text{$\mathrm{Ind}$}(\hat{a}_i)  \times\mcQ\big([p(\text{${\mathBM a}$})]\setminus\{i\}, [2]\big)\Big\}\,,
\end{equs}
the index set for the flow of the cumulant $\Kappa^{\bma}_{\eps,\mu}$. Then, we have the following flow equation for the irrelevant cumulants
\begin{equs}\label{eq:floweqcumulIrrel}
\Kappa_{\eps,\mu}^{\bma}=&-  \int_0^\mu\sum_{(i,\sigma,\text{\scriptsize{$\bmb,\tilde\bfd$}},\pi,\rho)\in\text{\scriptsize{$\mathrm{Ind}$}}(\text{\scriptsize{$\bma$}})     } \text{$\mathrm A$}_\nu\Big(\big(\Kappa_{\eps,\nu}^{\text{\scriptsize{$\bmc_k$}}}\big)_{k\in|\rho|}\Big)\rmd\nu
       \,.
\end{equs}
Finally, we deal with $\text{$\bma$}\in\widehat\mfM$ such that $t(\bma)=0$, $\mfo(\bma)\geqslant1$, and $|\bma|\leqslant0$ (in particular we have $p(\bma)=1$, and we write $\bma=(\tilde a)$ where $\tilde a= (a,\mfl^a,s,0)$). We define
 \begin{equs}
     {\mathrm{Ind}}(\text{${\mathBM a}$})\eqdef\Big\{&(\sigma,\text{${\mathBM b}$},\text{${\tilde\bfd}$},\rho)\in \text{$\mathrm{Ind}$}(\hat{a})\times\mcP( [2]) \Big\}\,,
\end{equs}
the index set for the flow of $\E[\bfI^a\big(\xi^{\text{\scriptsize{$\tilde a$}}}_{\eps,\mu}\big)]=\E[\bfI^a\big(\bfX^{\mfl^a}\partial_\eps^s\xi^a_{\eps,\mu}\big)]$. Then, we have the following flow equation for the local part of the relevant expectations
\begin{equs}
\E[\text{$\bfI$}^a\big(\xi^{\text{\scriptsize{$\tilde a$}}}_{\eps,\mu}\big)]&=\partial_\eps^s 
  {\text{$\mfc$}}^{a,{\text{\scriptsize{$\mfl$}}}^a}_\eps-\int_0^\mu\sum_{(\sigma,\text{\scriptsize{$\bmb,\tilde\bfd,\rho$}})\in\text{\scriptsize{$\mathrm{Ind}$}}(\text{\scriptsize{$\bma$}})
     }     \text{$\bfI$}^a    \text{$\mathrm A$}_\nu\Big(\big(\Kappa_{\eps,\nu}^{\text{\scriptsize{$\bmc_k$}}}\big)_{k\in|\rho|}\Big)\rmd\nu  \\
       &= \int^1_\mu \sum_{(\sigma,\text{\scriptsize{$\bmb,\tilde\bfd,\rho$}})\in\text{\scriptsize{$\mathrm{Ind}$}}(\text{\scriptsize{$\bma$}})
     }               
    \text{$\bfI$}^a \text{$\mathrm A$}_\nu\Big(\big(\Kappa_{\eps,\nu}^{\text{\scriptsize{$\bmc_k$}}}\big)_{k\in|\rho|}\Big)\rmd\nu\bigg)\,,\label{eq:flowcumulRel}
\end{equs}
where the counterterm $\mfc^{a,\mfl^a}_\eps$ is chosen such that $\xi^{(a,\mfl^a,0,0)}_{\eps,0}=\delta^a\mfc^{a,\mfl^a}_\eps$, and is defined as
\begin{equs}
    {\text{$\mathfrak c$}}^{a,{\text{\scriptsize{$\mfl$}}}^a}_\eps\equiv  {\text{$\mathfrak c$}}^{a,{\text{\scriptsize{$\mfl$}}}^a}_\eps(x)\eqdef
\sum_{(\sigma,\text{\scriptsize{$\bmb,\tilde\bfd,\rho$}})\in\text{\scriptsize{$\mathrm{Ind}$}}(\text{\scriptsize{$(a,\mfl^a,0,1)$}})
     }    \int_0^1\text{$ \bfI$}^a      \text{$\mathrm A$}_\nu\Big(\big(\Kappa_{\eps,\nu}^{\text{\scriptsize{$\bmc_k$}}}\big)_{k\in|\rho|}\Big)(x)\rmd\nu\,.
\end{equs}
Note that $\mfc^{a,\mfl^a}_\eps$ is local and that, by stationarity, it is independent of $x$.
\end{lemma}
\begin{proof}
    These flow equations directly stem from the combination of \eqref{eq:flowGen} with \eqref{eq:relcum}.
\end{proof}
\begin{remark}
   The flow equations of the cumulants are hierarchical in the order of the cumulants, in the sense that we have $\mfo(\bmc_k)\leqslant\mfo(\bma)-1$ for every $k\in|\rho|$.
\end{remark}
\begin{remark}
We have the following expression of the counterterm for $x\in\Lambda_{0;1}$ as a local functional: 
    \begin{equs}\label{eq:counterterm}    \mfc_\eps[\psi_\eps](x)=\mfc_\eps(\psi_\eps(x),\partial_{\ttx}\psi_\eps(x))=\sum_{\substack{(a,\mfl^a)\in\mfM\\|a|+|\mfl^a|\leqslant0} }
    \langle \mfc^{a,\mfl^a}_\eps ,\partial_{\ttx}^{\mfl^a}\Upsilon^a[\psi](x^{[a]})\rangle_{\mcH^a}\,,
\end{equs}
where we denote by $x^{[a]}$ the element $(x,\dots,x)\in\Lambda_{0;1}^{[a]}$ with all entries equal to $x$.
\end{remark}
We are now ready to make use of the flow equations for cumulants to construct all the cumulants inductively. We start by the following lemma that deals with the base case.
\begin{lemma}\label{lem:basecaseCum}
Recall that the noise is contained in the effective force coefficient $\xi_\epmu^{\1_0^\mfh}(x,y)=\xi_\eps(y)\delta(x-y)$. In particular, $\xi_\epmu^{\1_0^\mfh}$ is constant along the flow, and its only non-vanishing cumulant is its covariance $\Kappa_\epmu^\bma$ for $\text{$\bma$}=\big((\1_0^\mfh,0,s_1,0),(\1_0^\mfh,0,s_2,0)\big)$, for which there exists $N_2\geqslant1$ and $c^\star>0$ such that for all $c\in(0,c^\star]$ it verifies
   \begin{equs}   \nnorm{\Kappa_{\eps,\mu}^{\bma}}_{N_2}\lesssim \eps^{s(\bma)(-1+c)}\mu^{|\bma|-s(\bma)c-t(\bma)}
    \end{equs}
uniformly in $(\eps,\mu)\in(0,1]$.
\end{lemma}
\begin{proof}
First, observe that
\begin{equs}
   \Kappa^{\bma}_\epmu(x_1,x_2,y_1,y_2)= \delta(x_1-y_2)\delta(x_2-y_2)\Cov_\eps^{s_1,s_2}(y_1-y_2)\,,
\end{equs}
where $\Cov_\eps^{s_1,s_2}\eqdef\d_\eps^{s_1}\mcS_\eps\rho*\d_\eps^{s_2}\mcS_\eps\rho*\Cov$, which implies that, using the relation $\Id=\mcP_\mu^{N_2}K_{N_2,\mu}$,
\begin{equs}   K^{\otimes4}_{N_2,\mu}&\Kappa^{\bma}_\epmu(x_1,x_2,y_1,y_2)
\\&=
\int_{\Lambda^2}K_{N_2,\mu}(y_1-r_1)K_{N_2,\mu}(y_2-r_2) K_{N_2,\mu}(x_1-r_1)K_{N_2,\mu}(x_2-r_2)\Cov_\eps^{s_1,s_2}(r_1-r_2)\rmd r_1\rmd r_2
\\&=
\int_{\Lambda^2}
\big(\mcP_\mu^{N_2}\big)^\dagger\big(K_{N_2,\mu}(y_1-\bigcdot)K_{N_2,\mu}(x_1-\bigcdot)\big)(r_1)\big(\mcP_\mu^{N_2}\big)^\dagger\big(K_{N_2,\mu}(y_2-\bigcdot) K_{N_2,\mu}(x_2-\bigcdot)\big)(r_2)
\\&\quad\times
\big(K_{N_2,\mu}*\Cov_\eps^{s_1,s_2}*K_{N_2,\mu}\big)(r_1-r_2)
\rmd r_1\rmd r_2
\,,
\end{equs}
Therefore, using Hölder's inequality and the definition of the norm $\nnorm{\bigcdot}_{N_2}$,
\begin{equs}  \label{eq:stepinproofcovnoise} {}&\nnorm{\Kappa^{\bma}_\epmu}_{N_2}
\\&\lesssim
\norm{
\big(\mcP_\mu^{N_2}\big)^\dagger\big(K_{N_2,\mu}(y_1-\bigcdot)K_{N_2,\mu}(x_1-\bigcdot)\big)(r_1)
}_{L^\infty_{x_1}L^1_{y_1}L^1_{r_1}}\norm{\big(\mcP_\mu^{N_2}\big)^\dagger\big(K_{N_2,\mu}(y_2-\bigcdot) K_{N_2,\mu}(x_2-\bigcdot)\big)(r_2)
}_{L^{\infty}_{r_{2,0}}L^{r'}_{{\tt r}_2}L^1_{x_2}L^1_{y_2}}
\\&\quad\times\norm{K_{N_2,\mu}*\Cov_\eps^{s_1,s_2}*K_{N_2,\mu}(r_1-r_2)
}_{L^\infty_{r_1}L^{r,1}_{r_2}}\,,
\end{equs}
where $r'$ is Hölder conjugate to $r$.

We first study the first two terms of the RHS of \eqref{eq:stepinproofcovnoise}, that are very similar once, for the second term, the $L^{r'}$ norm in ${\tt r}_2$ is bounded by an $L^\infty$ norm. The aim is to show that this two terms are bounded  uniformly in $\mu$. It turns out that when the operator $\big(\mcP_\mu^{N_2}\big)^\dagger$ hits one of the kernels $K_{N_2,\mu}$, this possibly creates some space-time derivatives of the kernel $K_{N_2,\mu}$ multiplied by $\mu$. Then, by \eqref{eq:spacederivKmu}, the newly created kernel is also in $L^1$. Overall, this means that there exists a finite set $I$ and some kernels $\big(A_\mu^{(i)},B_\mu^{(i)}:i\in I\big)$ belonging to $L^1$ uniformly in $\mu$ such that 
\begin{equs}
    \big(\mcP_\mu^{N_2}\big)^\dagger\big(K_{N_2,\mu}(y-\bigcdot)K_{N_2,\mu}(x-\bigcdot)\big)(r)
=\sum_{i\in I} A_\mu^{(i)}(y-r)B_\mu^{(i)}(x-r)\,.
\end{equs}
By translation invariance, this kernel really depends on two variables of the three variables $x$, $y$ and $r$, and by taking the $L^1$ norm in two of the variables, we can indeed conclude that the first two terms are bounded uniformly in $\mu\in(0,1]$.

We turn turn to the last term of the RHS of \eqref{eq:stepinproofcovnoise}, that is to say the norm of $K_{N_2,\mu}*\partial^{s_1}_\eps\mcS_\eps\rho*K_{N_2,\mu}*\partial^{s_2}_\eps\mcS_\eps\rho*\Cov
$. We first deal with the case of a noise which is not white (i.e. $(n,\alpha)\neq(1,1/2)$). 
By translation invariance, the supremum over $r_1$ can be eliminated. Moreover, Young's inequality for convolution implies that
\begin{equs}
\nnorm{\Kappa^\bma_\epmu}_{N_2}  
&\lesssim
\norm{ K_{N_2^{1/2},\mu}*\partial^{s_1}_\eps\mcS_\eps\rho}_{L^{1}}
\norm{ K_{N_2^{1/2},\mu}*\partial^{s_2}_\eps\mcS_\eps\rho}_{L^{1}}
\norm{ K_{N_2,\mu}*\Cov}_{L^{r,1}}
\,.
\end{equs}
To obtain this last inequality, we have redistributed the powers of $Q_\mu$ to put a $K_{N_2,\mu}$ in front of the covariance.
Noting that $\partial^2_\eps\mcS_\eps\rho=\eps^{-1}\partial_\eps\mcS_\eps\tilde\rho$ for some other smooth compactly supported function $\tilde\rho$
and using \eqref{eq:KmuRhoeps}, we first obtain
\begin{equs}
    \norm{ K_{N_2^{1/2},\mu}*\partial^{s_1}_\eps\mcS_\eps\rho}_{L^{1}}
\norm{ K_{N_2^{1/2},\mu}*\partial^{s_2}_\eps\mcS_\eps\rho}_{L^{1}}\lesssim\eps^{s(\bma)(-1+c)}\mu^{-s(\bma)c}\,.
\end{equs}
Then, note that the short scale properties of pseudo differential operators imply that the kernel of $(1-\Delta)^{1-n/2-\alpha}$ is bounded by $|\ttx|^{-(2-2\alpha)}$ on short scales. With this observation in hand, we have
\begin{equs}
\norm{ K_{N_2,\mu}*\Cov}_{L^{r,1}}&\lesssim\bigg(\int_{\T^n}\Big(\mu^{-n}\int_{\T^n} (1-\Delta)^{-N_2}\big(({\tt y-x})/{\mu}\big)(1-\Delta)^{1-n/2-\alpha}({\tt y})\rmd { \tt y}\Big)^r\rmd \ttx\bigg)^{1/r}\\
&\lesssim\bigg(\mu^{n}\int_{\T^n}\Big(\int_{\T^n} (1-\Delta)^{-N_2}({\tt y-x})(1-\Delta)^{1-n/2-\alpha}(\mu {\tt y})\rmd {\tt y}\Big)^r\rmd \ttx\bigg)^{1/r}\\
&\lesssim\mu^{-(2-2\alpha)+n/r}=\mu^{2(-2+\alpha)+2+n/r}=\mu^{|\bma|}
\end{equs}
uniformly in $\mu\in(0,1]$. Here, note that the fact that $-(2-2\alpha)+n/r<0$ is crucial, since otherwise the above quantity would be bounded by one. Here, we would also like to draw the reader's attention on the fact that we \textit{do} cover the case $n=1$ and $\alpha<1/2$, in which $(1-\Delta)^{1-n/2-\alpha}$ is \textit{not} in $L^1$. This requires $N_2$ to be large enough so that $(1-\Delta)^{-N_2}$ can be convolved with $(1-\Delta)^{1-n/2-\alpha}$ and yield a bounded kernel: in particular, we need the former to be better that $L^\infty$, so that $N_2=1/2 $ would not sufficient. However, in the regime we are interested in, that is to say $\alpha>1/4$, it suffices to take $N_2\geqslant3/4$, so that in practice at this step we can set $N_2=1$. 

The case of the white noise ($(n,\alpha)=(1,1/2)$) is simpler: here $r=1$ and $\Cov=\delta$, so that we end up with 
\begin{equs}    \nnorm{\Kappa^\bma_\epmu}_{N_2}&\lesssim\norm{ \big(K_{N_2,\mu}*\partial^{s_1}_\eps\mcS_\eps\rho*K_{N_2,\mu}*\partial^{s_2}_\eps\mcS_\eps\rho*\Cov\big(x_1-x_2)}_{L_{x_1}^{\infty}L^1_{x_2}}\\&\lesssim
\norm{ K_{N_2,\mu}*\partial^{s_1}_\eps\mcS_\eps\rho*K_{N_2,\mu}*\partial^{s_2}_\eps\mcS_\eps\rho}_{L^1}\lesssim \eps^{s(\bma)(-1+c)}\mu^{-s(\bma)c}
\,.
\end{equs}
\end{proof}
We can now combine the base case with the flow equations to propagate the estimates.
\begin{lemma}\label{lem:cumul}
There exist $N_2\geqslant1$ and $c^\star>0$ such that for all $\text{$\bma$}\in\widehat\mfM$ and $c\in(0,c^\star]$, it holds
    \begin{equs}   \label{eq:Boundcum}\nnorm{\Kappa_{\eps,\mu}^{\bma}}_{N_2^{\text{\tiny{$\Gamma$}}+1}}\lesssim \eps^{s(\bma)(-1+c)}\mu^{|\bma|-s(\bma)c-t(\bma)}
    \end{equs}
uniformly in $(\eps,\mu)\in(0,1]$.
\end{lemma}
\begin{proof}
Again, we argue by induction on the order of $\bma$, the base case being handled by Lemma~\ref{lem:basecaseCum}. Let us now deal with the induction step. 

In this proof, we pay a particular attention to the power of $N$ in the norm $\nnorm{\bigcdot}_{N}$ necessary to control the cumulants. The idea is that for the first $\Gamma$ steps of induction, we have to deal with some relevant cumulants, which forces us to lose a factor $N_2$ at each step. We therefore first work with $N=N_2^{\text{\scriptsize{$\mfo$}}(\text{\scriptsize{$\bma$}})+1}$. Afterwards, no loss is necessary any more, so we can take $N=N_2^{\Gamma+1}$. We thus deal separately with these two cases. 

The case of cumulants with $\text{$\mfo$}(\text{$\bma$})\leqslant\Gamma$ and $t(\bma)\geqslant1$ or $|\bma|>0$ is easily dealt with using equations \eqref{eq:flowcumuldmu} and \eqref{eq:floweqcumulIrrel}, and the property of the operator $\A$ stated in Lemma~\ref{lem:boundA}. 
For example, using \eqref{eq:Kmunu}, we have in the latter situation
\begin{equs}
\nnorm{\Kappa_{\eps,\mu}^{\bma}}_{N_2^{\mfo(\bma)}}\lesssim  &\int_0^\mu\sum_{\substack{(i,\sigma,\text{\scriptsize{$\bmb,\tilde\bfd$}},\pi,\rho)\in\text{\scriptsize{$\mathrm{Ind}$}}(\text{\scriptsize{$\bma$}})  } } \label{eq:inter4_15}\norm{\mcP^{N_2^{\text{\tiny{$\mfo(\bma)$}}}}_\nu\text{$\bfX^{\mfl^\bfd}$}\partial^{k_0+1-k_1}_{\ttx}\dot G_\nu}_{\mcL^{{\tt I}(\rho),\infty}}\prod_{k\in|\rho|}\nnorm{\Kappa^{\bmc_k}_\epnu}_{N_2^{\mfo(\bmc_k)+1}} \rmd\nu\,.\textcolor{white}{blablaa}
\end{equs}
Note that we took care to evaluate the irrelevant cumulant in the norm $\nnorm{\bigcdot}_{N_2^{\text{\tiny{$\mfo(\bma$}})}}$ and not $\nnorm{\bigcdot}_{N_2^{\text{\tiny{$\mfo(\bma$}})+1}}$ (or $\nnorm{\bigcdot}_{N_2^{\text{\tiny{$\Gamma$}}+1}}$): keeping some room will turn out to be useful when dealing with relevant cumulants. At this stage, the desired result now follows using the induction hypothesis and \eqref{eq:heat1}. We end up with 
\begin{equs}    \nnorm{\Kappa_{\eps,\mu}^{\bma}}_{N_2^{\mfo(\bma)}}\lesssim &\eps^{s(\bma)(-1+c)} \int_0^\mu\sum_{\substack{(i,\sigma,\text{\scriptsize{$\bmb,\tilde\bfd$}},\pi,\rho)\in\text{\scriptsize{$\mathrm{Ind}$}}(\text{\scriptsize{$\bma$}})    } }\\ &\Big(\nu^{\sum_{j\neq i}|a_j|+\sum_{j\neq i}|\text{\scriptsize{$\mfl$}}^{a_j}|+|b_1|+|\text{\scriptsize{$\mfl$}}^{b_1}|+|b_2|+|\text{\scriptsize{$\mfl$}}^{b_2}|+1+\text{\scriptsize{$|\mfl^\bfd|$}}-(k_0+1-k_1)+(p(\bma)-1)(2+n/r)-s(\bma)c}\\
    &+\nu^{\sum_{j\neq i}|a_j|+\sum_{j\neq i}|\text{\scriptsize{$\mfl$}}^{a_j}|+|b_1|+|\text{\scriptsize{$\mfl$}}^{b_1}|+|b_2|+|\text{\scriptsize{$\mfl$}}^{b_2}|-1-n/r+\text{\scriptsize{$|\mfl^\bfd|$}}-(k_0+1-k_1)+p(\bma)(2+n/r)-s(\bma)c}\Big)\rmd\nu\\
    &\lesssim  \eps^{s(\bma)(-1+c)}  \int_0^\mu\nu^{|\bma|-1-s(\bma)c}\rmd\nu\lesssim\eps^{s(\bma)(-1+c)}\mu^{|\bma|-s(\bma)c}\,.
\end{equs}
In the last step, we have used the fact that the integral is convergent for $c$ small enough, which yields the desired result. 

When $\text{$\mfo$}(\text{$\bma$})>\Gamma$, we deal with the cumulants with $t(\bma)\geqslant1$ or $|\bma|>0$ in the exact same way, but using the norm $\nnorm{\bigcdot}_{N_2^{\Gamma+1}}$, and the fact that we do not need to lose any power of $N_2$ in the kernels $K_{N,\mu}$.

The case of a relevant expectation is more subtle, since we only have a flow equation for $\E[\text{$\bfI$}^a\big(\xi^{\text{\scriptsize{$\tilde a$}}}_{\eps,\mu}\big)]$ where $\tilde a=(a,\mfl^a,s,0)$ and $\mfo(a)\leqslant\Gamma$. First, observe that by stationarity, $\E[\text{$\bfI$}^a\big(\xi^{\text{\scriptsize{$\tilde a$}}}_{\eps,\mu}\big)]$ is constant as a function of space-time, which combined with \eqref{eq:normKmu} implies that $$\E[\text{$\bfI$}^a\big(\xi^{\text{\scriptsize{$\tilde a$}}}_{\eps,\mu}\big)]=K_{N_2^{\text{\tiny{$\mfo(a)$}}},\mu}\E[\text{$\bfI$}^a\big(\xi^{\text{\scriptsize{$\tilde a$}}}_{\eps,\mu}\big)]\,.$$ 
Moreover, the definition of $\bfI^a$ shows that we have the freedom to insert all the other needed convolutions with $K_{N_2^{\text{\tiny{$\mfo(a)$}}},\nu}$ to obtain, starting from \eqref{eq:flowcumulRel} and proceeding as in the irrelevant case 
\begin{equs}
|\E[\text{$\bfI$}^a\big(\xi^{\text{\scriptsize{$\tilde a$}}}_{\eps,\mu}\big)]|&\lesssim\int_\mu^1\sum_{\substack{(\sigma,\text{\scriptsize{$\bmb$}},   
    \text{\scriptsize{$\tilde\bfd$}}    ,\rho)\in\text{\scriptsize{$\mathrm{Ind}$}}(\text{\scriptsize{$\bma$}})   } } \nnorm{  \text{$\mathrm A$}_\nu\Big(\big(\Kappa_{\eps,\nu}^{\text{\scriptsize{$\bmc_k$}}}\big)_{k\in|\rho|}\Big)}_{N_2^{\text{\tiny{$\mfo(a)$}}}}\rmd\nu\\&\lesssim\int_\mu^1\sum_{\substack{(\sigma,\text{\scriptsize{$\bmb,\tilde\bfd$}},\rho)\in\text{\scriptsize{$\mathrm{Ind}$}}(\text{\scriptsize{$\bma$}})   } } \norm{\mcP^{N_2^{\text{\tiny{$\mfo(a)$}}}}_\nu\partial^{k_0+1-k_1}_{\ttx}\dot G_\nu}_{\mcL^{{\tt I}(\rho),\infty}}\prod_{k\in|\rho|}\nnorm{\Kappa^{\bmc_k}_\epnu}_{N_2^{\text{\tiny{$\mfo(\bmc_k)$}}+1}}\rmd\nu\\
    &\lesssim\eps^{s(-1+c)}\int_\mu^1\nu^{|a|+|\text{\scriptsize{$\mfl$}}^a|-1-sc}\rmd\nu\lesssim\eps^{s(-1+c)}\mu^{|a|+|\text{\scriptsize{$\mfl$}}^a|-sc}\,.\label{eq:intermediateStoBound}
\end{equs}
To go from $\E[\text{$\bfI$}^a\big(\xi^{\text{\scriptsize{$\tilde a$}}}_{\eps,\mu}\big)]$ to $\E[\xi^{\text{\scriptsize{$\tilde a$}}}_{\eps,\mu}]$, observe that by applying \eqref{eq:localization} to $\E[\xi^{\text{\scriptsize{$\tilde a$}}}_{\eps,\mu}]$ taking the smallest $\ell$ such that $|a|+|\text{{$\mfl$}}^a|+\ell>0$ ($\ell\in\{1,2\}$), we can re-express it in terms of some irrelevant expectations already constructed of the form $\E[\xi^{\text{\scriptsize{$\tilde a'$}}}_{\eps,\mu}]$ for $\tilde a'=(a,\mfl^a+\mfm^a,s,0)$ and of some $\E[\text{$\bfI$}^a\big(\xi^{\text{\scriptsize{$\tilde a'$}}}_{\eps,\mu}\big)]$ for $\tilde a'=(a,\mfl^a+\mfm^a,s,0)$ equally already constructed in view of the previous analysis. Indeed, for $\bma=(\tilde a)$, with this notation, we have
\begin{equs}  
\nnorm{\Kappa^{\text{\scriptsize{$\bma$}}}_\epmu}_{N_2^{\text{\tiny{$\mfo(a)$}}+1}}=\norm{K_{{N_2^{\text{\tiny{$\mfo(a)$}}+1}},\mu}^{\otimes[a]+1}\E[\xi^{\text{\scriptsize{$\tilde a$}}}_{\eps,\mu}]}_{L_x^\infty L^1_{y^a}}&\lesssim\sum_{\text{\scriptsize{$\mfm$}}^a:|\text{\scriptsize{$\mfl$}}^a+\text{\scriptsize{$\mfm$}}^a|<\ell}\Big(
\prod_{\text{\scriptsize{$\mfk$}} ij\in[a]}\norm{\partial_{{ y}_{\text{\tiny{$\mfk$}}ij}^a}^{\text{\scriptsize{$\mfm$}}_{\text{\tiny{$\mfk$}} ij}^a}  K_{{N_2},\mu} }_{\mcL^{\infty,\infty}}\Big)
\norm{K_{{N_2^{\text{\tiny{$\mfo(a)$}}}},\mu}\E[\text{$\bfI$}^a\big(\xi^{\text{\scriptsize{$\tilde a'$}}}_{\eps,\mu}\big)]}_{L^\infty_x}\\&\quad+\sum_{\text{\scriptsize{$\mfm$}}^a:|\text{\scriptsize{$\mfl$}}^a+\text{\scriptsize{$\mfm$}}^a|=\ell}\int_0^1\norm{\partial_{{y}^a}^{\text{\scriptsize{$\mfm$}}^a} K_{{N_2^{\text{\tiny{$\mfo(a)$}}+1}},\mu}^{\otimes[a]+1}\text{$\bfL$}_\tau \big(\E[\xi^{\text{\scriptsize{$\tilde a'$}}}_{\eps,\mu}]\big)}_{L_x^\infty L^1_{y^a}}\rmd\tau\,.
\end{equs}
While these new expressions with index $\mfl^a+\mfm^a$ have a better power counting than the original one, we trade the derivatives $\partial_{y^a}^{\mfm^a}$ appearing in \eqref{eq:localization} for some bad factors using \eqref{eq:spacederivKmu} (and thus taking $N_2$ large enough), which restores the scaling of $\bma$:
\begin{equs} \label{eq:step} 
\textcolor{white}{A}&\nnorm{\Kappa^\bma_\epmu}_{N_2^{\text{\tiny{$\mfo(a)$}}+1}}\\&\qquad\lesssim\sum_{\text{\scriptsize{$\mfm$}}^a:|\text{\scriptsize{$\mfl$}}^a+\text{\scriptsize{$\mfm$}}^a|<\ell}\mu^{-|\text{\scriptsize{$\mfm$}}^a|}
|\E[\text{$\bfI$}^a\big(\xi^{\text{\scriptsize{$\tilde a'$}}}_{\eps,\mu}\big)]|+\sum_{\text{\scriptsize{$\mfm$}}^a:|\text{\scriptsize{$\mfl$}}^a+\text{\scriptsize{$\mfm$}}^a|=\ell}\mu^{-|\text{\scriptsize{$\mfm$}}^a|}\int_0^1\norm{K_{{N_2^{\text{\tiny{$\mfo(a)$}}+1/2}},\mu}^{\otimes[a]+1}\text{$\bfL$}_\tau \big(\E[\xi^{\text{\scriptsize{$\tilde a'$}}}_{\eps,\mu}]\big)}_{L_x^\infty L^1_{y^a}}\rmd\tau\,.
\end{equs}
We now conclude using \eqref{eq:6_12B}, which yields
\begin{equs}    \norm{K_{{N_2^{\text{\tiny{$\mfo(a)$}}+1/2}},\mu}^{\otimes[a]+1}\text{$\bfL$}_\tau \big(\E[\xi^{\text{\scriptsize{$\tilde a'$}}}_{\eps,\mu}]\big)}_{L_x^\infty L^1_{y^a}}\lesssim
\norm{K_{{N_2^{\text{\tiny{$\mfo(a)$}}}},\mu}^{\otimes[a]+1}\E[\xi^{\text{\scriptsize{$\tilde a'$}}}_{\eps,\mu}]}_{L_x^\infty L^1_{y^a}}=\nnorm{\Kappa^{\text{\scriptsize{$\bma$}}'}_\epmu}_{
{N_2^{\text{\tiny{$\mfo(a)$}}}}
}\lesssim\eps^{(-1+c)s}\mu^{|a|+|\mfl^a|+|\mfm^a|-cs}\,,
\end{equs}
where $\bma'=(\tilde a')$, and in the last inequality we have used the induction hypothesis, thus enforcing $N_2\geqslant 36$. Note that we used the fact that the irrelevant cumulant $\Kappa^{\bma'}_\epmu$ has been controlled in the norm $\nnorm{\bigcdot}_{N_2^{\text{\tiny{$\mfo(\bma$}})}}$ in the previous analysis. Controlling the first term of \eqref{eq:step} with \eqref{eq:intermediateStoBound} finally concludes the proof.
\end{proof}
For the reader's convenience, we recall with our notation \cite[Lemma~6.12 (B)]{Duch21}: 
\begin{lemma}\label{lem:Lemma4_18}
  Fix $N\in\N$, $a\in\mcM$ and $\psi\in\mcD(\Lambda^{[a]+1})$. It holds
    \begin{equs}\label{eq:6_12B}
     \norm{ K_{6N,\mu}^{\otimes[a]+1}\text{$\bfL$}_\tau \psi}_{L_x^\infty L^1_{y^a}}\lesssim
\norm{ K_{N,\mu}^{\otimes[a]+1}\psi}_{L_x^\infty L^1_{y^a}}
    \end{equs}
uniformly in $\mu,\tau\in(0,1]$.
\end{lemma}
\begin{remark}\label{rem:Rem4_16}
    The cumulant analysis can be carried out for a covariance kernel bounded by $(|x_0|^{1/2}+|\ttx|)^{-4+2\alpha}$ for small $(x_0,\ttx)$, along with analogous bounds on
space-time derivatives, replacing the $L^{r,1}$ norm by an $L^{(n+2)(1+\iota)/(4-2\alpha)}$ norm when the noise is better behaved than a white noise, and thus adapting the power counting.
\end{remark}

\subsection{Kolmogorov argument}\label{subsec:K_argument}
In this subsection, we first control thanks to a probabilistic argument the integrated and relevant force coefficients that are introduced in Definition~\ref{def:defu4_20} below. The main result of this subsection is Lemma~\ref{lem:ConcluProba}. The integrated and relevant force coefficients turn out to be sufficient to control all the force coefficients thanks to a last inductive argument which is postponed to Subsection~\ref{subsec:44}. 
\begin{definition}\label{def:defu4_20} Let
    \begin{equs}       \mfM_{\rel}\eqdef  \{(a,\mfl^a)\in\mfM : |a|+|\mfl^a|\leqslant 0\}\,.
    \end{equs}
    For $(a,\mfl^a)\in\mfM_{\rel}$ and $\epmu\in(0,1]$, we define by
    \begin{equs}\label{eq:defXi}
        \Xi_\epmu^{ a,\mfl^a}\eqdef  \text{$\bfI$}^a\big(\xi^{( a,\mfl^a,0,0)}_\epmu\big)\in C^\infty(\Lambda)
    \end{equs}  
    the \textit{integrated and relevant force coefficients}. Finally, for any time weight $u\in\mcD(\R)$, we write $ \big(u\Xi_\epmu^{ a,\mfl^a}\big)(x)\eqdef u(x_0)\Xi_\epmu^{ a,\mfl^a}(x)$.
\end{definition}
\begin{lemma}\label{lem:ConcluProba}
Recall that $P\in2\N_{\geqslant1}$ is fixed, that $\eta=(2+n/r)/P$, and fix $u\in\mcD(\R)$. There exists $N_0\geqslant1$ and $c^\star>0$ such that for all $c\in(0,c^\star]$ and all $\eta>0$ we have
\begin{equs}\label{eq:coro423}\max_{(a,\mfl^a)\in\mfM_{\rel}}
  \max_{s\in\{0,1\}}  \E\Big[\Big(\sup_{\eps,\mu\in(0,1]} \eps^{s(1-c)}\mu^{-|a|-|\mfl^a|+sc+2\eta}\norm{\d_\eps^s\big(K_{N_0^{\Gamma+1},\mu}(u\Xi^{a,\mfl^a}_{\eps,\mu})\big)}_{L^\infty_x(\Lambda)}\Big)^P\Big]^{1/P}\lesssim_P 1\,.\textcolor{white}{blablablablab}
\end{equs}
\end{lemma}
\begin{proof}  
Let $(a,\mfl^a)\in\mfM_{\rel}$, $s\in\{0,1,2\}$ and $t\in\{0,1\}$ be fixed. Throughout this proof, we write $\tilde a\eqdef(a,\mfl^a,s,t)\in\widetilde{\mfM}$ and for $p\in[P]$ we set 
\begin{equs}
 \text{$\tilde a^p$}\eqdef \underbrace{(\tilde a,\dots,\tilde a)}_{p\; \text{\footnotesize times}}\in\widehat\mfM\,.
\end{equs}
We first aim to establish that
\begin{equs}   \label{eq:FirstStep} \E[\norm{\partial_\eps^s\partial_\mu^tK_{N_3^{\Gamma+1},\mu}(u\Xi^{a,\mfl^a}_{\eps,\mu})}_{L_x^P(\Lambda)}^P]\lesssim \Big(\eps^{s(-1+c)}\mu^{|a|+|\mfl^a|-sc-t-\eta}\Big)^P\,.
\end{equs}
The thesis then follows combining \eqref{eq:FirstStep} with Lemma~\ref{lem:Kolmo}.

It remains to prove \eqref{eq:FirstStep}. In view of the definition~\ref{eq:defXi} of $\Xi_\epmu^{a,\mfl^a}$, for every $p\in[P]$, it holds
\begin{equs}    {}&\norm{\kappa_p\big(K_{N_2^{\Gamma+1},\mu}\partial_\eps^s\partial_\mu^t\Xi^{a,\mfl^a}_{\eps,\mu}(x_1),\dots,K_{N_2^{\Gamma+1},\mu}\partial_\eps^s\partial_\mu^t\Xi^{a,\mfl^a}_{\eps,\mu}(x_p)\big)}_{L^\infty_{x_1}L^r_{{\tt x}_2,\dots,{\tt x}_p}L^1_{{ x}_{2,0},\dots,{ x}_{p,0}}(\Lambda^p)}\\&\qquad\qquad\lesssim\nnorm{\Kappa_{\eps,\mu}^{\text{\scriptsize{$\tilde a^p$}}}}_{N_2^{\Gamma+1}}\lesssim \Big(\eps^{s(-1+c)}\mu^{|a|+|\mfl^a|+2+n/r-sc-t-\eta}\Big)^P\,,
\end{equs}
where on the second inequality we used \eqref{eq:Boundcum}. To obtain the first inequality, we used the fact that we can bound the integrals in the definition of $\Xi_\epmu^{a,\mfl^a}$ by $L^1$ norms, and that by stationarity one has the freedom to introduce the kernels $K_{N_2^{\Gamma+1},\mu}$ in front of each variable $y_a$.

Thus, applying Lemma~\ref{lem:Duch8.17} to $\lambda=\partial_\eps^s\partial_\mu^t\Xi^{a,\mfl^a}_{\eps,\mu}$ and taking $N_3\geqslant N_2+2+n/r$, we have
\begin{equs}    \E[\big(K_{N_3^{\Gamma+1},\mu}\partial_\eps^s\partial_\mu^t\Xi^{a,\mfl^a}_{\eps,\mu}(x)\big)^P]\lesssim \Big(\eps^{s(-1+c)}\mu^{|a|+|\mfl^a|-sc-t-\eta}\Big)^P
\end{equs}
uniformly in $x\in\Lambda$ and $\epmu\in(0,1]$, which implies that for any time weight $v\in\mcD(\R)$, we have
\begin{equs}    \E[\norm{vK_{N_3^{\Gamma+1},\mu}\partial_\eps^s\partial_\mu^t\Xi^{a,\mfl^a}_{\eps,\mu}}_{L_x^P(\Lambda)}^P]\lesssim \Big(\eps^{s(-1+c)}\mu^{|a|+|\mfl^a|-sc-t-\eta}\Big)^P\,.
\end{equs}
Therefore, using Lemma~\ref{lem:Duch9.10} with $\lambda=\partial_\eps^s\partial_\mu^t\Xi^{a,\mfl^a}_{\eps,\mu}$, we can conclude that uniformly in $\epmu\in(0,1]$
\begin{equs}   \label{eq:BoundForceIntermediate} \E[\norm{K_{N_3^{\Gamma+1},\mu}(u\partial_\eps^s\partial_\mu^t\Xi^{a,\mfl^a}_{\eps,\mu})}_{L_x^P(\Lambda)}^P]\lesssim \Big(\eps^{s(-1+c)}\mu^{|a|+|\mfl^a|-sc-t-\eta}\Big)^P\,.
\end{equs}
which is the desired result \eqref{eq:FirstStep} for $t=0$. Regarding the case $t=1$, rewriting
\begin{equs}
    \d_\eps^s\d_\mu \big(K_{N_3^{\Gamma+1},\mu}(u\Xi_{\eps,\mu}^{a,\mfl^a})\big)= \mcP^{N_3^{\Gamma+1}}_\mu\d_\mu K_{N_3^{\Gamma+1},\mu}K_{N_3^{\Gamma+1},\mu} \big( u\d_\eps^s \Xi_{\eps,\mu}^{a,\mfl^a}\big)+ K_{N_3^{\Gamma+1},\mu} \big(u\d_\eps^s\d_\mu\Xi_{\eps,\mu}^{a,\mfl^a}\big)\,,
\end{equs} 
and combining the latter with an application of \eqref{eq:derivKmu} and \eqref{eq:BoundForceIntermediate}
 yields \eqref{eq:FirstStep}, which finally concludes the proof.

\end{proof}
The following technical lemmas from \cite{Duch21} are required in the proof of Lemma~\ref{lem:ConcluProba}. First, we need an adaptation of \cite[Lemma 8.17]{Duch21} to our $L^{r,1}$ norms:
\begin{lemma}\label{lem:Duch8.17}
    Suppose that the cumulants of a random field $\lambda\in C^\infty(\Lambda)$ verify for $p\in[P]$
    \begin{equs}        \norm{\kappa_p\big(\lambda(x_1),\dots,\lambda(x_p)\big)}_{L^\infty_{x_1}L^r_{{\tt x}_2,\dots,{\tt x}_p}L^1_{{ x}_{2,0},\dots,{ x}_{p,0}}(\Lambda^p)}\lesssim1\,.
    \end{equs}
    Then, for every $N\geqslant 2+n/r$, it holds
    \begin{equs}        \E[\big(K_{N,\mu}\lambda(x)\big)^P]\lesssim \mu^{-P(2+n/r)}
    \end{equs}
    uniformly in $x\in\Lambda$ and $\mu\in(0,1]$.
\end{lemma}
\begin{proof}
First, write
\begin{equs}
    \E[\big(K_{N,\mu}\lambda(x)\big)^P]&=\int_{\Lambda^P}K_{N,\mu}(x-y_1)\cdots K_{N,\mu}(x-y_P)\E[\lambda(y_1)\cdots\lambda(y_P)]\rmd y_1\dots\rmd y_P\\&\lesssim
    \norm{K_{N,\mu}^{\otimes P}*\E[\lambda^{\otimes P}](x_1,\cdots,x_P)}_{L_{x_1,\dots,x_P}^\infty(\Lambda^P)}\,,
\end{equs}
where $*$ denote the convolution of $\Lambda^P$. Using \eqref{eq:KmuLp}, or more precisely the fact that $K_{N,\mu}$ is an exterior product of some space and time components, and that the $\mcL(L^1,L^\infty)$ norm of the time component is of order $\mu^{-2}$ and the $\mcL(L^r,L^\infty)$ norm of the space component of order $\mu^{-n/r}$, we can therefore conclude that
\begin{equs}    \E[\big(K_{N,\mu}\lambda(x)\big)^P]&\lesssim    \mu^{-(P-1)(2+r/n)}\norm{\E[\lambda(x_1)\cdots\lambda(x_P)]}_{L^\infty_{x_1}L^r_{{\tt x}_2,\dots,{\tt x}_P}L^1_{x_{2,0},\dots,x_{P,0}}(\Lambda^P)}\,.
\end{equs}
Note that to derive the above inequality, we took care to first use the operators norms estimates in $x_{2,0},\cdots,x_{P,0}$, and only after in ${\tt x}_{2,0},\cdots,{\tt x}_{P,0}$.

The statement now follows from the relation between moments and cumulants.
\end{proof}
Then, we recall \cite[Lemma 9.10]{Duch21}, that allows to put the time weights inside the kernels $K_{N,\mu}$:
\begin{lemma}\label{lem:Duch9.10}
    Fix $N\in\N$, $u\in\mcD(\R)$, and $\lambda\in C^\infty(\Lambda)$. It holds
    \begin{equs}
        \norm{K_{N,\mu}(u\lambda)}_{L_x^1(\Lambda)}\lesssim\sum_{M=0}^N \mu^{M}   \norm{\partial_{x_0}^M u K_{N,\mu}\lambda}_{L_x^1(\Lambda)}
    \end{equs}
    uniformly in $\mu\in(0,1]$.
\end{lemma}
Finally, we recall the Kolmogorov estimates of \cite[Lemma~9.5]{Duch21}:
\begin{lemma}\label{lem:Kolmo}
Pick $(a,\mfl^a)\in\mfM_{\rel}$, and suppose that for all $c\in(0,c^\star]$ and all $\eta>0$ we have
\begin{equs}
  \max_{s\in\{0,1,2\}}\max_{t\in\{0,1\}}\sup_{\eps,\mu\in(0,1]}  \E\Big[\Big( \eps^{s(1-c)}\mu^{-|a|-|\mfl^a|+sc+t+\eta}\norm{\d_\eps^s\d_\mu^t\big(K_{N_3^{\Gamma+1},\mu}(u\Xi^{a,\mfl^a}_{\eps,\mu})\big)}_{L^P_x(\Lambda)}\Big)^P\Big]^{1/P}\lesssim_{P,a,\mfl^a,u} 1\,.
\end{equs}
Then, for every $N_0\geqslant N_3+2$, it holds
  \begin{equs}
  \max_{s\in\{0,1\}}  \E\Big[\Big(\sup_{\eps,\mu\in(0,1]} \eps^{s(1-c)}\mu^{-|a|-|\mfl^a|+sc+2\eta}\norm{\d_\eps^s\big(K_{N_0^{\Gamma+1},\mu}(u\Xi^{a,\mfl^a}_{\eps,\mu})\big)}_{L^\infty_x(\Lambda)}\Big)^P\Big]^{1/P}\lesssim_{P,a,\mfl^a,u} 1\,.
\end{equs}
\end{lemma}
\begin{remark}
  The statement of \cite[Lemma~9.5]{Duch21} does not include uniform control on the additional parameter $\eps\in(0,1]$. See \cite[Lemma~13.6]{Duch22} for a similar statement including the uniform control on $\eps$.
\end{remark}

\subsection{Deterministic post-processing}\label{subsec:44}
In this last subsection, we conclude the proof of Theorem~\ref{thm:sto} by a post-processing of the bound~\eqref{eq:coro423}: by means of a second induction, we show that the control over the integrated and relevant force coefficients is sufficient to control all the force coefficients.

\begin{proof}[of Theorem~\ref{thm:sto}]
Just like Lemma~\ref{lem:cumul}, Theorem~\ref{thm:sto} is proved by induction over the order of $a$. However, rather than the flow equation for cumulants, we can directly rely on the hierarchy of flow equations~\ref{eq:flowGen}. Moreover, as a consequence of the structure of \ref{eq:flowGen}, rather than $\partial^s_\eps\xi_\epmu^a$, we actually control all the $\xi_\epmu^{(a,\mfl^a,s,0)}$ for $\mfl^a\in\big(\N^{n+1}\big)^{[a]}$.

The base case is a straightforward consequence of Lemma~\ref{lem:ConcluProba}, since one has $\partial^s_\eps\xi_\epmu^{\1_0^\mfh}(x,y)=\partial^s_\eps\xi_\eps(x)\delta(x-y)$ so that $\partial^s_\eps\Xi_\epmu^{\1_0^\mfh,0}(x)=\partial^s_\eps\xi_\eps(x)$, and control on the $L^\infty_x$ norm of $\partial^s_\eps\Xi_\epmu^{\1_0^\mfh,0}$ directly entails control on the $L^\infty_xL^1_y$ norm of $\partial^s_\eps\xi_\epmu^{\1_0^\mfh}$. Note this importance of the time weight $u$ in order to control the noise on $(-\infty,1]\times\T^n$.

The induction step is also direct for the irrelevant force coefficients $\xi_\epmu^{(a,\mfl^a,s,0)}$, that is to say when $|a|+|\mfl^a|>0$, since, in this case, the RHS of \eqref{eq:flowGen} applied with $\tilde a=(a,\mfl^a,s,1)$ is integrable at $\mu=0$.

As in the proof of Lemma~\ref{lem:cumul}, we take care of the choice of the value of $N$ for the kernel $K_{N,\mu}$ in the norms for the irrelevant coefficients. Rather than directly working with $N=N_0^{2\Gamma+1}$, with actually show the desired bound for some $(a,\mfl^a)\in\mfM$ with $N=N_0^{\Gamma+\mfo(a)+1}$. Moreover, in the case of irrelevant coefficients, we rather
set $N=N_0^{\Gamma+\text{\scriptsize{$\mfo(a)$}}}$ rather than $N=N_0^{\Gamma+\text{\scriptsize{$\mfo(a)$}}+1}$, since we do not need to lose any factor in the induction. This is possible since we do not need to use any Sobolev embedding type bound.

We thus have the following estimate on irrelevant coefficients:
\begin{equs}  \label{eq:irrelcoefs}  \norm{K_{{N_0^{\Gamma+\text{\tiny{$\mfo(a)$}}}},\mu}^{\otimes[a]+1}(u\xi^{(a,\mfl^a,s,0)}_{\eps,\mu})}_{L_x^\infty L^1_{y^a}}&\lesssim \eps^{s(-1+c)}\mu^{|a|+|\mfl^a|-sc-2\eta}\,.
\end{equs}
It remains to deal with the case of relevant force coefficients $\xi_\epmu^{(a,\mfl^a,s,0)}$ for $(a,\mfl^a)\in\mfM_{\rel}$. In this case, the flow equation \eqref{eq:flowGen} is not integrable at $\mu=0$, so that we have to make use of \eqref{eq:localization} taking $\ell$ to be the smallest integer such that $|a|+|\mfl^a|+\ell>0$ (again $\ell\in\{1,2\}$). This yields the estimate
\begin{equs}  
\norm{K_{{N_0^{\Gamma+\text{\tiny{$\mfo(a)$}}+1}},\mu}^{\otimes[a]+1}(u\xi^{(a,\mfl^a,s,0)}_{\eps,\mu})}_{L_x^\infty L^1_{y^a}}&\lesssim\sum_{\text{\scriptsize{$\mfm$}}^a:|\text{\scriptsize{$\mfl$}}^a+\text{\scriptsize{$\mfm$}}^a|<\ell}\Big(
\prod_{\text{\scriptsize{$\mfk$}} ij\in[a]}\norm{\partial_{{ y}_{\text{\tiny{$\mfk$}}ij}^a}^{\text{\scriptsize{$\mfm$}}_{\text{\tiny{$\mfk$}} ij}^a}  K_{{N_0},\mu} }_{\mcL^{\infty,\infty}}\Big)
\norm{\d_\eps^s\big(K_{{N_0^{\Gamma+\text{\tiny{$\mfo(a)$}}}},\mu}(u\Xi^{a,\mfl^a+\mfm^a}_{\eps,\mu})\big)}_{L^\infty_x}\\&\quad+\sum_{\text{\scriptsize{$\mfm$}}^a:|\text{\scriptsize{$\mfl$}}^a+\text{\scriptsize{$\mfm$}}^a|=\ell}\int_0^1\norm{\partial_{{y}^a}^{\text{\scriptsize{$\mfm$}}^a} K_{{N_0^{\Gamma+\text{\tiny{$\mfo(a)$}}+1}},\mu}^{\otimes[a]+1}\text{$\bfL$}_\tau \big(u\xi^{(a,\mfl^a+\mfm^a,s,0)}_{\eps,\mu}\big)}_{L_x^\infty L^1_{y^a}}\rmd\tau\,.
\end{equs}
As in the proof of Lemma~\ref{lem:cumul}, we trade the derivatives $\partial_{y^a}^{\mfm^a}$ appearing in \eqref{eq:localization} for some bad factors using \eqref{eq:spacederivKmu} (taking $N_0$ large enough), which restores the scaling of $a$:
\begin{equs} 
{}\norm{K_{{N_0^{\Gamma+\text{\tiny{$\mfo(a)$}}+1}},\mu}^{\otimes[a]+1}(u\xi^{(a,\mfl^a,s,0)}_{\eps,\mu})}_{L_x^\infty L^1_{y^a}}\lesssim &\sum_{\text{\scriptsize{$\mfm$}}^a:|\text{\scriptsize{$\mfl$}}^a+\text{\scriptsize{$\mfm$}}^a|<\ell}\mu^{-|\text{\scriptsize{$\mfm$}}^a|}
\norm{\d_\eps^s\big(K_{{N_0^{\Gamma+\text{\tiny{$\mfo(a)$}}}},\mu}(u\Xi^{a,\mfl^a+\mfm^a}_{\eps,\mu})\big)}_{L^\infty_x}\\&+\sum_{\text{\scriptsize{$\mfm$}}^a:|\text{\scriptsize{$\mfl$}}^a+\text{\scriptsize{$\mfm$}}^a|=\ell}\mu^{-|\text{\scriptsize{$\mfm$}}^a|}\int_0^1\norm{K_{{N_0^{\Gamma+\text{\tiny{$\mfo(a)$}}+1/2}},\mu}^{\otimes[a]+1}\text{$\bfL$}_\tau \big(u\xi^{(a,\mfl^a+\mfm^a,s,0)}_{\eps,\mu}\big)}_{L_x^\infty L^1_{y^a}}\rmd\tau\,.
\end{equs}
The first term of the RHS can now be controlled using \eqref{eq:coro423}, while we can get rid of $\text{$\bfL$}_\tau$ in the second term using \eqref{eq:6_12B}, and thus enforcing $N_0\geqslant36$. As in the proof of Lemma~\ref{lem:cumul}, this yields
\begin{equs}    \norm{K_{{N_0^{\Gamma+\text{\tiny{$\mfo(a)$}}+1/2}},\mu}^{\otimes[a]+1}\text{$\bfL$}_\tau \big(u\xi^{(a,\text{\scriptsize{$\mfl^a+\mfm^a$}},s,0)}_{\eps,\mu}\big)}_{L_x^\infty L^1_{y^a}}\lesssim
\norm{K_{{N_0^{\Gamma+\text{\tiny{$\mfo(a)$}}}},\mu}^{\otimes[a]+1}(u\xi^{(a,\text{\scriptsize{$\mfl^a+\mfm^a$}},s,0)}_{\eps,\mu})}_{L_x^\infty L^1_{y^a}}\lesssim\eps^{s(-1+c)}\mu^{|a|+|\mfl^a|+|\mfm^a|-sc-2\eta}\,,
\end{equs}
where on the last inequality we used \eqref{eq:irrelcoefs}. This concludes the proof.
\end{proof}
 
\section{Construction of the non-stationary effective force}\label{Sec:Sec5}
This section is devoted to the proof of Theorem~\ref{coro:2}.
We construct the non-stationary force coefficients $(\zeta^a_\epmu)^{a\in\mcM}_{\epmu\in(0,1]}$ starting from the stationary force coefficients $(\xi^a_\epmu)^{a\in\mcM}_{\epmu\in(0,1]}$ by deterministic means, and with the input that the stationary force coefficients verify the estimate \eqref{eq:boundXi}, which was proven in Section~\ref{sec:Sec4}.

First, recall that we can express the stationary effective force $S_\epmu$ defined as
\begin{equs}
     S_\epmu[\psi](x)=\sum_{a\in\mcM} S^a_\epmu[\psi](x)=\sum_{a\in\mcM}\int_{\Lambda^{[a]}} \langle \xi^a_\epmu(x,y^a) ,\Upsilon^a[\psi](y^a)\rangle_{\mcH^a}\rmd y^a\,,
\end{equs}
and that it verifies $S_\epmu=\Pi^{\leqslant\Gamma}S_\epmu$ and solves $\Pol_\mu^{\leqslant\Gamma}(S_\epmu)=0$ with initial condition $S_\eps$.

We now aim to make sense of the solution $F_\epmu$ to $\Pol_{\mu}^{\leqslant\Gamma}(F_\epmu)=0$ with initial condition $F_{\eps,0}=\1^0_{(0,\infty)}S_\eps[\bigcdot]$. The strategy is to show that far from 0, $F_\epmu$ coincides with $S_\epmu$, and that it is therefore constructed with the same renormalization, which ultimately turns out to be the same renormalization necessary to make sense of $S_\epmu$.

Moreover, since for $\psi\in\mcD(\Lambda)$, $F_\eps[\psi](x)$ is supported on positive times, one is only interested in constructing $F_\epmu[\psi](x)$ supported on positive times. We therefore make the ansatz that $F_\epmu$ is of the form
\begin{equs}\label{eq:Fmu_ansatz}
     F_\epmu[\psi](x)=\sum_{a\in\mcM} F^a_\epmu[\psi](x)=\sum_{a\in\mcM}\int_{\Lambda^{[a]}} \langle \zeta^a_\epmu(x,y^a) ,\Upsilon^a[\psi](y^a)\rangle_{\mcH^a}\rmd y^a\,,
\end{equs}
for $(\zeta^a_\epmu)^{a\in\mcM}_{\epmu\in(0,1]}$ a collection of force coefficients such that $\zeta^a_\epmu(x,y^a)=\1_{(0,\infty)}(x_0)\zeta^a_\epmu(x,y^a)$.

The crucial observation in Section 10.4 of \cite{Duch21} is that $\xi^a_\epmu$ and $\zeta^a_\epmu$ coincide when evaluated sufficiently far (depending on $\mu$) from the the zero time plane in their first arguments. 
This is the content of the following lemma, analogous to \cite[Lemma~10.47]{Duch21}.  
\begin{lemma}\label{lem:support}
  Fix $a\in\mcM$, $\epmu\in(0,1]$ and $y^a\in\Lambda^{[a]}$. For all $x=(x_0,\ttx)\in (2\mu^2\mfo(a),1]\times\T^n$, it holds
    \begin{equs}
        \zeta^a_\epmu(x,y^a)=\xi^a_\epmu(x,y^a)\,.
    \end{equs}
\end{lemma}
\begin{proof}
    The proof is by recursion on the order of $a$. Since on positive times, $\xi_{\eps,0}^a(x,y^a)=\zeta_{\eps,0}^a(x,y^a)=\mfc^a_\eps\delta^a(x,y^a)$, it suffices to prove that $\d_\mu\xi_\epmu^a(x,y^a)=\d_\mu\zeta_\epmu^a(x,y^a)$ for $x_0>2\mu^2\mfo(a)$. Moreover, we know that it holds
\begin{equs}
      \d_\mu &\xi^a_\epmu(x,y^a)=-\sum_{\substack{(\sigma,b,c,\bfd)\in\text{\scriptsize{$\mathrm{Ind}(a)$}}}}  
    \B_\mu(\xi^b_\epmu,\xi^c_\epmu)(x,y^a)
    \,.
\end{equs}
By construction, the $(\zeta^a_\epmu)^{a\in\mcM}_{\epmu\in(0,1]}$ solve the same hierarchy of equations:

\begin{equs}\label{eq:flowzeta}
       \d_\mu &\zeta^a_\epmu(x,y^a) =-\sum_{\substack{(\sigma,b,c,\bfd)\in\text{\scriptsize{$\mathrm{Ind}(a)$}}}}      \B_\mu(\zeta^b_\epmu,\zeta^c_\epmu) (x,y^a)
    \,.
\end{equs}
The recursion hypothesis already ensures that we can replace $\zeta^b_\epmu$ by $\xi^b_\epmu$ on the RHS. At this stage, we need to us the support property of $\xi^a_\epmu$ stated in \eqref{eq:supportXi}. This property implies that, using the notation $z,w$ as in \eqref{eq:defB}, we have $z_0\geqslant x_0-2\mu^2\mfo(b)>2\mu^2(\mfo(a)-\mfo(b))$. In the second inequality, we have used the hypothesis that $x_0>2\mu^2\mfo(a)$. Moreover, since $\dot G_\mu$ is supported on $[0,2\mu^2]\times\T^n$, we therefore have $w_0\geqslant z_0-2\mu^2>2\mu^2(\mfo(a)-\mfo(b)-1)=2\mu^2\mfo(c)$ so that by the induction hypothesis we can replace $\zeta^c_\epmu$ by $\xi^c_\epmu$. This confirms that after a time $2\mu^2\mfo(a)$, $ \d_\mu \zeta^a_\epmu =\d_\mu \xi^a_\epmu$, and therefore concludes the proof. 
\end{proof}
Since $\zeta^a_\epmu(x,y^a)$ is now well-defined for $x_{0} > 2\mu^2\mfo(a)$, it remains to construct it for $x_0\in(0,2\mu^2\mfo(a)]$. 
Here we follow the proof of  \cite[Theorem~10.50]{Duch21} and leverage the fact that the region $(0,2\mu^2\mfo(a)]\times\T^n$ has Lebesgue measure of order $\mu^2 $ which will result in a good factor allowing one to integrate the flow equation for $\zeta^a_\epmu(x,y^a)$, provided one works with some $L^1$ norms in time instead of the usual $L^\infty$ norms. 

\begin{proof}[of Theorem~\ref{coro:2}]
 To lighten notation, we only prove the bound for $\Norm{\zeta}_{P,\eta}$, the proof for $\Norm{\zeta}_{P,\eta,c}$ being totally similar, and only requiring a heavier notation. Moreover, we only prove the statement for $\mfl^a=0$. Again, the generalisation to other polynomial weights is straightforward and would just make the notation heavier.

Recall the two collections of time weights $v=(v_\mu)_{\mu\in(0,1]}$, $w=(w_\mu)_{\mu\in(0,1]}$ introduced in Definition~\ref{def:weightvw}. Because we need more room, we define two other families $\hat v=(\hat v_\mu)_{\mu\in(0,1]}$, $\hat w=(\hat w_\mu)_{\mu\in(0,1]}$ by $\hat v_\mu(t)\eqdef v_\mu(t/\Gamma)$ and $\hat w_\mu(t)\eqdef w_\mu(t/\Gamma)$. These weights are defined in such a way that we have $\hat v_\mu\in\mcW^\infty_{N,\mu}$ and $\hat w_\mu\in\mcW^1_{N,\mu}\cap\mcW^\infty_{N,\mu}$ for all $N\in\N$ and that it holds 
\begin{equs}    \text{supp}\,\hat v_\mu\subset[2\Gamma\mu^2,\infty)\,,\;\text{and}\;\text{supp}\,\hat w_\mu\subset[-3\Gamma\mu^2,3\Gamma\mu^2]\,,
\end{equs}
as well as $\hat v_\mu(t)+\hat w_\mu(t)=1$ if $t\geqslant0$.

Recall that for any $a\in\mcM$, any function $\psi:\Lambda^{[a]+1}\rightarrow\R$ and any time weight $u$, we write $u \psi(x,y^a)=u(x_0)\psi(x,y^a)$. Using Lemma~\ref{lem:support}, for $x\in\Lambda$, we thus have
\begin{equs} \label{eq:zetamu}   \zeta^a_\epmu(x,y^a)=\hat v_\mu\xi_\epmu^a(x,y^a)+\hat w_\mu\zeta^a_\epmu(x,y^a)\,.
\end{equs}

Again, in this induction, we are careful about the value of $N$ in the kernels $K_{N,\mu}$, and rather than directly showing the desired result for $N=N_1^{3\Gamma+1}$, in the induction step for some $a\in\mcM$, we actually show it for $N=N_1^{2\Gamma+\mfo(a)+1}$.

We deal with the first term in \eqref{eq:zetamu} as follows: we want to control
\begin{equs}    \nnorm{\hat v_\mu\xi^a_\epmu}_{N_1^{2\Gamma+\mfo(a)+1}}\,.
\end{equs}
Here, recall that the supremum over the time $x_0$ at which is evaluated $K_\mu^{\otimes[a]+1}(\hat v_\mu\xi^a_\epmu)(x,y^a)$ is taken over $(-\infty,1]$. Since the kernel of $K_\mu$ is supported on $\R_{\geqslant0}\times\T^n$, we have the freedom to add a smooth compacted function $u$ supported on $(-2\Gamma\mu^2-1,2]$ and which is equal to one on $(-2\Gamma\mu^2,1]$. $\hat v_\mu$ lies in $\mcW^\infty_{N,\mu}$. We can therefore use \eqref{eq:1052a} and the fact that $N_1\geqslant N_0$ to obtain the bound
\begin{equs}    \nnorm{\hat v_\mu\xi^a_\epmu}_{N_1^{2\Gamma+\mfo(a)+1}}=\nnorm{u\hat v_\mu\xi^a_\epmu}_{N_1^{2\Gamma+\mfo(a)+1}}\lesssim\nnorm{u\xi^a_\epmu}_{N_0^{2\Gamma+\mfo(a)+1}}\lesssim\nnorm{u\xi^a_\epmu}_{N_0^{2\Gamma+1}}\lesssim\mu^{|a|-\eta}\,.
\end{equs}
Let us turn to the second term in \eqref{eq:zetamu}. The proof is by induction on the order of $a$, and relies on the following flow equation:
\begin{equs}
    \d_\mu\big(\hat w_\mu\zeta^a_\epmu\big)(x,y^a)&=\hat w_\mu\d_\mu\zeta^a_\epmu(x,y^a)+(\d_\mu \hat w_\mu)\zeta^a_\epmu(x,y^a)\\
    &=\hat w_\mu\d_\mu\zeta^a_\epmu(x,y^a)-(\d_\mu \hat v_\mu)\xi^a_\epmu(x,y^a)\,.\label{eq:dmuwmuzeta}
\end{equs}
While we use \eqref{eq:boundXi} to control the second term in the RHS of \eqref{eq:dmuwmuzeta}, we use the flow equation \eqref{eq:flowzeta} to handle the first one. Indeed, taking the weight $\hat w_\mu$ into account, the flow equation \eqref{eq:flowzeta} rewrites
\begin{equs}
     \hat w_\mu  \d_\mu &\zeta^a_\epmu =-\sum_{\substack{(\sigma,b,c,\bfd)\in\text{\scriptsize{$\mathrm{Ind}(a)$}}}}      \B_\mu(\hat w_\mu\zeta^b_\epmu,\zeta^c_\epmu)     
    \,.\label{eq:flowwmuzetamu}
\end{equs}
Recall the notation $
        K_{N,\mu}^{\otimes[a]+1}\lambda^a_\epmu(x,y^a)=\big(K_{N,\mu}\otimes\dots\otimes K_{N,\mu}\big)*\lambda^a_\epmu(x,y^a)$. For $N\in\N_{\geqslant1}$, we define a new norm $\nnorm{\bigcdot}_{1,N}$ inspired by $\nnorm{\bigcdot}_{N}$ by setting
    \begin{equs}        \nnorm{\lambda^a_{\eps,\mu}}_{1,N}\eqdef \norm{K^{\otimes[a]+1}_{N,\mu}\lambda^a_{\eps,\mu}}_{L^{\infty,1}_{x}L^1_{y^a}(\Lambda^{[a]+1})}\equiv \sup_{\ttx \in\T^n}\int_{(-\infty,1]\times \Lambda^{[a]}} |K^{\otimes[a]+1}_{N,\mu} \lambda^a_{\eps,\mu}(x,y^a)|\rmd x_0\rmd y^a\,.
    \end{equs}    
The flow equation \eqref{eq:flowwmuzetamu} and \eqref{eq:propB} imply that
\begin{equs}    \nnorm{\hat w_\mu\d_\mu\zeta^a_\epmu}_{1,N_0^{2\Gamma+\mfo(a)+1}} &\lesssim 
\sum_{\substack{(\sigma,b,c,\bfd)\in\text{\scriptsize{$\mathrm{Ind}(a)$}}}}   
    \norm{\mcP^{2N_0^{2\Gamma+\mfo(a)+1}}_\mu\partial_{\ttx}^{k_0+1-k_1}\dot G_\mu}_{\mcL^{\infty,\infty}}\nnorm{\hat w_\mu\zeta^b_\epmu}_{1,N_0^{2\Gamma+\mfo(a)+1}}\nnorm{\zeta^c_\epmu}_{N_0^{2\Gamma+\mfo(a)+1}}
    \,.
\end{equs}
Here, we can use the induction hypothesis and \eqref{eq:1052b} to control 
\begin{equs}
    \nnorm{\hat w_\mu\zeta^b_\epmu}_{1,N_0^{2\Gamma+\mfo(a)+1}}\lesssim\mu^2
    \nnorm{\zeta^b_\epmu}_{N_0^{2\Gamma+\mfo(a)+1}}\lesssim
    \mu^2
    \nnorm{\zeta^b_\epmu}_{N_1^{2\Gamma+\mfo(b)}}
    \lesssim\mu^{|b|-\eta+2}\,,
\end{equs}
and
\begin{equs}   
    \nnorm{\zeta^c_\epmu}_{N_0^{2\Gamma+\mfo(a)+1}}\lesssim
        \nnorm{\zeta^c_\epmu}_{N_1^{2\Gamma+\mfo(c)+1}}
    \lesssim\mu^{|c|-\eta}\,.
\end{equs}
Here, we required that $N_0^{2\Gamma+\mfo(d)+2}\geqslant N_1^{2\Gamma+\mfo(d)+1}$ for $d=b,c$: taking $N_1\leqslant N_0^{(3\Gamma+1)/(3\Gamma)}$ suffices.
We therefore end up with 
\begin{equs}      \nnorm{\hat w_\mu\d_\mu\zeta^a_\epmu}_{1,N_0^{2\Gamma+\mfo(a)+1}}&\lesssim \sum_{\substack{(\sigma,b,c,\bfd)\in\text{\scriptsize{$\mathrm{Ind}(a)$}}}}  
   \mu^{1-(k_0+1-k_1)+|b|-\eta+2+|c|-\eta}
    \lesssim      \sum_{\substack{(\sigma,b,c,\bfd)\in\text{\scriptsize{$\mathrm{Ind}(a)$}}}}  
   \mu^{|a|-2\eta+1}
 \lesssim\mu^{|a|-2\eta+1}
    \,.
\end{equs}
Observe that $|a|-2\eta+1\geqslant-1+\alpha-2\eta$ so that the singularity at $\mu=0$ is now integrable. Recalling \eqref{eq:dmuwmuzeta} and \eqref{eq:Kmunu}, we have that
\begin{equs}
    \nnorm{\hat w_\mu\zeta^a_\epmu}_{1,N_0^{2\Gamma+\mfo(a)+1}}\lesssim\int_0^\mu\Big(\nnorm{\hat w_\nu\d_\nu\zeta^a_\epnu}_{1,N_0^{2\Gamma+\mfo(a)+1}}+\nnorm{\d_\nu \hat v_\nu\xi^a_\epnu}_{1,N_0^{2\Gamma+\mfo(a)+1}} \Big)\rmd\nu\,.
\end{equs}
Here, we need to observe that $\d_\nu \hat v_\nu(t)=\nu^{-1}z_\nu(t)$ with $z_\nu(t)=-2(\t v')(t/\Gamma\nu^2)$, from which we deduce that $z_\nu\in\mcW^1_{N,\mu}$ for all $N\in\N$. 
Thus, \eqref{eq:1052b} and \eqref{eq:boundXi} imply that 
\begin{equs}
    \nnorm{\d_\nu \hat v_\nu\xi^a_\epnu}_{1,N_0^{2\Gamma+\mfo(a)+1}}\lesssim\nu\nnorm{u\xi^a_\epnu}_{{N_0^{2\Gamma+\mfo(a)+1}}}\lesssim    \nu\nnorm{u\xi^a_\epnu}_{{N_0^{2\Gamma+1}}}
    \lesssim\nu^{|a|-\eta+1}\,,
\end{equs}
where as before we inserted a smooth compactly supported function $u$. By \eqref{eq:KmuLp} we finally have the desired result  
\begin{equ}    \nnorm{\hat w_\mu\zeta^a_\epmu}_{N_1^{2\Gamma+\mfo(a)+1}}\lesssim\mu^{-2}\nnorm{\hat w_\mu\zeta^a_\epmu}_{1,N_0^{2\Gamma+\mfo(a)+1}}\lesssim\mu^{-2}\int_0^\mu\nu^{|a|-2\eta+1}\rmd\nu\lesssim\mu^{|a|-2\eta}\,.
\end{equ}
Here, we enforce $N_1^{2\Gamma+1}\geqslant N_0^{2\Gamma+1}+2$ in order to take into account the loss due to the use of the Sobolev embedding.
\end{proof}

\begin{appendix}

\section{The regularizing kernels and effective Green's function}\label{app:A}
We recall here some very important properties of the operators $K_{N,\mu}$ stated in \cite{Duch21}. We do not prove these statements, the proof of which is already given therein.
\begin{lemma}
\begin{equs} \label{eq:normKmu}\norm{K_{N,\mu}}_{\mcL^{\infty,\infty}}&\leqslant1   \,,\\
\norm{\partial^{\mfl}K_{N,\mu}}_{\mcL^{\infty,\infty}}&\lesssim\mu^{-|\mfl|}\;\mathrm{for}\;\mathrm{all} \;\mfl\in\N^{n+1}\,,\;|\mfl|\leqslant N\,,
\label{eq:spacederivKmu}
\\
\label{eq:KmuLp} \norm{K_{N,\mu}}_{\mcL^{(p,q),\infty}}&\lesssim \mu^{-n/p-2/q} \;\mathrm{for}\;\mathrm{all} \;p,q\in[1,\infty] \,,\,n/p+2/q\leqslant N\,,\\\norm{\mcR_\mu\d_\mu K_{N,\mu}}_{\mcL^{\infty,\infty}}&\lesssim\mu^{-1}\,,\label{eq:derivKmu} \end{equs}
uniformly in $\mu\in(0,1]$. Moreover, for $\psi\in\mcD(\Lambda)$ and $N\geqslant1$, we have
\begin{equs}  
\norm{K_{N,\mu}\psi}_{L^\infty}&\leqslant\bigg(1\vee\Big(2\big(\nu/\mu\big)^{2}-1\Big)^{2N}\bigg)  \norm{K_{N,\nu}\psi}_{L^\infty}\,,\label{eq:Kmunu}
\\
    \norm{K_{N,\mu}(1-K_{N,\nu})\psi}_{L^\infty}&\lesssim_N \big(\nu/\mu\big)^{2}\norm{ K_{N,\nu}\psi}_{L^\infty}\;\mathrm{for}\;\nu\leqslant\mu\,,\,,\label{eq:comKmuKnu}
\end{equs}
uniformly in $\mu,\nu\in(0,1]$.
\end{lemma}
\begin{proof}    
\eqref{eq:normKmu} is an immediate consequence of the definition of $K_{N,\mu}$. \eqref{eq:spacederivKmu} is \cite[Lemma~4.14 (A)]{Duch21}. \eqref{eq:KmuLp} is a minor modification of \cite[Lemma~4.14 (D)]{Duch21} (see also \cite[Lemma~10.52 (C)]{Duch21}). \eqref{eq:derivKmu} is \cite[Lemma~4.17]{Duch21}, \eqref{eq:Kmunu} is \cite[Lemma~4.10]{Duch21}, and \eqref{eq:comKmuKnu} was first observed in \cite{Duch23} in an elliptic context. The idea is that we have
\begin{equs}
    (1+\mu^2\partial_{x_0})^{-1}\big(1-  (1+\nu^2\partial_{x_0})^{-1}\big)&=\mu^2\partial_{x_0}  (1+\mu^2\partial_{x_0})^{-1}  (1+\nu^2\partial_{x_0})^{-1}\\
    &=(\nu/\mu)^2   (1+\nu^2\partial_{x_0})^{-1}\big(1-  (1+\mu^2\partial_{x_0})^{-1}\big)\,.
\end{equs}
A similar statement holds for $ (1-\mu^2\Delta)^{-1}\big(1-  (1-\nu^2\Delta)^{-1}\big)$, so that at least one factor $(\nu/\mu)^2$ is created in all the terms appearing in $K_{N,\mu}(1-K_{N,\nu})$.
\end{proof}
\begin{lemma}
Fix a smooth function $\rho:\Lambda\rightarrow\R$. Recall that for $\eps>0$ and $x\in\Lambda$, the rescaling operator is given by $\mcS_\eps\rho(x)=\eps^{-(n+2)}\rho(x_0/\eps^{2},{\tt x}/\eps)$. We have
\begin{equs}    \norm{\mcS_\eps\rho}_{L^{p,q}}&\lesssim\eps^{-n(p-1)/p-2(q-1)/q}\norm{\rho}_{L^{p,q}}\;\text{$\mathrm{for}$}\;\text{$\mathrm{all}$}\;p,q\in[1,\infty]\label{eq:scaleL1}
\end{equs}
uniformly in $\eps\in(0,1]$. Moreover, there exists $c^\star>0$ such that for all $c\in(0,c^\star]$ and $N\geqslant1$, it holds
    \begin{equs}\label{eq:KmuRhoeps}        \norm{K_{N,\mu}\partial_\eps\mcS_\eps\rho}_{L^1}\lesssim\eps^{-1+c}\mu^{-c}
    \end{equs}
uniformly in $\eps,\mu\in(0,1]$.
\end{lemma}
\begin{proof}
   \eqref{eq:KmuRhoeps} is \cite[Lemma~4.19 (B)]{Duch21}. 
\end{proof}
Let us now recall the bounds on the scale derivative of the effective Green's function.
\begin{lemma}\label{lem:dotG}
Fix $N,M\in\N$, $\mfm\in\N^n$ and $\mfl\in\N^{n+1}$. We have
\begin{equs}        \label{eq:heat1}
    \norm{\mcP^N_\mu\big(\mcP^N_\mu\big)^\dagger\bfX^\mfl \partial_{\tt x}^\mfm \dot G_\mu}_{\mcL^{(p,q),\infty}}&\lesssim \mu^{1-n/p-2/q+|\mfl|-|\mfm|}\;\mathrm{for}\;\mathrm{all}\;p,q\in[1,\infty]
    \end{equs}
uniformly in $\mu\in(0,1]$.
\end{lemma}
\begin{proof} 
\eqref{eq:heat1} is \cite[Lemma~4.24]{Duch21}. \eqref{eq:heat1} can be proven similarly, we recall the main steps. First, observe that we have $\dot G_\mu(x)=\mu^{-1}\tilde\chi(x_0/\mu^2) G(x)$ where $\tilde\chi=-2\t\chi'$.
 Moreover, by the properties of the heat kernel on the torus \cite{GrNote}, there exists a smooth function $A:\T^n\rightarrow\R$ with all derivatives bounded verifying for all $i\geqslant0$ the estimate
\begin{equs}
   | \partial^\mfm_{\ttx}A({\tt u})|\lesssim(1+{\tt u})^{-i}\,,
\end{equs}
and such that $G(x)=\sqrt{x_0}^{-d}A(\ttx/\sqrt{x_0})$. Therefore, $\dot G_\mu(x)=\mu\mcS_\mu\big(\tilde\chi G\big)$ and, with the same reasoning, $\bfX^\mfl\partial_{\ttx}^\mfm\dot G_\mu(x)=\mu^{1+|\mfl|-|\mfm|}\mcS_\mu\big(\tilde\chi\bfX^\mfl\partial_{\ttx}^\mfm G\big)$. The properties of $A$ then imply that $\tilde\chi\bfX^\mfl\partial_{\ttx}^\mfm G$ is bounded (and thus in $L^{p,q}$), so that the thesis follows using \eqref{eq:scaleL1}.
\end{proof}

\end{appendix}
\bibliographystyle{Martin}
\bibliography{refs.bib}

\end{document}